\documentclass[VANCOUVER,STIX1COL]{WileyNJD-v2}

\articletype{Article Type}%

\received{xxx}
\revised{xxx}
\accepted{xxx}

\usepackage{tikz}
\usetikzlibrary{shapes,patterns,snakes}

\DeclareMathOperator{\e}{e}
\DeclareMathOperator{\Span}{span}
\DeclareMathOperator{\Diag}{diag}

\newcommand{\revision}[1]{{\color{black} #1}}

\begin{document}

\title{Convergence analysis for parallel-in-time solution of
  hyperbolic systems}

\author[1]{Hans De Sterck}

\author[2]{Stephanie Friedhoff*}

\author[3]{Alexander J.M. Howse}

\author[4]{Scott P. MacLachlan}

\authormark{De Sterck, Friedhoff, Howse, MacLachlan}

\address[1]{\orgdiv{Department of Applied Mathematics}, \orgname{University of Waterloo}, \orgaddress{\state{Waterloo, ON}, \country{Canada}}}

\address[2]{\orgdiv{Department of Mathematics}, \orgname{Bergische Universit\"at Wuppertal}, \orgaddress{\state{Wuppertal}, \country{Germany}}}

\address[3]{\orgdiv{Core Data Science}, \orgname{Verafin}, \orgaddress{\state{St. John's, NL}, \country{Canada}}}

\address[4]{\orgdiv{Department of Mathematics and Statistics}, \orgname{Memorial University of Newfoundland}, \orgaddress{\state{St. John's, NL}, \country{Canada}}}

\corres{*Stephanie Friedhoff, Department of Mathematics, Bergische Universit\"at Wuppertal, Gau{\ss}str. 20, 42119 Wuppertal, Germany. \email{friedhoff@math.uni-wuppertal.de}}

\abstract[Summary]{Parallel-in-time algorithms have been successfully
  employed for reducing time-to-solution of a variety of partial
  differential equations, especially for diffusive (parabolic-type)
  equations. A major failing of parallel-in-time approaches to date,
  however, is that most methods show instabilities or poor convergence
  for hyperbolic problems. This paper focuses on the analysis of the
  convergence behavior of multigrid methods for the parallel-in-time
  solution of hyperbolic problems. Three analysis tools are considered
  that differ, in particular, in the treatment of the time dimension:
  (1) space-time local Fourier analysis, using a Fourier ansatz in
  space and time, (2) semi-algebraic mode analysis, coupling standard
  local Fourier analysis approaches in space with algebraic
  computation in time, and (3) a two-level reduction analysis,
  considering error propagation only on the coarse time grid. In this
  paper, we show how insights from reduction analysis can be used to
  improve feasibility of the semi-algebraic mode analysis, resulting
  in a tool that offers the best features of both analysis
  techniques.  Following validating numerical results, we investigate
  what insights the combined analysis framework can offer for two
  model hyperbolic problems, the linear advection equation in one
  space dimension and linear elasticity in two space dimensions.}


\maketitle

%
%
\section{Introduction}\label{sec:intro}

While spatial parallelism is a well-known tool in scientific
computing, hardware trends and scaling limits have served to renew
interest in algorithms that also allow ``space-time'' parallelism.
These techniques consider the solution of time-dependent systems of
partial differential equations (PDEs), and aim to compute the solution
in an ``all-at-once'' manner, breaking the sequential nature of
traditional time-stepping approaches.  While this is not a new
idea\cite{JNievergelt_1964, Gander2015_Review}, recent years have seen
significant effort devoted to space-time and time-parallel approaches
for the solution of time-dependent PDEs\cite{RDFalgout_etal_2014, MJGander_MNeumueller_2014, OngEtAl2016,
  BoltenEtAl2017, ANielsen_etal_2018a, AHessenthaler_etal_2018,
  Franco2018, HDeSterck_etal_2018a}.  With such intense interest in
the development of new schemes, there is a pressing need for
complementary analysis tools, to provide understanding of the relative
performance of related schemes and to inform the optimization of
algorithmic parameters as schemes are adapted to new problems.  The
central aim of this paper is to compare and contrast three such
analysis schemes for parallel-in-time algorithms of the
Parareal\cite{MGander_SVandewalle_2007, JLLions_etal_2001a} and
multigrid-reduction-in-time\cite{RDFalgout_etal_2014} (MGRIT)
methodologies.

From the multigrid perspective, mode analysis is an attractive tool
for analyzing convergence of these methods, using either eigenvectors
or Fourier modes (and related techniques) to predict convergence
rates.  In the context of space-time discretizations, three approaches
to mode analysis have been discussed in the recent literature:
space-time local Fourier
analysis\cite{SFriedhoff_etal_2012a,SVandewalle_GHorton_1995,
  MJGander_MNeumueller_2014} (LFA), semi-algebraic mode
analysis\cite{SFriedhoff_SMacLachlan_2015a, Franco2018, FGaspar_CRodrigo_2017a} (SAMA), and reduction-based
theory\cite{VADobrev_etal_2017,
  AHessenthaler_etal_2018,Hessenthaler_etal_2018a, BSouthworth_2018a}.  A fourth approach, Fourier-Laplace mode
analysis\cite{STaasan_HZhang_1995a, STaasan_HZhang_1995b,
  JJanssen_SVandewalle_1996a}, can be seen as an intermediate between
space-time LFA and the other two approaches, and will be discussed
only in that context.  While not considered here, the recent analysis
of Gander et al.\cite{Gander_etal_2018} extends the earlier result
of Gander and Hairer \cite{Gander_Hairer_2008a} from the Parareal
context to MGRIT.  These approaches bound the convergence of the
algorithm based on Lipschitz continuity and other properties of the
time propagators at the coarse and fine levels; while this offers
insight into the general convergence properties of the algorithms, it
is difficult to compare fairly to the mode analysis tools that are the
focus here.

In broad terms, the advantages and disadvantages of the three mode
analysis tools are as follows.  LFA is a well-known and widely used
tool for analysis of spatial multigrid
methods\cite{RWienands_WJoppich_2005, KStuben_UTrottenberg_1982a}, for
which it allows quantitative predictions of two- and multi-level
convergence.  However, numerical
experience\cite{SFriedhoff_etal_2012a, SVandewalle_GHorton_1995} shows
that it is not predictive for space-time methods in many situations, particularly for meshes
with few points in time or discretizing ``short'' temporal domains.
The extension of LFA to semi-algebraic mode
analysis\cite{SFriedhoff_SMacLachlan_2015a} overcomes this limitation
on space-time LFA; however, as implemented by Friedhoff and
MacLachlan, the quantitative estimates of SAMA required the
computation of Euclidean operator norms of matrices with size equal to
the number of time steps multiplied by the dimension of the spatial
LFA symbol.  When done exactly, this can be prohibitively expensive
for complicated problems over long time intervals.  Two-level
reduction analysis\cite{VADobrev_etal_2017, AHessenthaler_etal_2018}
overcomes this computational expense, under certain conditions, but is
based on solving eigenvalue problems of size equal to the number of
degrees of freedom in the spatial discretization.  In the results
below, we explore each of these methods for two specific hyperbolic
model problems, linear advection in one dimension and incompressible
linear elasticity in two dimensions, and compare and contrast their
predictions in these settings.

Several important observations come from this comparison.  First, we
see that the bounds used to make the reduction analysis
computationally tractable can be directly applied to SAMA.  This
results in a tool that combines the best of both approaches, using the
flexibility of spatial LFA in settings where the spatial eigenproblem
is intractable, but the ease of computing a bound for large
numbers of time steps that comes from reduction analysis.  Secondly,
we expose the complications of assuming unitary diagonalizability, as
is done in the reduction analysis, and detail how this can be avoided
in the SAMA bound.  Overall, this allows us to get more insight into
MGRIT convergence, particularly as we vary both spatial and temporal
problem sizes.  \revision{Such insight, in turn, can enable the development of improved MGRIT algorithms for difficult problems, as has recently been done for high-order explicit discretizations of the linear advection equation \cite{HDeSterck_etal_2019}.}     A final benefit is the ability to critically compare
space-time LFA with the improved SAMA analysis, and gain better
insight into problems and parameter regimes for which space-time LFA
may be a feasible and sufficiently accurate option.

The remainder of this paper is organized as follows.  Details of the
model problems, linear advection in one spatial dimension and
incompressible linear elasticity in two spatial dimensions, are given
in Section \ref{sec:problems}.  The Parareal and MGRIT algorithms are reviewed in
Section \ref{sec:pint_methods}.  Section \ref{sec:analysis_tools}
describes the three mode analysis tools in detail, with discussion on
how to combine aspects of SAMA and reduction analysis in Section
\ref{ssec:comparison}.  Numerical results comparing the analysis methods as
they exist in the literature are given in Section
\ref{sub:comparison}.  Sections \ref{sub:investigating_adv} and
\ref{sub:investigating_elasticity} present numerical results giving
new insight into convergence of the Parareal and MGRIT algorithms for linear advection
and elasticity, respectively, based on SAMA improved by insights from
reduction analysis. Concluding remarks are given in Section \ref{sec:conclusions}.

%
%
\section{Model problems}\label{sec:problems}

\revision{In this section, we discuss the details of the two model problems for which the analysis tools will be compared in Section \ref{sec:num_results}.  For linear advection, we consider a simple upwinding-in-space and implicit-Euler-in-time discretization, corresponding to that used by De Sterck et al. \cite{HDeSterck_etal_2018a}  For linear elasticity, we use a mixed-finite-element-in-space and implicit-Euler-in-time discretization corresponding to that used by Hessenthaler et al. \cite{AHessenthaler_etal_2018}  We note that ``better'' (primarily less diffusive) discretizations of both equations certainly exist, but (to our knowledge) successful application of Parareal and MGRIT to those discretizations either requires a modification of the algorithm (e.\,g., the PITA framework\cite{CFarhat_etal_2006}) or is not well documented in the published literature, aside from some forthcoming results\cite{HDeSterck_etal_2019}.  The analysis presented herein can be applied to such discretizations in the same manner as used here; we leave such details to future work, to reflect the (hopefully) continued success of developing effective parallel-in-time algorithms for a wider class of discretizations and problems.
}

\subsection{Linear advection}
We consider the advection of a scalar quantity, $u(x,t)$, subject to a
known non-zero constant flow speed, $c$, in the domain $\Omega \times
[0,T]$, where $\Omega$ is a finite interval, $[a,b]$, and $T$ denotes
the final time. For example, this models advection of nonreactive
particles by an incompressible fluid, where particle density,
$u(x,t)$, depends only on advection of particles by the fluid. The governing equation is given by
\begin{equation}
	u_{t} + cu_x = 0 \text{ in } \Omega \times [0,T],\label{eq:advection}
\end{equation}
where we assume $c > 0$ for the subsequent discussion. We prescribe the initial condition $u(x,0) = u_0(x)$ and impose the periodic boundary condition $u(a,t) = u(b,t)$ in all that follows. We discretize \eqref{eq:advection} on a uniform space-time grid, with $N_x$ spatial intervals of width ${\Delta x} = (b-a)/N_x$ and $N_t$ temporal intervals of time step $\Delta t = T/N_t$, using a first-order implicit upwind scheme:
\begin{equation}
\left(1 + \frac{c\Delta t}{{\Delta x}}\right) u_{j,i} - \frac{c\Delta t}{{\Delta x}}u_{j-1,i} = u_{j,i-1}, \label{eq:advection:discrete} \quad i=1,\ldots,N_t, ~j=0,1,\ldots,N_x.
\end{equation}


\subsection{Linear elasticity} 
We consider the dynamic and linear elastic response of an
incompressible solid structure in the domain $\Omega\times [0,T]$,
where $\Omega$ is an open domain in $\mathbb{R}^2$, and $T$ denotes
the final time. Denoting the current and reference position of a
material point by $x$ and $X$, respectively, and the respective
Eulerian and Lagrangian gradient operators by $\nabla$ and $\nabla_X$,
we define the deformation gradient, $F$, by $F = \nabla_X x = I +
\nabla_X u$, where $x(X,t) = X + u(X,t)$ defines the displacement,
$u$, of the material point, $x$, in the current configuration at time $t$ with respect to its position in the reference configuration, $X$, and where $I$ denotes the identity matrix of corresponding size. Then, the governing equations are given by
\begin{align}
	\rho u_{tt} - \nabla_X\cdot\sigma &= 0 && \text{in } \Omega \times [0,T],\label{eq:elasticity:2nd:order}\\
	\det F &= 1 && \text{in } \Omega\times [0,T],
\end{align}
where $\rho$ denotes material density, and where $\sigma = \sigma(u,p)
= \mu(F-I) - pI$ is the Cauchy stress tensor for an incompressible linear elastic material with stiffness parameter, $\mu$, and hydrostatic pressure, $p$. We prescribe $u$ and $u_t$ at $t = 0$, $u(X,0) = 0$ and $u_t(X,0) = \widehat{v}_0$, and impose homogeneous Dirichlet boundary conditions, $u(X,t) = 0$ for $X\in\Gamma^D$ and $t\in[0,T]$, where $\Gamma^D$ denotes the Dirichlet boundary of the domain $\Omega$. Using $\sigma = \mu\nabla_X u - pI$, neglecting higher-order effects of the deforming domain, and transforming Equation \eqref{eq:elasticity:2nd:order} to a system of first-order equations, we obtain
\begin{align}
	\rho v_t &= \mu\nabla^2 u - \nabla p && \text{in } \Omega\times [0,T],\label{eq:vt}\\
	u_t &= v && \text{in } \Omega\times [0,T],\label{eq:ut:v}\\
	\nabla\cdot v &= 0 && \text{in } \Omega\times [0,T]\label{eq:incompr}
\end{align}
for displacement, $u$, velocity, $v$, and hydrostatic pressure, $p$.

We discretize Equations \eqref{eq:vt}-\eqref{eq:incompr} on an
equidistant time grid consisting of $N_t\in\mathbb{N}$ time intervals
using a time-step size $\Delta t = T/N_t$. Motivated by existing results\cite{AHessenthaler_etal_2018}, we use implicit Euler for the time discretization. Denoting displacement, velocity, and pressure at time $t_i = i\Delta t$, $i = 0, \ldots, N_t$, by $u_i, v_i$, and $p_i$, respectively, for $i = 1, \ldots, N_t$, Equations \eqref{eq:vt} and \eqref{eq:ut:v} are discretized as
\begin{equation}\label{eq:vt:time:discr}
	\rho v_i - \Delta t\mu\nabla^2 u_i + \Delta t\nabla p_i = \rho v_{i-1},
\end{equation}
and
\begin{equation}\label{eq:ut:v:time:discr}
	u_i - \Delta t v_i= u_{i-1}.
\end{equation}

The weak form of Equations \eqref{eq:vt:time:discr},
\eqref{eq:ut:v:time:discr}, and \eqref{eq:incompr}
is found by multiplying by test functions, $\chi\in (H^1(\Omega))^2$,
$\phi\in (L_2(\Omega))^2$, and $\psi\in L^2(\Omega)$, respectively, giving
\begin{align}
	\langle \rho v_i -\Delta t\mu\nabla^2u_i+\Delta t\nabla p_i, \chi \rangle &= \langle \rho v_{i-1}, \chi\rangle && \forall~\chi\nonumber\\
	\Leftrightarrow\quad \rho \langle v_i, \chi\rangle + \Delta t^2\mu\langle \nabla v_i, \nabla \chi\rangle - \Delta t\langle p_i, \nabla\cdot\chi\rangle &= \rho\langle v_{i-1}, \chi\rangle - \Delta t\mu\langle \nabla u_{i-1}, \nabla\chi\rangle && \forall~\chi\label{eq:vt:weak},\\
	\langle u_i, \phi\rangle - \Delta t\langle v_i,\phi\rangle &= \langle u_{i-1},\phi\rangle && \forall~\phi\label{eq:ut:v:weak}\\
	-\langle \nabla\cdot v_i, \psi\rangle & = 0 && \forall~\psi\label{eq:incompr:weak}.
\end{align}

We discretize the spatial domain, $\Omega$, on a uniform quadrilateral
grid with mesh size ${\Delta x}$ using Taylor-Hood $(Q2 - Q1)$ elements
\cite{DBraess_2001, SCBrenner_LRScott_1994a} for velocity, $v$, and
pressure, $p$, and $Q2$ elements for displacement, $u$, ensuring
uniqueness of the solution. Denoting the mass and stiffness matrices
of the discretized vector Laplacian by $M$ and $K$, respectively, and
the (negative) discrete divergence operator by $B^T$, Equations \eqref{eq:vt:weak}-\eqref{eq:incompr:weak} are discretized to
\begin{align}
	\rho Mv_i + \Delta t^2 \mu Kv_i + \Delta t B p_i &= \rho Mv_{i-1} - \Delta t\mu Ku_{i-1}\label{eq:vt:discrete},\\
	Mu_i - \Delta t Mv_i&= Mu_{i-1},\label{eq:ut:v:discrete}\\
	\Delta t B^Tv_i &= 0,\label{eq:incompr:discrete}
\end{align}
which are equivalent to the following linear system of equations
\begin{equation}\label{eq:lin:system}
\begin{bmatrix}
	\rho M + \Delta t^2\mu K & 0 & \Delta t B\\
	-\Delta t M & M & 0 \\
	\Delta t B^T & 0 & 0
\end{bmatrix}\begin{bmatrix}
	v_i\\
	u_i\\
	p_i
\end{bmatrix}=
\begin{bmatrix}
	\rho Mv_{i-1} - \Delta t\mu Ku_{i-1}\\
	Mu_{i-1}\\
	0 
\end{bmatrix}. 
\end{equation}

It is tempting to simply remove the rows and columns corresponding to
the pressure variable from this system, since $p_{i-1}$ does not
appear in the equations at time step $t_i$.  Indeed, this approach was
taken in the analysis by Hessenthaler et al.\cite{AHessenthaler_etal_2018}.  
However, to do so ignores the important role that $p_i$ plays as a Lagrange
multiplier, enforcing the incompressiblity constraint for the solution
at time step $t_i$ (but was properly accounted for in simulations).  
Instead, we eliminate the contribution from this
block by considering the block factorization of the matrix in
\eqref{eq:lin:system} and the resulting Schur complement.  
Factoring this matrix gives
\[
\begin{bmatrix}
	\rho M + \Delta t^2\mu K & 0 & \Delta t B\\
	-\Delta t M & M & 0 \\
	\Delta t B^T & 0 & 0
\end{bmatrix}
=
\begin{bmatrix}
	\rho M + \Delta t^2\mu K & 0 & 0\\
	-\Delta t M & M & 0 \\
	\Delta t B^T & 0 & -\Delta t^2 S
\end{bmatrix}
\begin{bmatrix}
	I & 0 & \Delta t (\rho M + \Delta t^2\mu K)^{-1}B\\
	0 & I & \Delta t^2 (\rho M + \Delta t^2\mu K)^{-1}B \\
	0 & 0 & I
\end{bmatrix},
\]
for $S = B^T(\rho M + \Delta t^2\mu K)^{-1}B$.  Thus, $p_i$ satisfies
\[
-\Delta t^2 Sp_i = -\Delta t B^T(\rho M + \Delta
t^2 \mu K)^{-1}\left(\rho Mv_{i-1} - \Delta t\mu Ku_{i-1}\right),
\]
and can be directly eliminated from the equation for $v_i$ by
subtracting $\Delta t B p_i$ from the right-hand side of
\eqref{eq:vt:discrete}.  Noting that both velocity and displacement 
are two-dimensional vector fields of $Q2$ degrees of freedom, this yields 
the reduced system with four scalar functions,
\begin{align*}
\begin{bmatrix}
	\rho M + \Delta t^2\mu K & 0 \\
	-\Delta t M & M \\
\end{bmatrix}\begin{bmatrix}
	v_i\\
	u_i
\end{bmatrix}& =
\begin{bmatrix}
	\left(I - BS^{-1}B^T(\rho M + \Delta t^2 \mu K)^{-1}\right)\left(\rho Mv_{i-1} - \Delta t\mu Ku_{i-1}\right)\\
	Mu_{i-1}
\end{bmatrix}\\
& =
\begin{bmatrix}
	\rho\left(I - BS^{-1}B^T(\rho M + \Delta t^2 \mu
        K)^{-1}\right)M &  - \Delta t\mu \left(I - BS^{-1}B^T(\rho M + \Delta t^2 \mu
        K)^{-1}\right)K\\
	0 & M
\end{bmatrix}
\begin{bmatrix}
	v_{i-1}\\
	u_{i-1}
\end{bmatrix}.
\end{align*}
It is this form of the propagator that we analyse below.  We note that
although the only nonzero operator acting on the pressure acts on
$p_i$, its effect is to change the dependence of $v_i$ and $u_i$ on
$v_{i-1}$ and $u_{i-1}$.  This can easily be seen to take
the form of an orthogonal projection operator acting on the data from
the previous time-step; it is also easy to check that this formulation
guarantees that $B^Tv_{i} = 0$, as required by the incompressibility
constraint.

As is common in MGRIT, we primarily analyse this propagator in
``$\Phi$-form'', i.\,e., in the form $[v_i ~u_i]^T = \Phi[v_{i-1}~u_{i-1}]^T$, 
writing the projection operator, $P =I-(\rho M + \Delta t^2 \mu K)^{-1}BS^{-1}B^T$, 
to give 
\begin{align}
\begin{bmatrix}
	v_i\\
	u_i
\end{bmatrix}& =
\begin{bmatrix}
	\rho M + \Delta t^2\mu K & 0 \\
	-\Delta t M & M \\
\end{bmatrix}^{-1}\begin{bmatrix}
	\rho\left(I - BS^{-1}B^T(\rho M + \Delta t^2 \mu
        K)^{-1}\right)M &  - \Delta t\mu \left(I - BS^{-1}B^T(\rho M + \Delta t^2 \mu
        K)^{-1}\right)K\\
	0 & M
\end{bmatrix}
\begin{bmatrix}
	v_{i-1}\\
	u_{i-1}
\end{bmatrix} \nonumber\\
& =
\begin{bmatrix}
\rho P \left(\rho M + \Delta t^2\mu K\right)^{-1}M & -\Delta t\mu
P\left(\rho M + \Delta t^2\mu K\right)^{-1}K \\
\rho\Delta tP \left(\rho M + \Delta t^2\mu K\right)^{-1}M & -(\Delta t)^2\mu
P\left(\rho M + \Delta t^2\mu K\right)^{-1}K+I
\end{bmatrix}
\begin{bmatrix}
	v_{i-1}\\
	u_{i-1}
\end{bmatrix}.\label{eq:Phi_form}
\end{align}

%
%
\section{Parallel-in-time methods}\label{sec:pint_methods}

\subsection{Parareal}\label{sub:Parareal}
The Parareal algorithm \cite{JLLions_etal_2001a} is a parallel method
for solving systems of ordinary differential equations of the form
\begin{equation}\label{eq:ode:system}
	u'(t) = f(t,u(t)), \quad u(0) = g_0, \quad t\in[0,T],
\end{equation}
arising, for example, when solving a system of PDEs using a
method-of-lines approximation to discretize the spatial
domain. Parareal can be interpreted in a variety of ways, including as
multiple shooting, domain decomposition, and multigrid methods
\cite{MGander_SVandewalle_2007, RDFalgout_etal_2014}. Here, we
describe Parareal as a two-level time-multigrid scheme. For ease of
presentation, we only describe the method in a linear setting, i.\,e.,
in the case that $f$ is a linear function of $u(t)$; the full
approximation storage (FAS) approach \cite{ABrandt_1977b} can be
applied in the same manner to accomodate nonlinear problems.

Parareal combines time stepping on the discretized temporal domain,
the fine grid, with time stepping on a coarser temporal grid that uses
a larger time step. More precisely, consider a fine temporal grid
with points $t_i = i\Delta t$, $i=0,1,\ldots,N_t$, with constant time
step $\Delta t=T/N_t>0$, and let $u_i$ be an approximation to $u(t_i)$
for $i=1,\ldots,N_t$, with $u_0 = u(0)$. Further, consider a one-step
time integration method,
\[
u_i = \Phi u_{i-1} + g_i, \quad i = 1,\ldots, N_t,
\]
with time-stepping operator, $\Phi$, that takes a solution at time
$t_{i-1}$ to that at time $t_i$, along with a time-dependent forcing
term, $g_i$.  (Note that both the assumption of constant time step and
of a time-independent single-step time-stepping operator are for
notational convenience only, and can easily be relaxed.)  The discrete
approximation to the solution of \eqref{eq:ode:system} can be
represented as a forward solve of the block-structured linear system
\begin{equation}\label{eq:fine:linear:system}
	Au\equiv\begin{bmatrix}
		I\\
		-\Phi & I\\
		& \ddots & \ddots\\
		& & -\Phi & I
	\end{bmatrix} \begin{bmatrix}
		u_0\\
		u_1\\
		\vdots\\
		u_{N_t}
	\end{bmatrix} = \begin{bmatrix}
		g_0\\
		g_1\\
		\vdots\\
		g_{N_t}
	\end{bmatrix}\equiv g.
\end{equation}
Note that, in the time dimension, this forward solve is completely sequential.

Parareal enables parallelism in the solution process by introducing a coarse temporal
grid, or (using multigrid terminology) a set of C-points, derived from the original
(fine) temporal grid by considering only every $m$-th temporal point, where $m>1$
is an integer called the coarsening factor. Thus, the coarse temporal grid consists
of $N_T=N_t/m$ intervals, with points $T_j=j\Delta T$, $j=0,1,\ldots,N_T$, with
spacing $\Delta T = m\Delta t$; the remaining temporal points define the set of
F-points, as visualized in Figure \ref{fig:fine:coarse:t:grid}.

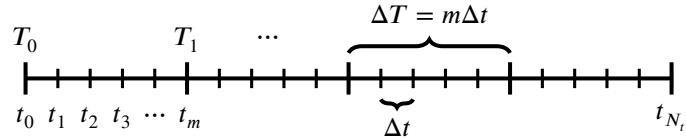
\begin{figure}[h!t]
	\centering
	\begin{tikzpicture}[scale=.85]
		\draw[line width=1.5pt] (0,0) -- (10,0);
		\foreach \i in {0,.5,1,...,10}{
			\draw[line width=1.15pt] (\i,.15) -- (\i,-.15);
		}
		\foreach \i in {0,2.5,5,7.5,10}{
			\draw[line width=1.5pt] (\i,.25) -- (\i,-.25);
		}
		\foreach \i/\n in {0/0,.5/1,1/2,1.5/3}{
			\draw (\i,-.6) node {$t_\n$};
		}
		\draw (2,-.6) node {$\cdots$};
		\draw (2.55,-.6) node {$t_m$};
		\draw (10,-.6) node {$t_{N_t}$};
		\draw (0,.6) node {$T_0$};
		\draw (2.5,.6) node {$T_1$};
		\draw (3.75,.6) node {$\cdots$};
			
		\draw[line width=1.5pt,,decorate,decoration={brace,amplitude=3pt},yshift=5pt] (6,-.5) -- (5.5,-.5) node[below=3pt,midway] {$\Delta t$};
		\draw[line width=1.5pt,,decorate,decoration={brace,amplitude=5pt},yshift=-3pt] (5,.5) -- (7.5,.5) node[above=7pt,midway] {$\Delta T = m\Delta t$};
	\end{tikzpicture}
	\caption{Fine- and coarse time discretization meshes.}
	\label{fig:fine:coarse:t:grid}
\end{figure}

The coarse-grid problem, $A_c u_\Delta=g_\Delta$, is defined by
considering a one-step method with time-stepping operator, $\Phi_c$,
using time step $\Delta T$, given by
\[
	u_{\Delta,j} = \Phi_c u_{\Delta,j-1} + g_{\Delta,j}, \quad j = 1, \ldots, N_T,
\]
where $u_{\Delta,j}$ denotes an approximation to $u(T_j)$ for $j = 1,\ldots, N_T$, and $u_{\Delta,0} = u(0)$. 

Rather than sequentially solving for each $u_i$, Parareal alternates
between a ``relaxation scheme'' on the F-points and a sequential solve
on the C-points.  The first of these two processes, known as
F-relaxation, updates the unknowns at F-points by propagating the
values from each C-point, $T_j$, across a coarse-scale time interval,
$(T_j, T_{j+1})$, $j=0,1,\ldots,N_T-1$ using the fine-grid
time-stepper, $\Phi$. Note that within each coarse-scale time
interval, these updates are sequential, but that there is no
dependency across coarse time intervals, enabling parallelism in the
relaxation process.  After F-relaxation, the residual is evaluated
only at the C-points and ``injected'' to the coarse temporal grid.
After a sequential solve of the coarse-grid equations (for which
typical choices in Parareal use high-order integration schemes over
longer time intervals), a correction is interpolated to the fine grid,
using ``ideal interpolation'' (which is ``ideal'' as it corresponds to
using the Schur complement on the coarse grid).  This interpolation
operator is defined by taking the corrected approximate solution at at
each C-point, $T_j$, and time-stepping across the coarse-scale time
interval, $(T_j, T_{j+1})$, $j=0,1,\ldots,N_T-1$, again using the
fine-grid time-stepper, $\Phi$.  The injection and ideal interpolation
operators are block operators of size $(N_T+1)\times(N_t+1)$ or
$(N_t+1)\times(N_T+1)$, respectively, given by
\begin{center}
	\begin{tikzpicture}
		\node {$\setcounter{MaxMatrixCols}{20}
		R_I = \begin{bmatrix}
		I & 0 & \cdots & 0\\
		& & & & I & 0 & \cdots & 0\\
		& & & & & & & & \ddots\\
		& & & & & & & & & I & 0 & \cdots & 0 &\\
		& & & & & & & & & & & & & I
	\end{bmatrix}\quad\text{and}\quad P_\Phi = \begin{bmatrix}
		Z^{(\Phi)}\\
		& Z^{(\Phi)}\\
		&& \ddots\\
		&&& Z^{(\Phi)}\\
		&&&& I
	\end{bmatrix} ~\text{with } Z^{(\Phi)} = \begin{bmatrix}
		I\\
		\Phi\\
		\vdots\\
		\Phi^{m-1}
	\end{bmatrix}.
		$};
		\draw[line width=1.15pt,decorate,decoration={brace,amplitude=3pt},yshift=5pt] (-5.1,.5) -- (-6.5,.5) node[below=3pt,midway] {\scriptsize$m$ blocks};
	\end{tikzpicture}
\end{center}
As the error propagator of F-relaxation can be written as $P_\Phi R_I$, the Parareal algorithm may be represented by the two-level iteration matrix,
\begin{equation}\label{eq:Parareal:it:matrix}
	E^F = (I-P_\Phi A_c^{-1}R_I A)P_\Phi R_I = P_\Phi(I-A_c^{-1}A_S)R_I,
\end{equation}
where equivalence holds since $R_IAP_\Phi$ defines the Schur complement coarse-grid operator,
\[
	A_S = \begin{bmatrix}
		I\\
		-\Phi^m & I\\
		& \ddots & \ddots\\
		& & -\Phi^m & I
	\end{bmatrix}.
\]


\subsection{Multigrid-reduction-in-time}
The major sequential bottleneck for Parareal is the sequential solve
of the coarse-grid system, $A_cu_\Delta = g_\Delta$.  Directly
applying multigrid principles and using Parareal recursively to solve
this system is observed to often have significantly degraded
convergence properties.  This motivates the
multigrid-reduction-in-time (MGRIT) algorithm
\cite{RDFalgout_etal_2014}, which is based on applying multigrid
reduction techniques
\cite{MRies_UTrottenberg_1979,MRies_UTrottenberg_GWinter_1983a} to the
time dimension. Like Parareal, MGRIT uses injection and ideal
interpolation for intergrid transfer operations, and it uses the same
coarse-grid operator, $A_c$, as can be used in the two-level Parareal
method. However, in order to obtain a scalable multilevel algorithm,
MGRIT augments F-relaxation, using relaxation also at the
C-points. This C-relaxation is defined analogously to F-relaxation, by
leaving the values of F-points unchanged, and updating the unknowns at
C-points using the time-stepper, $\Phi$, and values at neighboring
F-points. Relaxation in MGRIT consisting of combined sweeps of F- and
C-relaxation; typically, FCF-relaxation, F-relaxation, followed by
C-relaxation, and again F-relaxation, is employed.  This scheme can be
shown to be equivalent to Richardson relaxation on the coarse time
grid with the Schur complement coarse-grid operator, $A_S$, followed
by ideal interpolation\cite{RDFalgout_etal_2014}.

Writing the error propagator of FCF-relaxation as $P_\Phi(I-A_S)R_I$, the two-level MGRIT algorithm may be represented by the two-level iteration matrix,
\begin{equation}\label{eq:two:level:mgrit}
	E^{FCF} = (I-P_\Phi A_c^{-1}R_I A)P_\Phi(I-A_S)R_I = P_\Phi(I-A_c^{-1}A_S)(I-A_S)R_I.
\end{equation}
Instead of inverting $A_c$ directly as in the two-level MGRIT
algorithm \eqref{eq:two:level:mgrit}, in the three-level method, the
system on the first coarse grid is approximated by a two-grid cycle
with zero initial guess, that is, the term $A_c^{-1}$ in the two-level
iteration matrix \eqref{eq:two:level:mgrit} is replaced by
$(I_c-E^{(2,3)-FCF})A_c^{-1}$ (where $E^{(2,3)-FCF}$ is defined
analogously to \eqref{eq:two:level:mgrit}) to obtain the three-level
V-cycle iteration matrix,
\begin{equation}
	E^{V(1,3)-FCF} = \left(I - P_\Phi \left(I_c - E^{(2,3)-FCF}\right)A_c^{-1}R_IA\right)P_\Phi(I-A_S)R_I,
\end{equation}
or, it is replaced by $(I_c-E^{(2,3)-FCF}E^{(2,3)-FCF})A_c^{-1}$ to
obtain the three-level F-cycle (or, equivalently, W-cycle) iteration
matrix
\begin{equation}
	E^{F(1,3)-FCF} = \left(I - P_\Phi \left(I_c - E^{(2,3)-FCF}E^{(2,3)-FCF}\right)A_c^{-1}R_IA\right)P_\Phi(I-A_S)R_I,
\end{equation}
respectively.  Iteration matrices of three-level V- and F-cycle
Parareal methods can be derived analogously.

While there are clear and important differences, both historical and
in practice, between Parareal and MGRIT, in this paper we view them as
being two methods within the same broader family.  As such, we will
consider in detail only the case where the coarse-grid operator,
$A_c$, is given by rediscretization of the fine-grid operator with the
coarse time-step.  Since Parareal methods often make use of different
coarse- and fine-propagators, we will instead refer to the methods
where this condition is imposed as MGRIT with F- and
FCF-relaxation, instead of as Parareal and MGRIT, respectively,
but use the more distinctive terminology for statements that are true
regardless of the choice of coarse propagator. An interesting property of 
Parareal and MGRIT is the following exactness property: Assuming
exact arithmetic, one iteration of F-relaxation computes the exact
solution at the first $m-1$ time steps, i.\,e., at all F-points in the
first coarse-scale time interval. Furthermore, one iteration of
FCF-relaxation computes the exact solution at the first $2m-1$ time
steps. Therefore, MGRIT with F-relaxation solves for the exact
solution in $N_t/m$ iterations, while MGRIT with FCF-relaxation solves 
for the exact solution in $N_t/(2m)$ iterations, corresponding to half the 
number of points on the coarse grid.

%
%
\section{Mode analysis tools for parallel-in-time methods}\label{sec:analysis_tools}
While both Parareal and MGRIT can be analysed based on abstracted
properties of the coarse and fine time
integrators\cite{Gander_Hairer_2008a, Gander_etal_2018}, these tools
do not provide deep insight into the success or failure of algorithmic
choices within the methodology.  Instead, here we present three tools
based on mode analysis principles applied to the iteration matrices,
$E^F$ and $E^{FCF}$, given in Equations \eqref{eq:Parareal:it:matrix}
and \eqref{eq:two:level:mgrit}, respectively, and to the iteration matrices
 of the respective three-level methods.

\subsection{Space-time local Fourier analysis}
Local Fourier analysis is a classical tool for analysing convergence
of multigrid methods\cite{ABrandt_1977b, KStuben_UTrottenberg_1982a},
that has been applied to space-time problems with mixed
results in terms of accuracy in predictions \cite{SFriedhoff_etal_2012a, MJGander_MNeumueller_2014}.  
Here we present a brief review of the calculation of spatial Fourier
symbols (as will be used in all of the mode analysis tools considered
in this section), along with their coupling with local Fourier
analysis in the time direction.

\subsubsection{Spatial Fourier symbols}
The local Fourier analysis of the time-stepping operator, $\Phi$,
makes use of its Fourier representation, also called the \emph{Fourier
  symbol}.  For the scalar advection equation, this follows from
standard analysis; however, LFA for systems of PDEs requires more
advanced tools\cite{UTrottenberg_etal_2001a,
  SPMacLachlan_CWOosterlee_2011a}. Note that we do not consider any
coarsening of the spatial domain here \revision{to limit the scope of the paper} 
and, thus, we only need the
spatial Fourier symbol of the time-stepping operator. 

Standard LFA for a problem in one spatial dimension considers a scalar
discrete problem posed on an infinite uniform grid with mesh size ${\Delta x}>0$,
\[
G_{\Delta x} = \left\{x_j = j{\Delta x} : j \in \mathbb{Z}\right\}.
\]
Extending a constant-coefficient discrete operator from a finite
spatial domain, such as that on the left of Equation
\eqref{eq:advection:discrete}, to the infinite grid $G_{\Delta x}$ leads to a
Toeplitz operator acting on the space of square-summable sequences
($\ell_2$) that is formally diagonalizable by standard Fourier modes,
\begin{alignat*}{4}
	\varphi(\theta,x_{j}) & ~=~ & \e^{\imath\theta\cdot x_{j}/{\Delta x}} \quad & \text{for } \theta\in(-\pi,\pi].
\end{alignat*}
Noting that the operator $\Phi$ is the inverse of the lower bidiagonal
matrix on the left of \eqref{eq:advection:discrete}, we find that its Fourier symbol, denoted $\widetilde{\Phi}_\theta$, is the scalar
\begin{equation}
\widetilde{\Phi}_\theta = \left[1 + \dfrac{c \Delta t}{{\Delta x}}\left(1-e^{-\imath\theta}\right)\right]^{-1}.
\label{eq:advection:phi:symbol}
\end{equation}
Furthermore, the Fourier symbols for the time integrators $\Phi_c$ and
$\Phi_{cc}$ on the first and second coarse grids can be obtained by
replacing $\Delta t$ by $m\Delta t$ in \eqref{eq:advection:phi:symbol}
to account for factor-$m$ temporal coarsening.  We note that for a
spatial problem with periodic boundary conditions, the eigenvalues of
$\Phi$ are given precisely by sampling $\widetilde{\Phi}_\theta$ at evenly
spaced points in $(-\pi,\pi]$; for other boundary conditions,  a good
  estimate of the discrete eigenvalues of $\Phi$ is given by the
  range of its Fourier symbol\cite{CRodrigo_etal_2018a}.

For the elasticity operator in \eqref{eq:Phi_form}, we must generalize
this analysis to consider a two-dimensional infinite uniform grid,
along with adaptations to handle the block structure imposed by the
mixed finite-element discretization.  Accordingly, we consider the
two-dimensional infinite uniform grid, decomposed as
\begin{equation}\label{eq:lfa:space:grid}
	G_{\Delta x} := G_{{\Delta x},N} \cup G_{{\Delta x},XE} \cup G_{{\Delta x},YE} \cup G_{{\Delta x},C},
\end{equation}
with
\begin{alignat*}{6}
	G_{{\Delta x},N} & ~=~ & \{\mathbf{x}_{j,k}^N = (j{\Delta x}, k{\Delta x}) : (j,k)\in\mathbb{Z}^2\}, \quad & G_{{\Delta x},XE} & ~=~ & \{\mathbf{x}_{j,k}^{XE} = ((j+1/2){\Delta x}, k{\Delta x}) : (j,k)\in\mathbb{Z}^2\},\\
	G_{{\Delta x},YE} & ~=~ & \{\mathbf{x}_{j,k}^{YE} = (j{\Delta x}, (k+1/2){\Delta x}) : (j,k)\in\mathbb{Z}^2\}, \quad & G_{{\Delta x},C} & ~=~ & \{\mathbf{x}_{j,k}^C = ((j+1/2){\Delta x}, (k+1/2){\Delta x}) : (j,k)\in\mathbb{Z}^2\}.
\end{alignat*}
This decomposition arises by considering the variation in basis
functions for the standard $Q2$ approximation space, with nodal basis
functions defined for mesh points at corners of ${\Delta x}\times {\Delta x}$ quadrilateral elements, $\mathbf{x}^N\in G_{{\Delta x},N}$,
mesh points located on the $x$- and $y$-edges of elements,
$\mathbf{x}^{XE}\in G_{{\Delta x},XE}$ or $\mathbf{x}^{YE}\in G_{{\Delta x},YE}$, and
cell-centered mesh points, $\mathbf{x}^C\in G_{{\Delta x},C}$.

While the finite-element discretization of a scalar PDE on a
structured two-dimensional infinite uniform grid using nodal $(Q1)$ finite elements
leads to an operator that is block Toeplitz with Toeplitz blocks (BTTB) that has a scalar symbol, the same is not
true for higher-order elements.  However, by using the same basis on
each element and permuting the operator into block form ordered by
``type'' of basis function, we readily obtain a block operator with
blocks that are themselves BTTB operators and can be used to generate a matrix-based Fourier
symbol\cite{TBoonen_etal_2008a, SPMacLachlan_CWOosterlee_2011a}.
On $G_{\Delta x}$, this arises by considering a discrete operator, $L_{\Delta x}$,
defined in terms of its stencil for each ``type'' of degree of freedom. For example, for the 
nodal degrees of freedom, we have 
\[
	L_{\Delta x}^N \widehat{=} [s^N_{\boldsymbol{\kappa}}]_{\Delta x}, ~\boldsymbol{\kappa}=(\kappa_1,\kappa_2)\in V; \qquad L_{\Delta x}^Nw_{\Delta x}(\mathbf{x}) = \sum_{\boldsymbol{\kappa}\in V} s^N_{\boldsymbol{\kappa}}w_{\Delta x}(\mathbf{x}+\boldsymbol{\kappa}{\Delta x}), ~w_{\Delta x}(\mathbf{x})\in\ell^2(G_{\Delta x}),
\]
with constant coefficients, $s_{\boldsymbol{\kappa}}\in\mathbb{C}$,
and $\boldsymbol{\kappa} = (\kappa_1, \kappa_2)$ taken from a finite
index set, $V := V_N \cup V_C \cup V_{XE} \cup V_{YE}$, with
$V_N\subset\mathbb{Z}^2$, $V_{XE}\subset\{(z_1+1/2, z_2) : (z_1,
z_2)\in\mathbb{Z}^2\}$, $V_{YE}\subset\{(z_1, z_2+1/2) : (z_1,
z_2)\in\mathbb{Z}^2\}$, and $V_C\subset\{(z_1+1/2, z_2+1/2) : (z_1,
z_2)\in\mathbb{Z}^2\}$.  Similar definitions are used for the discrete
operator acting on degrees of freedom corresponding to mesh edges and
cell centers. Note that the decomposition of the index set, $V$,
naturally defines $L_{\Delta x}^N$ as a block operator, with each block
corresponding to one type of mesh point. For the elasticity operator in 
\eqref{eq:Phi_form}, we obtain a block operator of size $16\times 16$, 
since there are four scalar functions in the system, each discretized using
Q2 elements with four types of 
degrees of freedom.

With this, the Fourier representation (symbol), of an operator, $L_{\Delta x}$, denoted by $\widetilde{L}_{\Delta x}(\boldsymbol{\theta})$, with $\boldsymbol{\theta}\in(-\pi,\pi]^2$, is a block matrix that can be computed using a Fourier basis that accounts for the different types of mesh points of $G_{\Delta x}$. This Fourier basis is given by
\begin{equation}
	\Span\left\{ \varphi_N(\boldsymbol{\theta},\mathbf{x}), \varphi_{XE}(\boldsymbol{\theta},\mathbf{x}), \varphi_{YE}(\boldsymbol{\theta},\mathbf{x}), \varphi_C(\boldsymbol{\theta},\mathbf{x}) \right\}
\end{equation}
for $\boldsymbol{\theta} \in (-\pi,\pi]^2$, with
\begin{align*}
	\varphi_N(\boldsymbol{\theta},\mathbf{x}) & ~=~
         \begin{cases}
            \e^{\imath\boldsymbol{\theta}\cdot\mathbf{x}/{\Delta
                x}} \quad & \text{for }
              \mathbf{x}\in G_{{\Delta x},N}\\
              0 & \text{for } \mathbf{x} \in G_{\Delta x} \setminus
              G_{{\Delta x},N} \end{cases} \quad
	&&\varphi_{XE}(\boldsymbol{\theta},\mathbf{x}) ~=~
         \begin{cases}
            \e^{\imath\boldsymbol{\theta}\cdot\mathbf{x}/{\Delta
                x}} \quad & \text{for }
              \mathbf{x}\in G_{{\Delta x},XE}\\
              0 & \text{for } \mathbf{x} \in G_{\Delta x} \setminus
              G_{{\Delta x},XE} \end{cases} \\
	\varphi_{YE}(\boldsymbol{\theta},\mathbf{x}) & ~=~
         \begin{cases}
            \e^{\imath\boldsymbol{\theta}\cdot\mathbf{x}/{\Delta
                x}} \quad & \text{for }
              \mathbf{x}\in G_{{\Delta x},YE}\\
              0 & \text{for } \mathbf{x} \in G_{\Delta x} \setminus
              G_{{\Delta x},YE} \end{cases} \quad
	&&\hspace*{.6em}\varphi_C(\boldsymbol{\theta},\mathbf{x}) ~=~
         \begin{cases}
            \e^{\imath\boldsymbol{\theta}\cdot\mathbf{x}/{\Delta
                x}} \quad & \text{for }
              \mathbf{x}\in G_{{\Delta x},C}\\
              0 & \text{for } \mathbf{x} \in G_{\Delta x} \setminus
              G_{{\Delta x},C} \end{cases}
\end{align*}
Considering the discretization of a system of PDEs with $q$ types of
degrees of freedom (corresponding to both different functions in the
PDE system and different basis functions used in higher-order
discretizations of a single function, e.\,g., $q=16$ for the elasticity 
system) on a quadrilateral grid, the discretized operator can be written 
as a block operator,
\[
	\mathcal{L}_{\Delta x} = \begin{bmatrix}
		L_{\Delta x}^{1,1} & \cdots & L_{\Delta x}^{1,q}\\
		\vdots & & \vdots\\
		L_{\Delta x}^{q,1} & \cdots & L_{\Delta x}^{q,q}
	\end{bmatrix}
\]
with scalar differential operators, $L_{\Delta x}^{i,j}$, $i, j = 1,\ldots, q$.
The Fourier symbol, $\widetilde{\mathcal{L}}_{\Delta x}(\boldsymbol{\theta})$, of $\mathcal{L}_{\Delta x}$ is computed from the Fourier symbols of each block,
\[
	\widetilde{\mathcal{L}}_{\Delta x}(\boldsymbol{\theta}) = \begin{bmatrix}
		\widetilde{L}_{\Delta x}^{1,1}(\boldsymbol{\theta}) & \cdots & \widetilde{L}_{\Delta x}^{1,q}(\boldsymbol{\theta})\\
		\vdots & & \vdots\\
		\widetilde{L}_{\Delta x}^{q,1}(\boldsymbol{\theta}) & \cdots & \widetilde{L}_{\Delta x}^{q,q}(\boldsymbol{\theta})
	\end{bmatrix},
        \]
noting that $L_{\Delta x}^{i,j}$ is a map from a function associated with
degree of freedom $j$ to that associated with degree of freedom $i$,
which may be defined on different types of meshpoints.  Formally, we
think of $\widetilde{\mathcal{L}}_{\Delta x}(\theta)$ as describing the action
of $\mathcal{L}_{\Delta x}$ on the $q$-dimensional space given by linear
combinations of the Fourier modes associated with each type of degree
of freedom in the system.  Written as a $q\times q$ matrix,
$\widetilde{\mathcal{L}}_{\Delta x}(\theta)$ maps the coefficients of a
vector-valued function, $\vec{u}_{\Delta x}$, in this basis to the coefficients
of the function described by $\mathcal{L}_{\Delta x}\vec{u}_{\Delta x}$.

For the linear elasticity problem, the time-stepping operator, $\Phi$, is defined in Equation \eqref{eq:Phi_form},
\[
	\Phi = \begin{bmatrix}
\rho P \left(\rho M + \Delta t^2\mu K\right)^{-1}M & -\Delta t\mu
P\left(\rho M + \Delta t^2\mu K\right)^{-1}K \\
\rho\Delta tP \left(\rho M + \Delta t^2\mu K\right)^{-1}M & -(\Delta t)^2\mu
P\left(\rho M + \Delta t^2\mu K\right)^{-1}K+I
\end{bmatrix},
\]
with projection operator, $P =I-(\rho M + \Delta t^2 \mu K)^{-1}BS^{-1}B^T$, and $S = B^T(\rho M+\Delta t^2\mu K)^{-1}B$.
Noting that both the velocity and displacement degrees of freedom are
two-dimensional vector fields of $Q2$ degrees of freedom, we have four
scalar functions in the system, whose basis is naturally partitioned
into four types, with nodal, edge-based, and cell-centered degrees of
freedom.  Thus, the Fourier symbol, $\widetilde{\Phi}$, of $\Phi$ is a
block matrix of size $16\times 16$, computed from the Fourier symbols
of its component parts,
\[
	\widetilde{\Phi} = \begin{bmatrix}
\rho\widetilde{P} \left(\rho\widetilde{M} + \Delta t^2\mu\widetilde{K}\right)^{-1}\widetilde{M} & -\Delta t\mu
\widetilde{P}\left(\rho\widetilde{M} + \Delta t^2\mu\widetilde{K}\right)^{-1}\widetilde{K} \\
\rho\Delta t\widetilde{P} \left(\rho\widetilde{M} + \Delta t^2\mu\widetilde{K}\right)^{-1}\widetilde{M} & -(\Delta t)^2\mu
\widetilde{P}\left(\rho\widetilde{M} + \Delta t^2\mu\widetilde{K}\right)^{-1}\widetilde{K}+I
\end{bmatrix},
\]
with $8\times 8$ symbols $\widetilde{P} =I-(\rho\widetilde{M} + \Delta t^2 \mu\widetilde{K})^{-1}\widetilde{B}\widetilde{S}^{-1}\widetilde{B}^T$, $\widetilde{S} = \widetilde{B}^T(\rho\widetilde{M}+\Delta t^2\mu\widetilde{K})^{-1}\widetilde{B}$,
\[
	\widetilde{M} = \begin{bmatrix}
		\widetilde{M}_x\\
		& \widetilde{M}_y
	\end{bmatrix},\quad
	\widetilde{K} = \begin{bmatrix}
		\widetilde{K}_x\\
		& \widetilde{K}_y
	\end{bmatrix}\quad\text{and}\quad 
	\widetilde{B} = \begin{bmatrix}
		\widetilde{B}_x \\
		\widetilde{B}_y
	\end{bmatrix}.
        \]
The Fourier symbols of $M$, $K$, and $B$ are derived from standard
calculations\cite{YHe_SPMacLachlan_2018a,YHe_SPMacLachlan_2018b}, and
given for reference in Appendix \ref{sec:appendix:spatialLFA}.    
    

\subsubsection{LFA in time}
We consider the infinite uniform fine and coarse temporal grids,
\begin{equation}\label{eq:lfa:time:grids}
	G_t = \{t_l := l\Delta t ~:~ l\in\mathbb{N}_0\} \quad\text{and}\quad G_T = \{T_L := Lm\Delta t ~:~ L\in\mathbb{N}_0, ~m\in\mathbb{N}\},
\end{equation}	
where the grid $G_T$ is derived from $G_t$ by multiplying the mesh size, $\Delta t$, by a positive coarsening factor, $m$. The fundamental quantities in LFA are the Fourier modes, given by the grid-functions,
\[
	\varphi(\omega,t) = \e^{\imath\omega\cdot t/\Delta t}\quad\text{for } \omega\in\left(-\frac{\pi}{m},2\pi-\frac{\pi}{m}\right], ~t\in G_t,
\]
with frequency, $\omega$, taken to vary continuously in the interval
$\left(-\frac{\pi}{m},2\pi-\frac{\pi}{m}\right]$, although any
  interval of length $2\pi$ could be used instead. Fourier modes on
  the coarse grid are defined analogously on the interval $(-\pi,\pi]$
    as explained below. Considering a coarsening factor of $m$, a
    constant-stencil restriction operator maps $m$ fine-grid
    functions, the Fourier harmonics, into one coarse-grid
    function. More precisely, these $m$ functions consist of the mode
    associated with some base index
    $\omega^{(0)}\in\left(-\frac{\pi}{m},\frac{\pi}{m}\right]$ and
      those associated with the frequencies
\[
	\omega^{(p)} = \omega^{(0)} + \frac{2\pi p}{m}, \quad p = 1, \ldots, m-1.
\]
The Fourier harmonics associated with base frequency, $\omega^{(0)}$, define subspaces,
\begin{equation}
	\mathcal{E}_t^{\omega^{(0)}} := \Span_{0\leq p\leq m-1} \left\{\e^{\imath\omega^{(p)}\cdot t_l/\Delta t}\right\}, \quad t_l\in G_t,
\end{equation}
that are left invariant by the iteration operator. As a consequence, using the matrix of Fourier modes, ordered by Fourier harmonics, we can block-diagonalize the infinite grid operators with blocks corresponding to the spaces of harmonic modes. Each block describes the action on the coefficients $\{\alpha_p\}_{p=0,\ldots,m-1}$ of the expansion,
\[
	e_l = \sum_{p=0}^{m-1} \alpha_p\e^{\imath\omega^{(p)}\cdot t_l/\Delta t},
\]
of a function $e_l\in \mathcal{E}_t^{\omega^{(0)}}$.

Instead of analyzing the error propagation on this basis for the space of harmonics, $\mathcal{E}_t^{\omega^{(0)}}$, we consider the transformed basis
\begin{equation}\label{eq:transformed:basis}
	\widehat{\mathcal{E}}_t^{\omega^{(0)}} := \Span_{0\leq r\leq m-1}\left\{\e^{\imath\omega^{(0)}\cdot t_{Lm+r}/\Delta t}\widehat{E}_{r}\right\}, \quad L\in\mathbb{N},
\end{equation}
where $\widehat{E}_{r}$ denotes the vector with entries equal to one
at time points $t_{Lm+r}$, for $L\in\mathbb{N}$, and zero
otherwise. This transformed basis is motivated by the following:
Consider a function, $e_{l}\in \mathcal{E}_t^{\omega^{(0)}}$, with $l
= Lm+r$ for $L\in\mathbb{N}$ and $r\in\{0,1,\ldots,m-1\}$. Then,
\[
	e_{Lm+r} = \sum_{p=0}^{m-1} \alpha_p\e^{\imath\omega^{(p)}\cdot t_{Lm+r}/\Delta t} = \sum_{p=0}^{m-1}\alpha_p\e^{\imath(\omega^{(0)}+\frac{2\pi p}{m})\cdot t_{Lm+r}/\Delta t} = \sum_{p=0}^{m-1}\alpha_p\e^{\imath\omega^{(0)}\cdot t_{Lm+r}/\Delta t}\e^{\imath\frac{2\pi pr}{m}} = \underbrace{\left[\sum_{p=0}^{m-1}\alpha_p\e^{\imath\frac{2\pi pr}{m}}\right]}_{=:\widehat{\alpha}_r}\e^{\imath\omega^{(0)}\cdot t_{Lm+r}/\Delta t}.
\]
Thus, any function in $\mathcal{E}_t^{\omega^{(0)}}$ can be expressed
in terms of the basis defining
$\widehat{\mathcal{E}}_t^{\omega^{(0)}}$, with coefficients
$\{\widehat{\alpha}_r\}_{r=0,\ldots,m-1}$ that depend only on the
``offset'' of fine-grid point $l=Lm+r$ from the coarse-grid indices
where $r=0$.  Moreover, implicit in this expression is an associated
coarse-grid frequency of $m\omega^{(0)}$, so that
$\e^{\imath\omega^{(0)}\cdot t_{Lm+r}/\Delta t} =
\e^{\imath(m\omega^{(0)})\cdot t_{Lm+r}/\Delta T}$.  In many senses,
this is a more natural basis for the space of Fourier harmonics, since
it directly associates a repeating pattern of coefficients in the basis of
$\widehat{\mathcal{E}}_t^{\omega^{(0)}}$ with the temporal mesh, 
as illustrated in Figure \ref{fig:repeating:pattern:m}.
	
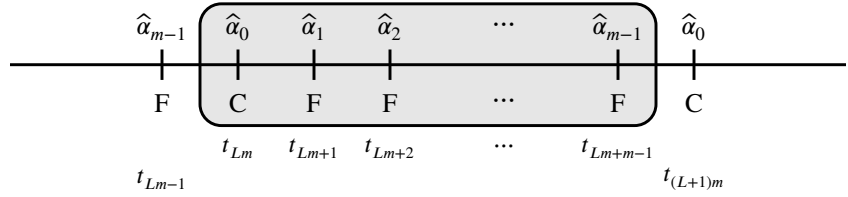
\begin{figure}[h!t]
	\centering
		\begin{tikzpicture}[line width=1.1pt]
			\draw[rounded corners=8pt,fill=gray!20] (-.5,-.8) rectangle (5.5,.8);
			\draw (-3,0) -- (8,0);
			\foreach \i in {-1,0,1,2,5,6}
				\draw (\i,-.2) -- (\i,.2);
			\foreach \i in {0,1,2}
				\draw (\i,.5) node {$\widehat{\alpha}_\i$};
			\draw (3.5,.5) node {$\cdots$};
			\draw (5,.5) node {$\widehat{\alpha}_{m-1}$};
				
			\draw (-1,.5) node {$\widehat{\alpha}_{m-1}$};
			\draw (6,.5) node {$\widehat{\alpha}_0$};
				
			\draw (0,-1.1) node {\scriptsize $t_{Lm}$}; \draw (1,-1.1) node {\scriptsize $t_{Lm+1}$};
			\draw (2,-1.1) node {\scriptsize $t_{Lm+2}$}; \draw (5,-1.1) node {\scriptsize $t_{Lm+m-1}$};
			\draw (3.5,-1.1) node {\scriptsize $\cdots$};
				
			\draw (6,-1.5) node {\scriptsize $t_{(L+1)m}$}; \draw (-1,-1.5) node {\scriptsize $t_{Lm-1}$};
				
			\foreach \i in {0,6}
				\draw (\i,-.5) node {C};
			\foreach \i in {-1,1,2,5}
				\draw (\i,-.5) node {F};
			\draw (3.5,-.5) node {$\cdots$};
		\end{tikzpicture}
	\caption{Visualization of the repeating pattern of $m$ data values on the time grid.}
	\label{fig:repeating:pattern:m}
\end{figure}

For the three-level analysis, we consider a hierarchy of three time grids, with grid spacings $\Delta t$, $m\Delta t$, and $m_2m\Delta t$, respectively. Correspondingly, we consider base frequencies, $\omega^{(0)}\in\left(-\frac{\pi}{m_2m},\frac{\pi}{m_2m}\right]$, and time-series coefficients $\{\widehat{\alpha}_r\}_{r=0,\ldots,m_2m-1}$, $\{\widehat{\alpha}_r\}_{r=0,m,\ldots,(m_2-1)m}$, and $\widehat{\alpha}_0$, respectively.
Figure \ref{fig:time:series:coeff:3grid} shows the time-series coefficients on the time grids for the case $m=4$ and $m_2=2$.

\begin{figure}[h!t]
	\begin{center}
		\begin{tikzpicture}[line width=1.1pt,scale=.75]
			\draw[rounded corners=8pt,fill=gray!20] (-.5,-1) rectangle (7.5,1);
			\draw (-5,0) -- (13,0);
			\foreach \i in {-4,...,12}
				\draw (\i,-.2) -- (\i,.2);
			\foreach \i in {0,1,...,7}
				\draw (\i,.55) node {$\widehat{\alpha}_\i$};
			\draw (-1,.55) node {$\widehat{\alpha}_7$};
			\draw (8,.55) node {$\widehat{\alpha}_0$};
			\foreach \i in {-4,0,4,8,12}
				\draw (\i,-.55) node {C};
			\foreach \i in {-3,-2,-1,1,2,3,5,6,7,9,10,11}
				\draw (\i,-.55) node {F};
			\draw[line width=1.5pt,,decorate,decoration={brace,amplitude=3pt},yshift=5pt] (11,-1.2) -- (10,-1.2) node[below=3pt,midway] {$\Delta t$};	
				
			\draw[rounded corners=8pt,fill=gray!20] (-.5,-4) rectangle (6,-2);
			\draw (-5,-3) -- (13,-3);
			\foreach \i in {-4,0,...,12}
				\draw (\i,-3.2) -- (\i,-2.8);	
			\draw (0,-2.45) node {$\widehat{\alpha}_0$}; \draw (4,-2.45) node {$\widehat{\alpha}_4$};
			\draw (-4,-2.45) node {$\widehat{\alpha}_4$};
			\draw (8,-2.45) node {$\widehat{\alpha}_0$};
			\foreach \i in {0,8}
				\draw (\i,-3.55) node {C};
			\foreach \i in {-4,4,12}
				\draw (\i,-3.55) node {F};
			\draw[line width=1.5pt,,decorate,decoration={brace,amplitude=3pt},yshift=5pt] (12,-4.2) -- (8,-4.2) node[below=3pt,midway] {$m\Delta t$};		
				
			\draw[rounded corners=8pt,fill=gray!20] (-.5,-6.6) rectangle (4,-5);
			\draw (-5,-6) -- (13,-6);
			\foreach \i in {0,8}
				\draw (\i,-6.2) -- (\i,-5.8);	
			\draw (0,-5.45) node {$\widehat{\alpha}_0$}; \draw (8,-5.45) node {$\widehat{\alpha}_0$};	
			\draw[line width=1.5pt,,decorate,decoration={brace,amplitude=3pt},yshift=5pt] (8,-7) -- (0,-7) node[below=3pt,midway] {$m_2m\Delta t$};
		\end{tikzpicture}
	\end{center}
	\caption{Visualization of the time-series coefficients on the three-level time-grid hierarchy with coarsening factors $m=4$ and $m_2=2$.}
	\label{fig:time:series:coeff:3grid}
\end{figure}
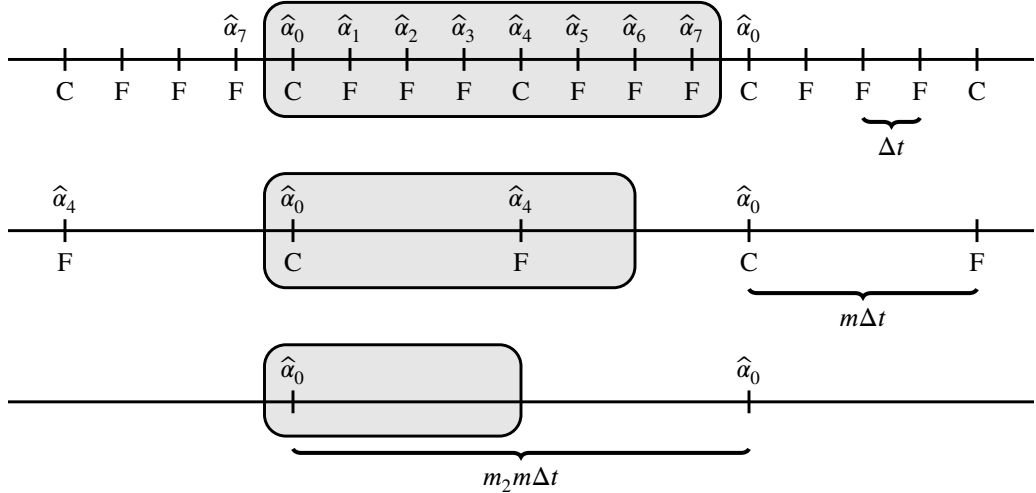


\subsubsection{Space-time LFA}
For simplicity, we only describe two-grid space-time LFA; the
three-level analysis can be done using inductive arguments. We
consider two infinite uniform grids, $\Omega_h := G_{\Delta x}\otimes
G_t$, and $\Omega_H := G_{\Delta x}\otimes G_T$ with $G_{\Delta x},
G_t$, and $G_T$ defined in Equations \eqref{eq:lfa:space:grid} and
\eqref{eq:lfa:time:grids}. The subscript $h$ represents the pair
$({\Delta x},\Delta t)$, where ${\Delta x}$ denotes the mesh size in
both spatial dimensions, and $\Delta t$ is the time-step size. The
grid $\Omega_H$ is derived from $\Omega_h$ by multiplying the mesh
size in the time dimension only, i.\,e., $H$ represents the pair
$(\Delta x, m\Delta t)$. On $\Omega_h$, we consider an infinite-grid
(multilevel) Toeplitz operator, $A_h$, corresponding to the
discretization of a time-dependent PDE, or a system of PDEs, in 2D
space. Since the operator, $A_h$, is (multilevel) Toeplitz, it can be block-diagonalized by a set of continuous space-time Fourier modes,
\[
	\psi(\boldsymbol{\theta},\omega,\mathbf{x},t) = \e^{\imath\boldsymbol{\theta}\cdot\mathbf{x}/\Delta x}\e^{\imath\omega\cdot t/\Delta t}\quad\text{for }\boldsymbol{\theta}\in(-\pi,\pi]^2, \omega\in(-\pi,\pi], (\mathbf{x},t)\in \Omega_h.
\]
Considering semicoarsening in time, we analyze the error propagation on the space $\mathcal{E}^{(\theta,\omega^{(0)})}:=\Span\{\e^{\imath\boldsymbol{\theta}\cdot\mathbf{x}_{j,k}/\Delta x}\}\otimes\widehat{\mathcal{E}}_t^{\omega^{(0)}}$, for spatial frequencies $\boldsymbol{\theta}\in(-\pi,\pi]^2$ and temporal base frequencies $\omega^{(0)}\in\left(-\frac{\pi}{m},\frac{\pi}{m}\right]$, i.\,e., we compute the action of the operators in the two-grid cycle on the coefficients
$\{\widehat{\alpha}_r\}_{r=0,\ldots,m-1}$ of the expansion
\[
	e_{j,k,l} = \widehat{\alpha}_{l\bmod m}\e^{\imath\omega^{(0)}\cdot t_l/\Delta t}\e^{\imath\boldsymbol{\theta}\cdot\mathbf{x}_{j,k}/\Delta x}
\]
of a function $e_{j,k,l}\in\mathcal{E}^{(\theta,\omega^{(0)})}$.
\begin{definition}
	Let $L_h$ be an infinite-grid (multilevel) Toeplitz operator, and for given $\boldsymbol{\theta}\in(-\pi,\pi]^2, \omega^{(0)}\in\left(-\frac{\pi}{m},\frac{\pi}{m}\right]$, let 
	\[
		e_{j,k,l}(\boldsymbol{\theta},\omega^{(0)}) =
                \widehat{\alpha}_{l \bmod m}\e^{\imath\omega^{(0)}\cdot t_l/\Delta t}\e^{\imath\boldsymbol{\theta}\cdot\mathbf{x}_{j,k}/\Delta x}\in\mathcal{E}^{(\theta,\omega^{(0)})}
	\]
	be a function with uniquely defined coefficients $\{\widehat{\alpha}_r\}_{r=0,\ldots,m-1}$. The transformation of the coefficients under the operator $L_h$,
	\[
		\begin{bmatrix}
			\widehat{\alpha}_0\\
			\widehat{\alpha}_1\\
			\vdots\\
			\widehat{\alpha}_{m-1}
		\end{bmatrix} \quad\leftarrow\quad \widehat{L}_h(\boldsymbol{\theta},\omega^{(0)})\begin{bmatrix}
			\widehat{\alpha}_0\\
			\widehat{\alpha}_1\\
			\vdots\\
			\widehat{\alpha}_{m-1}
		\end{bmatrix},
	\]
	defined by
	\[
		\left(L_h e(\boldsymbol{\theta},\omega^{(0)})\right)_{j,k,l} = \widehat{L}_h(\boldsymbol{\theta},\omega^{(0)}) e_{j,k,l}(\boldsymbol{\theta},\omega^{(0)}),
	\]
	is called the \emph{space-time Fourier symbol}, $\widehat{L}_h = \widehat{L}_h(\boldsymbol{\theta},\omega^{(0)})$, of $L_h$.
\end{definition}

A complete convergence analysis of the convergence properties of a Parareal or MGRIT algorithm arises from computing the symbols of each individual operator of the corresponding iteration matrix, $E$, to obtain the symbol, $\widehat{E}(\boldsymbol{\theta},\omega^{(0)})$, of the iteration matrix as a whole. The asymptotic convergence behavior can then be predicted by calculating the asymptotic convergence factor
\[
	\rho_\text{LFA}(E) = \sup\left\{\rho(\widehat{E}(\boldsymbol{\theta},\omega^{(0)})) ~:~ \boldsymbol{\theta}\in(-\pi,\pi]^2, \omega^{(0)}\in\left(-\frac{\pi}{m},\frac{\pi}{m}\right]\right\}.
\]
Note that, in practice, LFA has to be taken from its infinite-grid setting, and we maximize, instead, over a finite set of frequency tuples, $(\boldsymbol{\theta},\omega^{(0)})$. 

Similarly, we introduce the error reduction factor
\[
	\sigma_\text{LFA}(E) = \sup\left\{\|\widehat{E}(\boldsymbol{\theta},\omega^{(0)})\|_2 ~:~ \boldsymbol{\theta}\in(-\pi,\pi]^2, \omega^{(0)}\in\left(-\frac{\pi}{m},\frac{\pi}{m}\right]\right\},
\]
again maximizing over finite sets of values of $\boldsymbol{\theta}$ and $\omega^{(0)}$ to get a computable prediction. Since for both the Parareal and MGRIT approaches, iteration operators have only a single eigenvalue of zero, only the error reduction factor provides insight into the convergence behavior of a Parareal or MGRIT algorithm. Furthermore, due to the non-normality of the iteration operators, it is crucial to not only consider the iteration matrix itself, but also powers of the iteration matrix to examine the short-term convergence behavior. More precisely, for any initial error, $\boldsymbol{e}_0$, $\|E^k\boldsymbol{e}_0\|\leq\|E^k\|\|\boldsymbol{e}_0\|\leq\|E\|^k\|\boldsymbol{e}_0\|$, where $k$ denotes the number of iterations of the multigrid scheme. Thus, we introduce the error reduction factors
\begin{equation}\label{eq:err_red_factor_lfa}
	\sigma_\text{LFA}(E^k) = \sup\left\{\|(\widehat{E}(\boldsymbol{\theta},\omega^{(0)}))^k\|_2 ~:~ \boldsymbol{\theta}\in(-\pi,\pi]^2, \omega^{(0)}\in\left(-\frac{\pi}{m},\frac{\pi}{m}\right]\right\},\quad\text{for }k\geq 1,
\end{equation}
providing the worst-case error reduction, i.\,e., an upper bound for
the error reduction, in iteration $k$ for the time-periodic problem.
In practice, we observe that this also provides upper bounds for the
worst-case error reduction in the non-periodic case but this, of
course, is no longer a rigorous bound.


\paragraph{Space-time Fourier symbols of MGRIT with F- and FCF-relaxation}
Let $\widetilde{\Phi}_\theta$ and $\widetilde{\Phi}_{c,\theta}$ denote
the $q\times q$ ($q\geq 1$) spatial Fourier symbols of the fine- and
coarse-scale time integrators, $\Phi$ and $\Phi_c$, respectively, and
let $I$ denote an identity matrix of size $q\times q$. Then, the
space-time Fourier symbol of the fine-grid operator, $A$, is given by
the $qm\times qm$ matrix
\[
	\widehat{A} = \begin{bmatrix}
			I & 0  & 0 & 0 & \cdots & 0 & -\widetilde{\Phi}_\theta\e^{-\imath\omega^{(0)}}\\[.2\normalbaselineskip]
			-\widetilde{\Phi}_\theta\e^{-\imath\omega^{(0)}} & I & 0 & 0 & \cdots & 0 & 0\\[.2\normalbaselineskip]
			0 & -\widetilde{\Phi}_\theta\e^{-\imath\omega^{(0)}} & I & 0 & \cdots & 0 & 0\\[.2\normalbaselineskip]
			& & \ddots & \ddots\\[.2\normalbaselineskip]
			& & & & & -\widetilde{\Phi}_\theta\e^{-\imath\omega^{(0)}} & I
		\end{bmatrix}.
\]
The set of modes $\psi(\boldsymbol{\theta},m\omega^{(0)})$ for
$\omega^{(0)}\in\left(-\frac{\pi}{m},\frac{\pi}{m}\right]$ is a
  complete set of space-time Fourier modes on the coarse grid,
  $\Omega_H$. Further, any function in
  $\mathcal{E}^{(\theta,\omega^{(0)})}$ is aliased on the coarse grid
  with the function $\widehat{\alpha}_0\psi(\boldsymbol{\theta},m\omega^{(0)})$ with a fixed coefficient $\widehat{\alpha}_0$. As a consequence, the space-time Fourier symbol of the coarse-grid operator, $A_c$, is a block matrix with one block of size $q\times q$, given by
\[
	\widehat{A}_c = I-\widetilde{\Phi}_{c,\theta}\e^{-\imath m\omega^{(0)}}.
\]
The space-time Fourier symbols of the interpolation and restriction operators, $P_\Phi$, and $R_I$, are block matrices with $m\times 1$ blocks or $1\times m$ blocks of size $q\times q$,
\[
	\widehat{P}_\Phi = \begin{bmatrix}
			I\\[.2\normalbaselineskip]
			\widetilde{\Phi}_\theta\e^{-\imath\omega^{(0)}}\\[.2\normalbaselineskip]
			\widetilde{\Phi}^2_\theta\e^{-\imath2\omega^{(0)}}\\[.2\normalbaselineskip]
			\vdots\\[.2\normalbaselineskip]
			\widetilde{\Phi}^{m-1}_\theta\e^{-\imath(m-1)\omega^{(0)}}
		\end{bmatrix}\quad\text{and}\quad \widehat{R}_I = \begin{bmatrix}
			I & 0 & \cdots & 0
		\end{bmatrix}.
\]
Since the F- and FCF-relaxation operators, $S^F$ and $S^{FCF}$, are equal to $P_\Phi R_I$ and $P_\Phi(I-A_S)R_I$, respectively, the space-time Fourier symbols of F- and FCF-relaxation are defined by
\[
	\widehat{S}^F = \widehat{P}_\Phi\widehat{R}_I, \quad \widehat{S}^{FCF} = \widehat{P}_\Phi(I-\widehat{A}_S)\widehat{R}_I, \quad\text{with } \widehat{A}_S = I - \widetilde{\Phi}^m_\theta\e^{-\imath m\omega^{(0)}}.
\]
This completes the definition of the space-time Fourier symbols of the
operators for F- and FCF-relaxation and, by using the expressions \eqref{eq:Parareal:it:matrix} and \eqref{eq:two:level:mgrit}, this defines the space-time Fourier symbols of the two-level methods as a whole.

A similar approach extends to the three-grid case, however, the grid hierarchy affects the block size of the space-time Fourier symbols. Considering temporal semicoarsening by factors of $m$ and $m_2$, respectively, to obtain the first and second coarse grids, the space-time Fourier symbol of the fine-grid operator, $A$, is the same as in the two-level case, but with $m_2m\times m_2m$ blocks of size $q\times q$ instead of $m\times m$ blocks of size $q\times q$. Denoting the spatial Fourier symbols of the time integrators, $\Phi_c$ and $\Phi_{cc}$ on the first and second coarse grid by $\widetilde{\Phi}_{c,\theta}$ and $\widetilde{\Phi}_{cc,\theta}$, respectively, the space-time Fourier symbol of the coarse-grid operator, $A_c$, on the first coarse grid is a block matrix with $m_2\times m_2$ blocks of size $q\times q$,
\[
	\widehat{A}_c = \begin{bmatrix}
			I & 0  & 0 & 0 & \cdots & 0 & -\widetilde{\Phi}_{c,\theta}\e^{-\imath m\omega^{(0)}}\\[.2\normalbaselineskip]
			-\widetilde{\Phi}_{c,\theta}\e^{-\imath m\omega^{(0)}} & I & 0 & 0 & \cdots & 0 & 0\\[.2\normalbaselineskip]
			0 & -\widetilde{\Phi}_{c,\theta}\e^{-\imath m\omega^{(0)}} & I & 0 & \cdots & 0 & 0\\[.2\normalbaselineskip]
			& & \ddots & \ddots\\[.2\normalbaselineskip]
			& & & & & -\widetilde{\Phi}_{c,\theta}\e^{-\imath m\omega^{(0)}} & I
		\end{bmatrix},
\]
and the space-time Fourier symbol of the coarse-grid operator, $A_{cc}$, on the second coarse grid is a block matrix with one block of size $q\times q$,
\[
	\widehat{A}_{cc} = I-\widetilde{\Phi}_{cc,\theta}\e^{-\imath m_2m\omega^{(0)}}.
\]
The space-time Fourier symbols of the Schur complement coarse-grid operators on the first and second coarse grids, $A_S$ and $A_{c,S}$, are defined as those of the coarse-grid operators, $A_c$ and $A_{cc}$, with $\widetilde{\Phi}_{c,\theta}$ and $\widetilde{\Phi}_{cc,\theta}$ replaced by $\widetilde{\Phi}_\theta^m$ and $\widetilde{\Phi}_\theta^{m_2m}$.

The space-time Fourier symbols of the restriction operators, $R_I$ and $R_{c,I}$, from the fine grid to the first coarse grid and from the first coarse grid to the second coarse grid are block matrices with $m_2\times m_2m$ blocks of size $q\times q$ or $1\times m_2$ blocks of size $q\times q$, respectively,
\begin{center}
	\begin{tikzpicture}
		\node {$\setcounter{MaxMatrixCols}{20}
		\widehat{R}_I = \begin{bmatrix}
		I & 0 & \cdots & 0\\
		& & & & I & 0 & \cdots & 0\\
		& & & & & & & & \ddots\\
		& & & & & & & & & I & 0 & \cdots & 0
	\end{bmatrix}\quad\text{and}\quad \widehat{R}_{c,I} = \begin{bmatrix}
		I & 0 & \cdots & 0
	\end{bmatrix}.
		$};
		\draw[line width=1.15pt,decorate,decoration={brace,amplitude=3pt},yshift=5pt] (-2.6,.3) -- (-3.95,.3) node[below=3pt,midway] {\scriptsize$m$ blocks};
	\end{tikzpicture}
\end{center}
Similarly, the space-time Fourier symbols of the interpolation operators, $P_\Phi$ and $P_{\Phi_c}$, from the first coarse grid to the fine grid and from the second coarse grid to the first coarse grid, respectively, are block matrices with $m_2m\times m_2$ blocks of size $q\times q$ or $m_2\times 1$ blocks of size $q\times q$, respectively,
\[
	\widehat{P}_\Phi = \begin{bmatrix}
			Z^{(P_\Phi)}\\[.2\normalbaselineskip]
			& Z^{(P_\Phi)}\\[.2\normalbaselineskip]
			& & \ddots\\[.2\normalbaselineskip]
			& & & Z^{(P_\Phi)}
	\end{bmatrix}, \quad\text{with}\quad Z^{(P_\Phi)} = \begin{bmatrix}
		 	I \\[.2\normalbaselineskip]
			\widetilde{\Phi}_\theta\e^{-\imath\omega^{(0)}} \\[.2\normalbaselineskip]
			\widetilde{\Phi}_\theta^2\e^{-\imath2\omega^{(0)}} \\[.2\normalbaselineskip]
			\vdots\\[.2\normalbaselineskip]
			\widetilde{\Phi}_\theta^{m-1}\e^{-\imath(m-1)\omega^{(0)}} 
		\end{bmatrix}\quad\text{and}\quad \widehat{P}_{\Phi^c} = \begin{bmatrix}
		 	I \\[.2\normalbaselineskip]
			\widetilde{\Phi}_{c,\theta}\e^{-\imath m\omega^{(0)}} \\[.2\normalbaselineskip]
			\widetilde{\Phi}_{c,\theta}^2\e^{-\imath2m\omega^{(0)}} \\[.2\normalbaselineskip]
			\vdots\\[.2\normalbaselineskip]
			\widetilde{\Phi}_{c,\theta}^{m_2-1}\e^{-\imath(m_2-1)m\omega^{(0)}} 
		\end{bmatrix}.
\]


\subsection{Semi-algebraic mode analysis}

LFA focuses on the local character of the operators defining the
multilevel algorithms. This means that the effects of boundary
conditions, including the initial condition in the time dimension, are
ignored. Unless long time intervals are considered, LFA can fail to
produce its usual quality of predictive results for
convection-dominated or parabolic problems
\cite{ABrandt_1981a,SVandewalle_GHorton_1995,CWOosterlee_etal_1998a,SFriedhoff_etal_2012a}. Semi-algebraic
mode analysis (SAMA) \cite{SFriedhoff_SMacLachlan_2015a} is one
applicable approach of mode analysis that enables accurate predictions
of the convergence behavior for multigrid and related multilevel
methods. The analysis combines spatial LFA with algebraic computations
that account for the non-local character of the operators in time. Other applicable approaches, not considered here, include half-space mode analysis \cite{ABrandt_1981a,ABrandt_1984a,ABrandt_IYavneh_1992a,ABrandt_BDiskin_1999a,BDiskin_JLThomas_2000a} that considers convergence on a discrete half-plane instead of the full infinite lattice used in LFA, and Fourier-Laplace analysis \cite{STaasan_HZhang_1995a,STaasan_HZhang_1995b,JJanssen_SVandewalle_1996a,JJanssen_SVandewalle_1996b,JvanLent_SVandewalle_2002a}, based on using discrete Laplace transforms. 

For simplicity, we describe SAMA for the two-level MGRIT algorithm with FCF-relaxation, represented by the two-level iteration matrix, $E^{FCF}$, defined in Equation \eqref{eq:two:level:mgrit},
\[
	E^{FCF} = P_\Phi(I-A_c^{-1}A_S)(I-A_S)R_I;
\]
the analysis of other variants is done analogously by considering the
corresponding iteration matrix. Motivated by the global block-Toeplitz
space-time structure of the operators that define the iteration
matrix, the idea of SAMA is to use a Fourier ansatz in space to
block-diagonalize the spatial blocks of each individual operator. More
precisely, we block-diagonalize the blocks of each operator using the matrix, $\Psi$, of discretized spatial Fourier modes,
\[
	\psi(\boldsymbol{\theta},\mathbf{x}) = \e^{\imath\boldsymbol{\theta}\cdot\mathbf{x}/\Delta x} \quad\text{for } \boldsymbol{\theta}\in(-\pi,\pi]^2, \mathbf{x}\in G_{\Delta x},
\]
with $N_x$ Fourier frequency pairs,
$\boldsymbol{\theta}\in(-\pi,\pi]^2$, sampled on a uniform
  quadrilateral mesh given by the tensor product of an equally-spaced
  mesh over the interval of length $2\pi$ with itself, assuming $N_x$
  degrees of freedom in the spatial discretization. For example, the
  spatial blocks of a scalar coarse-grid operator, $A_c$, are
  block-diagonalized by computing $\mathcal{F}^{-1}A_c\mathcal{F}$,
  where $\mathcal{F} = I_{N_T+1}\otimes\Psi$. The resulting matrices
  (on the fine temporal grid) are then reordered from $N_t+1\times N_t+1$ blocks of size $N_x\times N_x$ to a block matrix with $N_x\times N_x$ blocks of size $N_t+1\times N_t+1$ to obtain a block diagonal structure, with each block corresponding to the evolution of one spatial Fourier mode over time. Applying SAMA to the iteration matrix, $E^{FCF}$, we obtain a block diagonal matrix,
\[
	\mathcal{P}^{-1}\mathcal{F}^{-1}E^{FCF}\mathcal{F}\mathcal{P} = \Diag(B_{\boldsymbol{\theta}}^{\left(E^{FCF}\right)})_{\boldsymbol{\theta}\in(-\pi,\pi]^2},
\]
with a discrete choice of Fourier frequencies, $\boldsymbol{\theta}$, and
\begin{equation}\label{eq:sama:two:level:mgrit}
	B_{\boldsymbol{\theta}}^{\left(E^{FCF}\right)} = B_{\boldsymbol{\theta}}^{(P_\Phi)}\left(I-\left(B_{\boldsymbol{\theta}}^{(A_c)}\right)^{-1}B_{\boldsymbol{\theta}}^{(A_S)}\right)\left(I-B_{\boldsymbol{\theta}}^{(A_S)}\right)B_{\boldsymbol{\theta}}^{(R_I)},
\end{equation}
where the diagonal blocks of the Fourier-transformed and permuted operators are denoted by $B_{\boldsymbol{\theta}}^{(\cdot)}$ marking the respective operator in the superscript and the spatial Fourier frequency pair, $\boldsymbol{\theta}$, in the subscript. Note that the grid hierarchy and size of the spatial Fourier symbol of the time integrators, $\Phi$ and $\Phi_c$, affect the block size of the transformed operators. Considering temporal semicoarsening by a factor of $m$ in the MGRIT approach and assuming that the spatial Fourier symbol of the time-integration operators, $\Phi$ and $\Phi_c$, are of size $q\times q$, the blocks $B_{\boldsymbol{\theta}}^{(A_c)}$ and $B_\theta^{(A_S)}$ are of size $(N_t/m+1)q\times (N_t/m+1)q$, and the blocks $B_{\boldsymbol{\theta}}^{(P_\Phi)}$ and $B_{\boldsymbol{\theta}}^{(R_I)}$ are of size $(N_t+1)q\times (N_t/m+1)q$ and $(N_t/m+1)q\times (N_t+1)q$, respectively. 

The short-term convergence behavior of MGRIT with iteration matrix, $E$, can then be predicted by calculating the error reduction factor (corresponding to worst-case 
error reduction in iteration $k$)
\begin{equation}\label{eq:err_red_factor_sama}
	\sigma_\text{SAMA}(E^k) = \sup\left\{\left\|\left(B_{\boldsymbol{\theta}}^{(E)}\right)^k\right\|_2 ~:~ \boldsymbol{\theta}\in(-\pi,\pi]^2\right\},\quad\text{for }k\geq 1.
\end{equation}
Similarly to LFA, in practice, we maximize over a finite set of frequencies, $\boldsymbol{\theta}$.  \revision{We note that the prediction given by SAMA is \textit{exact} in special cases, notably when $\Phi$ has (block-)circulant structure and we sample at $N_x$ evenly-spaced Fourier points for the underlying LFA symbols (so that LFA coincides with rigorous Fourier analysis).  In that setting, \eqref{eq:err_red_factor_sama} provides an exact worst-case convergence bound for MGRIT performance, corresponding to the worst-case initial space-time error, against which the quality of predictions by other approaches can be measured.  Note, however, that for any given initial value for the underlying PDE problem, the frequency content of the initial error may or may not sample the worst-case modes and, so, this prediction can still be an over-estimate of the actual MGRIT convergence factor, consistent with Bal's analysis for Parareal\cite{Bal_2005}.}

\paragraph{SAMA block matrices for MGRIT with F- and FCF-relaxation}
Consider temporal semicoarsening by factors of $m$ and $m_2$, respectively, to obtain the first and second coarse grids, and let $\widetilde{\Phi}_\theta$, $\widetilde{\Phi}_{c,\theta}$, and $\widetilde{\Phi}_{cc,\theta}$ again denote the spatial Fourier symbols of the time integrators, $\Phi$, $\Phi_c$, and $\Phi_{cc}$ on the fine, first, and second coarse grid, respectively; for a two-level hierarchy, consider the first coarse grid only. Then, the diagonal blocks of the Fourier-transformed and permuted operators, $A$, $A_c$, and $A_{cc}$ are block bidiagonal matrices with $(N_t+1)\times(N_t+1)$, $(N_t/m+1)\times(N_t/m+1)$, or $(N_t/(m_2m)+1)\times(N_t/(m_2m)+1)$ blocks, respectively, of size $q\times q$, given by
\[
	B_{\boldsymbol{\theta}}^{(A)} = \begin{bmatrix}
		I\\
		-\widetilde{\Phi}_\theta & I\\
		& \ddots & \ddots\\
		& & -\widetilde{\Phi}_\theta & I		
	\end{bmatrix},\quad B_{\boldsymbol{\theta}}^{(A_c)} = \begin{bmatrix}
		I\\
		-\widetilde{\Phi}_{c,\theta} & I\\
		& \ddots & \ddots\\
		& & -\widetilde{\Phi}_{c,\theta} & I		
	\end{bmatrix},\quad B_{\boldsymbol{\theta}}^{(A_{cc})} = \begin{bmatrix}
		I\\
		-\widetilde{\Phi}_{cc,\theta} & I\\
		& \ddots & \ddots\\
		& & -\widetilde{\Phi}_{cc,\theta} & I		
	\end{bmatrix}.
\]
The diagonal blocks of the Fourier-transformed and permuted Schur complement coarse-grid operators on the first and second coarse grids, $A_S$ and $A_{c,S}$, are defined analogously, with $-\widetilde{\Phi}_\theta^m$ and $-\widetilde{\Phi}_\theta^{m_2m}$ on the first subdiagonal.

The diagonal blocks of the Fourier-transformed and permuted restriction operators, $R_I$ and $R_{c,I}$, from the fine grid to the first coarse grid and from the first coarse grid to the second coarse grid are block matrices with $(N_t/m+1)\times(N_t+1)$ blocks of size $q\times q$ or $(N_t/(m_2m)+1)\times(N_t/m+1)$ blocks of size $q\times q$, respectively,
\begin{center}
	\begin{tikzpicture}
		\node {$\setcounter{MaxMatrixCols}{20}
		B_{\boldsymbol{\theta}}^{(R_I)} = \begin{bmatrix}
		I & 0 & \cdots & 0\\
		& & & & I & 0 & \cdots & 0\\
		& & & & & & & & \ddots\\
		& & & & & & & & & I & 0 & \cdots & 0\\
		& & & & & & & & & & & & & I
	\end{bmatrix}\quad\text{and}\quad B_{\boldsymbol{\theta}}^{(R_{c,I})}= \begin{bmatrix}
		I & 0 & \cdots & 0\\
		& & & & I & 0 & \cdots & 0\\
		& & & & & & & & \ddots\\
		& & & & & & & & & I & 0 & \cdots & 0\\
		& & & & & & & & & & & & & I
	\end{bmatrix}.
		$};
		\draw[line width=1.15pt,decorate,decoration={brace,amplitude=3pt},yshift=5pt] (-4.6,.5) -- (-5.95,.5) node[below=3pt,midway] {\scriptsize$m$ blocks};
		\draw[line width=1.15pt,decorate,decoration={brace,amplitude=3pt},yshift=5pt] (3.3,.5) -- (1.95,.5) node[below=3pt,midway] {\scriptsize$m_2$ blocks};
	\end{tikzpicture}
\end{center}
Similarly, the diagonal blocks of the Fourier-transformed and permuted interpolation operators, $P_\Phi$ and $P_{\Phi_c}$, from the first coarse grid to the fine grid and from the second coarse grid to the first coarse grid, respectively, are block matrices with $(N_t+1)\times(N_t/m+1)$ blocks of size $q\times q$ or $(N_t/m+1)\times(N_t/(m_2m)+1)$ blocks of size $q\times q$, respectively,
\[
	B_{\boldsymbol{\theta}}^{(P_\Phi)} = \begin{bmatrix}
		Z^{(\widetilde{\Phi}_\theta)}\\
		& \ddots\\
		& &Z^{(\widetilde{\Phi}_\theta)}\\
		& & & I
	\end{bmatrix}\quad\text{and}\quad B_{\boldsymbol{\theta}}^{(P_{\Phi_c})} = \begin{bmatrix}
		Z^{(\widetilde{\Phi}_{c,\theta})}\\
		& \ddots\\
		& &Z^{(\widetilde{\Phi}_{c,\theta})}\\
		& & & I
	\end{bmatrix}, \quad\text{with}\quad Z^{(\widetilde{\Phi}_\theta)} = \begin{bmatrix}
		I\\
		\widetilde{\Phi}_\theta\\
		\vdots\\
		\widetilde{\Phi}_\theta^{m-1}
	\end{bmatrix}\quad\text{and}\quad Z^{(\widetilde{\Phi}_{c,\theta})} = \begin{bmatrix}
		I\\
		\widetilde{\Phi}_{c,\theta}\\
		\vdots\\
		\widetilde{\Phi}_{c,\theta}^{m_2-1}
	\end{bmatrix}.
\]


\subsection{Two-level reduction analysis}

Two-level reduction analysis \cite{VADobrev_etal_2017, BSouthworth_2018a} provides convergence bounds for two-level Parareal and MGRIT algorithms applied to linear time-stepping problems with fine- and coarse-scale time-stepping operators that can be diagonalized by the same set of eigenvectors. In particular, the analysis allows predictions of the convergence behavior for time-stepping problems arising from a method-of-lines approximation of a time-dependent PDE when considering a fixed spatial discretization in the time-grid hierarchy.
In Section \ref{ssub:ra:scalar:pdes}, we review the ideas behind the two-level reduction analysis for time-stepping problems arising from scalar PDEs and discuss how it can be combined with LFA in space. Section \ref{ssub:ra:pde:systems} is devoted to extending the analysis to time-stepping problems arising from systems of PDEs using finite-element discretizations for the discretization in space.

\subsubsection{Two-level reduction analysis for scalar PDEs}\label{ssub:ra:scalar:pdes}
We consider solving a time-stepping problem
\eqref{eq:fine:linear:system}, arising from a scalar PDE, by two-level
MGRIT. As above, let $\Phi$ and $\Phi_c$ denote the two time-stepping
operators on the fine time grid with $N_t$ time intervals, and on the
coarse time grid with $N_T=N_t/m$ time intervals,
respectively. Furthermore, assume that 
$N_x$ degrees of freedom are used for the discretization in space, so that $\Phi$ and $\Phi_c$ are matrices of size $N_x\times N_x$. Motivated by the reduction aspect of MGRIT, in contrast to analyzing full iteration matrices as in the SAMA approach, we consider the iteration matrices only on the coarse grid,
\begin{equation}\label{eq:cg:Parareal:mgrit}
	E_\Delta^F = I - A_c^{-1}A_S \quad\text{and}\quad E_\Delta^{FCF} = (I-A_c^{-1}A_S)(I-A_S),
\end{equation}
where $A_c$ and $A_S$ denote the coarse-grid operator and the Schur complement coarse-grid operator, respectively, introduced in Section \ref{sub:Parareal}. 

Two-level reduction analysis is based on the same analysis techniques
used in SAMA, but applied to the coarse-grid iteration matrix,
$E_\Delta$, instead of to the fine-grid iteration matrix, $E$, of the
algorithm. Furthermore, instead of using a Fourier ansatz in space,
the eigenvectors of the fine-grid time-stepping operator are used to
diagonalize the fine- and coarse-scale time integrators (under the
assumption that they are simultaneously unitarily diagonalizable) and, thus, the spatial blocks of the coarse-grid iteration matrix, $E_\Delta$. More precisely, using the unitary transformation, $X$, that diagonalizes the fine- and coarse-scale time integrators, $\Phi$ and $\Phi_c$, the blocks of the coarse-grid iteration matrix, $E_\Delta$ are diagonalized by computing $\mathcal{X}^{-1}E_\Delta\mathcal{X}$, where $\mathcal{X} = I_{N_T+1}\otimes X$. Similarly to SAMA, the resulting matrix is then permuted to obtain a block diagonal structure, with each block corresponding to the evolution of one eigenvector over time. Denoting the eigenvectors of the fine- and coarse-scale time integrators, $\Phi$ and $\Phi_c$, by $\left\{\mathbf{x}_n\right\}$, and the corresponding eigenvalues by $\left\{\lambda_n\right\}$ and $\left\{\mu_n\right\}$, respectively, for $n=1,\ldots,N_x$, we obtain
\[
	\mathcal{P}^{-1}\mathcal{X}^{-1}E_\Delta^F\mathcal{X}\mathcal{P} = \Diag(E_{\Delta,n}^F)_{n=1,\ldots,N_x}\quad\text{and}\quad\mathcal{P}^{-1}\mathcal{X}^{-1}E_\Delta^{FCF}\mathcal{X}\mathcal{P} = \Diag(E_{\Delta,n}^{FCF})_{n=1,\ldots,N_x},
\]
with
\begin{equation}
	E_{\Delta,n}^F = (\lambda_n^m - \mu_n)\begin{bmatrix}
		0\\
		1 & 0\\
		\mu_n & 1 & 0\\
		\vdots & \ddots & \ddots & \ddots\\
		\mu_n^{N_T-1} & \cdots & \mu_n & 1 & 0
	\end{bmatrix}\quad\text{and}\quad E_{\Delta,n}^{FCF} = (\lambda_n^m - \mu_n)\lambda_n^m\begin{bmatrix}
		0\\
		0 & 0\\
		1 & 0 & 0\\
		\mu_n & 1 & 0 & 0\\
		\vdots & \ddots & \ddots & \ddots & \ddots\\
		\mu_n^{N_T-2} & \cdots & \mu_n & 1 & 0 & 0
	\end{bmatrix}.
\end{equation}
Using the standard norm inequality
$\|E_\Delta\|_2\leq\sqrt{\|E_\Delta\|_1\|E_\Delta\|_\infty}$ for an
operator $E_\Delta$, the convergence behavior of a two-level method is
then predicted by calculating the bounds,
\[
\sigma_\text{RA}(E_\Delta^F) = \max_{n=1,\ldots,N_x}\sqrt{\|E_{\Delta,n}^F\|_1\|E_{\Delta,n}^F\|_\infty}\quad\text{and}\quad \sigma_\text{RA}(E_\Delta^{FCF}) = \max_{n=1,\ldots,N_x}\sqrt{\|E_{\Delta,n}^{FCF}\|_1\|E_{\Delta,n}^{FCF}\|_\infty}.,
\]
that have since been shown to be accurate to order
$\mathcal{O}(1/N_T)$\cite{BSouthworth_2018a}.
Assuming that $|\mu_n| \neq 1$ for all $n=1,\ldots,N_x$, we obtain\cite{VADobrev_etal_2017}
\[
	\|E_{\Delta,n}^F\|_1 = \|E_{\Delta,n}^F\|_\infty = |\lambda_n^m-\mu_n|\sum_{j=0}^{N_T-1}|\mu_n|^j = \left|\lambda_n^m - \mu_n\right| \frac{\left(1-|\mu_n|^{N_T}\right)}{(1-|\mu_n|)},
\]
and
\[
	\|E_{\Delta,n}^{FCF}\|_1 = \|E_{\Delta,n}^{FCF}\|_\infty = \left|\lambda_n^m - \mu_n\right| \frac{\left(1-|\mu_n|^{N_T-1}\right)}{(1-|\mu_n|)}|\lambda_n|^m,
\]
and, thus,
\begin{equation}\label{eq:err_red_factor_ra_scalar_case}
	\sigma_\text{RA}(E_\Delta^F) = \max_{n=1,\ldots,N_x}\left\{\left|\lambda_n^m - \mu_n\right| \frac{\left(1-|\mu_n|^{N_T}\right)}{(1-|\mu_n|)}\right\}\quad\text{and}\quad \sigma_\text{RA}(E_\Delta^{FCF}) = \max_{n=1,\ldots,N_x}\left\{\left|\lambda_n^m - \mu_n\right| \frac{\left(1-|\mu_n|^{N_T-1}\right)}{(1-|\mu_n|)}|\lambda_n|^m\right\}.
\end{equation}

\begin{remark}
	Note that the two-level reduction analysis requires solving an
        eigenvalue problem of the size of the degrees of freedom of
        the spatial discretization to compute the eigenvalues of the
        two time-integrators, which are assumed to be simultaneous
        unitarily diagonalizable. To avoid the high computational cost
        of this eigensolve for high spatial resolutions (and
        complications in the non-normal case), LFA in space can be applied, providing predictions of the spectra of the time-stepping operators. Thus, the eigenvalues $\left\{\lambda_n\right\}$ and $\left\{\mu_n\right\}$ of $\Phi$ and $\Phi_c$ can be replaced by the spatial Fourier symbols, $\left\{\widetilde{\Phi}_{\boldsymbol{\theta}}\right\}$ and $\left\{\widetilde{\Phi}_{c,\boldsymbol{\theta}}\right\}$, respectively, choosing $N_x$ frequency values for $\boldsymbol{\theta}$.
\end{remark}

\begin{remark}\label{rem:RApowers}
	Considering powers of the matrices $E_{\Delta,n}^F$ and $E_{\Delta,n}^{FCF}$, we can analogously derive bounds of the $L_2$-norm of powers of the coarse-grid iteration matrices \revision{and, thus, capture the exactness property of MGRIT.} For $k\geq 2$, we obtain
	\begin{center}
	\begin{tikzpicture}
		\node {$\left(E_{\Delta,n}^F\right)^k = (\lambda_n^m-\mu_n)^k\begin{bmatrix}
		0\\
		\vdots\\
		0\\
		1 & 0 & 0 & 0\\
		{k\choose 1}\mu_n & 1 & 0 & 0 & 0\\[.3\normalbaselineskip]
		{k+1\choose 2}\mu_n^2 & {k\choose 1}\mu_n & 1 & 0 & 0 & 0\\[.3\normalbaselineskip]
		{k+2\choose 3}\mu_n^3 & {k+1\choose 2}\mu_n^2 & \hspace*{.5em}{k\choose 1}\mu_n & 1 & 0 & 0 & 0\\[.5\normalbaselineskip]
		\vdots & & \ddots & \ddots & \ddots & \ddots &  & \ddots\\[.5\normalbaselineskip]
		{N_T-1\choose N_T-k}\mu_n^{N_T-k} & \cdots & & {k+1\choose 2}\mu_n^2 & {k\choose 1}\mu_n & 1 & 0 & \cdots & 0
		\end{bmatrix}$};
		\draw[line width=1pt,decorate,decoration={brace,amplitude=5pt}] (-2.7,1.2) -- (-2.7,2.4) node[left=5pt,midway] {\scriptsize$k$ rows};
		\draw[line width=1pt,decorate,decoration={brace,amplitude=3pt}] (5.65,-2.45) -- (4.65,-2.45) node[below=5pt,midway] {\scriptsize$k$ columns};
	\end{tikzpicture}
\end{center}
and
\begin{center}
	\begin{tikzpicture}
		\node {$\left(E_{\Delta,n}^{FCF}\right)^k = (\lambda_n^m-\mu_n)^k(\lambda_n^m)^k\begin{bmatrix}
		0\\
		\vdots\\
		0\\
		1 & 0 & 0 & 0\\
		{k\choose 1}\mu_n & 1 & 0 & 0 & 0\\[.3\normalbaselineskip]
		{k+1\choose 2}\mu_n^2 & {k\choose 1}\mu_n & 1 & 0 & 0 & 0\\[.3\normalbaselineskip]
		{k+2\choose 3}\mu_n^3 & {k+1\choose 2}\mu_n^2 & {k\choose 1}\mu_n & 1 & 0 & 0 & 0\\[.5\normalbaselineskip]
		\vdots & & \ddots & \ddots & \ddots & \ddots & & \ddots\\[.5\normalbaselineskip]
		{N_T-1-k\choose N_T-2k}\mu_n^{N_T-2k} & \cdots & & {k+1\choose 2}\mu_n^2 & {k\choose 1}\mu_n & 1 & 0 & \cdots & 0
		\end{bmatrix}$};
		\draw[line width=1pt,decorate,decoration={brace,amplitude=5pt}] (-2.35,1.15) -- (-2.35,2.5) node[left=5pt,midway] {\scriptsize$2k$ rows};
		\draw[line width=1pt,decorate,decoration={brace,amplitude=3pt}] (6.25,-2.45) -- (5.25,-2.45) node[below=5pt,midway] {\scriptsize$2k$ columns};
	\end{tikzpicture}
\end{center}
\end{remark}
Thus,
\[
	\left\|\left(E_{\Delta,n}^F\right)^k\right\|_1 = \left\|\left(E_{\Delta,n}^F\right)^k\right\|_\infty = |\lambda_n^m-\mu_n|^k\left[\sum_{j=0}^{N_T-k}{j+(k-1) \choose j}|\mu_n|^j\right]
\]
and
\[
	\left\|\left(E_{\Delta,n}^{FCF}\right)^k\right\|_1 = \left\|\left(E_{\Delta,n}^{FCF}\right)^k\right\|_\infty = |\lambda_n^m-\mu_n|^k|\lambda_n^m|^k\left[\sum_{j=0}^{N_T-2k}{j+(k-1) \choose j}|\mu_n|^j\right].
\]
Using the definition of the binomial coefficient, for $k\geq 2$ we obtain
\begin{equation}\label{eq:err_red_factor_ra_F}
	\sigma_\text{RA}\left(\left(E_\Delta^F\right)^k\right) = \max_{n=1,\ldots,N_x}\left\{|\lambda_n^m-\mu_n|^k\frac{1}{(k-1)!}\left[\sum_{j=0}^{N_T-k}\left(\prod_{i=1}^{k-1}(j+i)\right)|\mu_n|^j\right]\right\}
\end{equation}
and
\begin{equation}\label{eq:err_red_factor_ra_FCF}
	\sigma_\text{RA}\left(\left(E_\Delta^{FCF}\right)^k\right) = \max_{n=1,\ldots,N_x}\left\{|\lambda_n^m-\mu_n|^k|\lambda_n^m|^k\frac{1}{(k-1)!}\left[\sum_{j=0}^{N_T-2k}\left(\prod_{i=1}^{k-1}(j+i)\right)|\mu_n|^j \right]\right\}.
\end{equation}

\subsubsection{Systems of PDEs}\label{ssub:ra:pde:systems}
The assumption that a time-stepping operator can be diagonalized by a
unitary transformation does not necessarily hold true for
time-stepping operators arising from systems of PDEs, since such time
integrators can easily be non-symmetric. However, when considering a
two-level time-grid hierarchy and a fixed spatial discretization on
both grid levels, the assumption that the fine- and coarse-scale time
integrators, $\Phi$ and $\Phi_c$, can be simultaneously diagonalized
may still hold true for time-stepping problems arising from systems of
PDEs. One possible generalization of the two-level reduction analysis
to handle this case is to derive bounds for the coarse-grid iteration
matrix in a mass matrix-induced norm instead of in the
$L_2$-norm\cite{VADobrev_etal_2017, BSouthworth_2018a}. Here, to enable a comparison of the different analysis techniques for the systems case, we derive bounds of the $L_2$-norm of the coarse-grid iteration matrices.

We consider $q\times q$ block time-stepping operators, $\Phi$ and
$\Phi_c$, arising from a
semi-discretization in space of a system of $q$ scalar PDEs with $q$
unknown functions; note that in the case of a PDE system with vector
equations, we first have to break down the vectors into their
scalar components. Furthermore, let $E_\Delta$ be the coarse-grid
iteration matrix of a two-level method, given in Equation
\eqref{eq:cg:Parareal:mgrit}, defined using the block time
integrators, $\Phi$ and $\Phi_c$. To derive an upper bound on the
$L_2$-norm of $E_\Delta$, we first consider a Fourier ansatz in
space. More precisely, we use the Fourier matrix, $F$, of discretized
Fourier modes of a basis that accounts for the $q\times q$-coupling
within $\Phi$ and $\Phi_c$ and, possibly, different nodal
coordinates. Note that $F$ is a square matrix of dimension equal to
that of $\Phi$ and $\Phi_c$. We then reorder the transformed block
matrix, $\mathcal{F}^{-1}E_\Delta\mathcal{F}$, where $\mathcal{F} =
I_{N_T+1}\otimes F$, from $N_T+1\times N_T+1$ blocks of size
$qN_x\times qN_x$, where $qN_x=\dim(\Phi)$, to a block-diagonal matrix
with $N_x\times N_x$ blocks of size $q(N_T+1)\times q(N_T+1)$. For the
two-level algorithm with $F$-relaxation, for example, we obtain
\begin{equation}\label{eq:ra:st:blk:diag}
	\mathcal{P}^{-1}\mathcal{F}^{-1}E_\Delta^F\mathcal{F}\mathcal{P} = \Diag(\widetilde{E}^F_{\Delta,n})_{n=1,\ldots,N_x}
	\quad\text{with }
	\widetilde{E}_{\Delta,n}^F = (\widetilde{\Phi}_n^m - \widetilde{\Phi}_{c,n})\begin{bmatrix}
		0\\
		I & 0\\
		\widetilde{\Phi}_{c,n} & I & 0\\
		\vdots & \ddots & \ddots & \ddots\\
		\widetilde{\Phi}_{c,n}^{N_T-1} & \cdots & \widetilde{\Phi}_{c,n} & I & 0
	\end{bmatrix},
\end{equation}
and where $F^{-1}\Phi F = \Diag(\widetilde{\Phi}_1,\ldots,\widetilde{\Phi}_{N_x})$ and $F^{-1}\Phi_c F = \Diag(\widetilde{\Phi}_{c,1},\ldots,\widetilde{\Phi}_{c,N_x})$. Note that the spatial Fourier symbols, $\widetilde{\Phi}_n$ and $\widetilde{\Phi}_{c,n}$, $n=1,\ldots,N_x$, are dense block matrices of size $q\times q$. Therefore, for each $n = 1, \ldots, N_x$, we simultaneously diagonalize the Fourier symbols, $\widetilde{\Phi}_n$ and $\widetilde{\Phi}_{c,n}$, using the eigenvector matrix, $U_n$, of $\widetilde{\Phi}_n$,
\begin{equation}\label{eq:ra:eigvec:transformation}
	U_n^{-1}\widetilde{\Phi}_{n}U_n = \Diag(\lambda_{n,1},\ldots,\lambda_{n,q}), \quad U_n^{-1}\widetilde{\Phi}_{c,n}U_n = \Diag(\mu_{n,1},\ldots,\mu_{n,q}).
\end{equation}
Finally, we reorder the transformed block-Toeplitz matrix with
diagonal blocks,
$\mathcal{U}_n^{-1}\widetilde{E}_{\Delta,n}\mathcal{U}_n$, where
$\mathcal{U}_n=I_{N_T+1}\otimes U_n$, to a block-diagonal matrix with
Toeplitz blocks. For a two-level method with F-relaxation, we obtain
\[
	\mathcal{Q}^{-1}\mathcal{U}_n^{-1}\widetilde{E}_{\Delta,n}^F\mathcal{U}_n\mathcal{Q} = \Diag(E_{\Delta,n,l}^F)_{l=1,\ldots,q}
	\quad\text{with }
	E_{\Delta,n,l}^F = (\lambda_{n,l}^m - \mu_{n,l})\begin{bmatrix}
		0\\
		1 & 0\\
		\mu_{n,l} & 1 & 0\\
		\vdots & \ddots & \ddots & \ddots\\
		\mu_{n,l}^{N_T-1} & \cdots & \mu_{n,l} & 1 & 0
	\end{bmatrix}.
\]
Using the norm computations from Section \ref{ssub:ra:scalar:pdes}, we obtain the following

\begin{lemma}
	Let $\Phi$ and $\Phi_c$ be $q\times q$ block time-stepping
        operators, and let $\widetilde{\Phi}_n$ and
        $\widetilde{\Phi}_{c,n}$, $n = 1,\ldots,N_x$, where
        $qN_x=\dim(\Phi)=\dim(\Phi_c)$ be the corresponding spatial
        Fourier symbols. Furthermore, assume that for all
        $n=1,\ldots,N_x$, $\widetilde{\Phi}_n$ and
        $\widetilde{\Phi}_{c,n}$, can be simultaneously diagonalized
        with eigenvalues $\left\{\lambda_{n,l}\right\}_{l=1,\ldots,q}$
        and $\left\{\mu_{n,l}\right\}_{l=1,\ldots,q}$, respectively,
        and that $|\mu_{n,l}| \neq 1$ for all
        $n = 1,\ldots, N_x$, $l=1,\ldots,q$. Then, the $L_2$-norm of
        the coarse-grid iteration matrices, $E_\Delta^F$ and
        $E_\Delta^{FCF}$, of the two-level methods satisfy
	\begin{equation}\label{eq:err_red_factor_ra_F_system}
		\|E_\Delta^F\|_2 \leq \max_{n=1,\ldots,N_x}\left\{\kappa(U_n)\max_{l=1,\ldots,q}\left\{|\lambda_{n,l}^m-\mu_{n,l}|\frac{(1-|\mu_{n,l}|^{N_T})}{(1-|\mu_{n,l}|)}\right\}\right\}
	\end{equation}
	and
	\begin{equation}\label{eq:err_red_factor_ra_FCF_system}
		\|E_\Delta^{FCF}\|_2 \leq \max_{n=1,\ldots,N_x}\left\{\kappa(U_n)\max_{l=1,\ldots,q}\left\{|\lambda_{n,l}^m-\mu_{n,l}|~|\lambda_{n,l}|^m\frac{(1-|\mu_{n,l}|^{N_T-1})}{(1-|\mu_{n,l}|)}\right\}\right\},
	\end{equation}
	where $\kappa(U_n)$ denotes the condition number of $U_n$.
\end{lemma}
\begin{proof}
	Let $F$ denote the Fourier matrix that transforms $\Phi$ and
        $\Phi_c$ into block-diagonal form with blocks given by the
        spatial Fourier symbols, $\left\{\widetilde{\Phi}_n\right\}$ and
        $\left\{\widetilde{\Phi}_{c,n}\right\}$. Furthermore, let
        $\mathcal{P}$ denote the permutation that reorders the
        Fourier-transformed block matrix,
        $\mathcal{F}^{-1}E_\Delta\mathcal{F}$, where $\mathcal{F} =
        I_{N_T+1}\otimes F$, into global block-diagonal structure with
        $N_x\times N_x$ Toeplitz blocks, $\widetilde{E}_{\Delta,n}$,
        given in Equation \eqref{eq:ra:st:blk:diag} for the two-level
        method with F-relaxation. Both applied transformations are unitary transformations and, thus, 
	\[
		\|E_\Delta\|_2 = \max_{n=1,\ldots,N_x}\|\widetilde{E}_{\Delta,n}\|_2.
	\]
	The bound for the two-level method with F-relaxation can then be derived as
	\begin{align*}
		\|E_\Delta^F\|_2 &= \max_{n=1,\ldots,N_x}\|\widetilde{E}_{\Delta,n}^F\|_2\\
		&\leq \max_{n=1,\ldots,N_x}\left\{\|\mathcal{U}_n\|_2~\|\mathcal{U}_n^{-1}\widetilde{E}_{\Delta,n}^F\mathcal{U}_n\|_2~\|\mathcal{U}_n^{-1}\|_2\right\} = \max_{n=1,\ldots,N_x}\left\{\kappa(U_n)\|\mathcal{U}_n^{-1}\widetilde{E}_{\Delta,n}^F\mathcal{U}_n\|_2\right\}\\
		&\leq \max_{n=1,\ldots,N_x}\left\{\kappa(U_n)\max_{l=1,\ldots,q}\left\{\|E_{\Delta,n,l}^F\|_2\right\}\right\},
	\end{align*}
	with $\mathcal{U}_n=I_{N_T+1}\otimes U_n$, where $U_n$ is the transformation defined in Equation \eqref{eq:ra:eigvec:transformation}, for $n=1,\ldots,N_x$, and $\mathcal{Q}^{-1}\mathcal{U}_n^{-1}\widetilde{E}_{\Delta,n}^F\mathcal{U}_n\mathcal{Q} = \Diag(E_{\Delta,n,l}^F)_{l=1,\ldots,q}$, with permutation $\mathcal{Q}$ that reorders  $\mathcal{U}_n^{-1}\widetilde{E}_{\Delta,n}^F\mathcal{U}_n$ into block-diagonal structure with Toeplitz-blocks, $E_{\Delta,n,l}^F$.
	The bound for the coarse-grid iteration matrix of the two-level method with $FCF$-relaxation is derived analogously.
\end{proof}
\begin{remark}
	When $|\mu_{n,l}| = 1$ for $n \in \{1,\ldots, N_x\}$, $l\in\{1,\ldots,q\}$, we obtain $\|E_{\Delta,n,l}^F\|_1 = \|E_{\Delta,n,l}^F\|_\infty = |\lambda_{n,l}-\mu_{n,l}|N_T$ and $\|E_{\Delta,n,l}^{FCF}\|_1 = \|E_{\Delta,n,l}^{FCF}\|_\infty = |\lambda_{n,l}-\mu_{n,l}|(N_T-1)|\lambda_{n,l}|^m$, and we can easily get analogous bounds to those in Equations \eqref{eq:err_red_factor_ra_F_system}-\eqref{eq:err_red_factor_ra_FCF_system}, and \eqref{eq:err_red_factor_ra_scalar_case}.
\end{remark}


\subsection{Blending SAMA and reduction analysis}\label{ssec:comparison}

The closeness of SAMA and the two-level reduction analysis motivates
applying ideas from SAMA in the two-level reduction analysis and vice
versa. In one direction, we can easily make use of the bound commonly
used in reduction analysis to improve the computability of the SAMA
prediction, in place of the singular value computation that was used
in prior work\cite{SFriedhoff_SMacLachlan_2015a}.  This yields
\begin{equation}\label{eq:err_red_factor_sama_bound}
	\sigma_\text{SAMA}(E^k) \leq \widetilde{\sigma}_\text{SAMA}(E^k) := \sup\left\{\sqrt{\left\|\left(B_{\boldsymbol{\theta}}^{(E)}\right)^k\right\|_1~\left\|\left(B_{\boldsymbol{\theta}}^{(E)}\right)^k\right\|_\infty} ~:~ \boldsymbol{\theta}\in(-\pi,\pi]^2\right\},\quad\text{for }k\geq 1,
\end{equation}
where $E$ is the iteration matrix using either F- or FCF-relaxation.
Similarly, we can compute the SAMA bound using only the coarse
representation of the propagators, with $E_\Delta^F$ or
$E_\Delta^{FCF}$.  Both of these are explored below, in Section
\ref{sub:comparison}.

In the other direction, we can investigate the differences between the
predictions made by reduction analysis on the coarse
grid\cite{VADobrev_etal_2017} or consider the full fine-grid iteration
matrix as in SAMA.  Here, it is easy to derive the corresponding
bounds, of
	\[
	\widetilde{\sigma}_\text{RA}(E^F) = \max_{n=1,\ldots,N_x}\sqrt{\|E_n^F\|_1\|E_n^F\|_\infty}\quad\text{and}\quad \widetilde{\sigma}_\text{RA}(E^{FCF}) = \max_{n=1,\ldots,N_x}\sqrt{\|E_n^{FCF}\|_1\|E_n^{FCF}\|_\infty}
	\]
	with
	\begin{align*}
	\|E_n^F\|_1 &= |\lambda_n^m-\mu_n|\left(|\mu_n^{N_T-1}| + \sum_{i=0}^{m-1}|\lambda_\omega^i|\cdot\sum_{j=0}^{N_T-2}|\mu_\omega^j|\right), &&\quad
	\hspace*{1.1em}\|E_n^F\|_\infty = \|E_{\Delta,n}^F\|_\infty \hspace*{.65em}= |\lambda_n^m - \mu_n|\sum_{j=0}^{N_T-1}|\mu_n^j|\\[.5\normalbaselineskip]
	\|E_n^{FCF}\|_1 &= |\lambda_n^m-\mu_n||\lambda_n^m|\left(|\mu_n^{N_T-2}| + \sum_{i=0}^{m-1}|\lambda_n^i|\cdot\sum_{j=0}^{N_T-3}|\mu_n^j|\right), &&\quad
	\|E_n^{FCF}\|_\infty = \|E_{\Delta,n}^{FCF}\|_\infty = |\lambda_n^m - \mu_n||\lambda_n^m|\sum_{j=0}^{N_T-2}|\mu_n^j|.
        \end{align*}
        For multiple iterations, $k\geq 2$, these bounds become
	\[
	\widetilde{\sigma}_\text{RA}\left(\left(E^F\right)^k\right) = \max_{n=1,\ldots,N_x}\sqrt{\left\|\left(E_n^F\right)^k\right\|_1~\left\|\left(E_n^F\right)^k\right\|_\infty},\quad \widetilde{\sigma}_\text{RA}\left(\left(E^{FCF}\right)^k\right) = \max_{n=1,\ldots,N_x}\sqrt{\left\|\left(E_n^{FCF}\right)^k\right\|_1~\left\|\left(E_n^{FCF}\right)^k\right\|_\infty}
	\]
	with
	\begin{align*}
	\left\|\left(E_n^F\right)^k\right\|_1 &= |\lambda_n^m-\mu_n|^k\left[{N_T-1 \choose N_T-k}|\mu_n|^{N_T-k}+\sum_{l=0}^{m-1}|\lambda_n^{kl}|\left(\frac{1}{(k-1)!}\sum_{j=0}^{N_T-k-1}\left(\prod_{i=1}^{k-1}(j+i)\right)|\mu_n|^j\right)\right],\\[.5\normalbaselineskip]
	\left\|\left(E_n^F\right)^k\right\|_\infty &= \left\|\left(E_{\Delta,n}^F\right)^k\right\|_\infty = |\lambda_n^m-\mu_n|^k\frac{1}{(k-1)!}\left[\sum_{j=0}^{N_T-k}\left(\prod_{i=1}^{k-1}(j+i)\right)|\mu_n|^j\right],\\[.5\normalbaselineskip]
	\left\|\left(E_n^{FCF}\right)^k\right\|_1 &= |\lambda_n^m-\mu_n|^k~|\lambda_n^m|^k\left[{N_T-k-1\choose N_T-2k}|\mu_n^{N_T-2k}|+\sum_{l=0}^{m-1}|\lambda_n^{kl}|\left(\frac{1}{(k-1)!}\sum_{j=0}^{N_T-2k-1}\left(\prod_{i=1}^{k-1}(j+i)\right)|\mu_n|^j\right)\right],\\[.5\normalbaselineskip]
	\left\|\left(E_n^{FCF}\right)^k\right\|_\infty &= \left\|\left(E_{\Delta,n}^{FCF}\right)^k\right\|_\infty = |\lambda_n^m-\mu_n|^k~|\lambda_n^m|^k\frac{1}{(k-1)!}\left[\sum_{j=0}^{N_T-2k}\left(\prod_{i=1}^{k-1}(j+i)\right)|\mu_n|^j \right].
\end{align*}

\subsection{\revision{Differences, advantages, and disadvantages of the three tools}}
\revision{
The three mode analysis tools presented above are all applied to the MGRIT iteration matrices, $E$, defined in Equations \eqref{eq:Parareal:it:matrix} and \eqref{eq:two:level:mgrit}, or to their coarse-level versions, $E_\Delta$. Each predicts the worst-case error reduction in the $k$-th iteration, $k\geq 1$, by providing an upper bound or an approximation of $\|E^k\|_2$ or of $\|E_\Delta^k\|_2$. The analysis of an iteration matrix, $E$, by means of the three tools can be summarized as follows.
\begin{enumerate}
	\item Compute the spatial Fourier symbols of $\widetilde{\Phi}_{\boldsymbol{\theta}}$ and $\widetilde{\Phi}_{c,\boldsymbol{\theta}}$ (each of size $q\times q$) of the time integrators on the fine and on the coarse grids for $N_{\boldsymbol{\theta}}$ discrete values of $\boldsymbol{\theta}\in(-\pi,\pi]^d$, $d = 1,2$.
	\item \textbf{LFA}\\[-1.5\baselineskip]
	\begin{enumerate}
		\item[i)] Compute the space-time Fourier symbol $\widehat{E}(\boldsymbol{\theta},\omega^{(0)})$ (of size $mq\times mq$) for $N_\omega$ discrete values of $\omega^{(0)}\in(-\pi/m,\pi/m]$.
		\item[ii)] For each $k = 1, 2, \ldots$\\
		\hspace*{.02\linewidth}\parbox{.98\linewidth}{Compute $\|(\widehat{E}(\boldsymbol{\theta},\omega^{(0)}))^k\|_2$ for all $(\boldsymbol{\theta},\omega^{(0)})$-tuples and maximize over the $N_{\boldsymbol{\theta}}N_\omega$ values.}
	\end{enumerate}
	\textbf{SAMA}\\[-1.5\baselineskip]
	\begin{enumerate}
		\item[i)] Compute the SAMA block $B_{\boldsymbol{\theta}}^{(E)}$ (of size $(N_t+1)q\times(N_t+1)q)$.
		\item[ii)] For each $k = 1, 2, \ldots$\\
		\hspace*{.02\linewidth}\parbox{.98\linewidth}{Compute $\|(B_{\boldsymbol{\theta}}^{(E)})^k\|_2$ or $\sqrt{\|(B_{\boldsymbol{\theta}}^{(E)})^k\|_1\|(B_{\boldsymbol{\theta}}^{(E)})^k\|_\infty}$ for all $\boldsymbol{\theta}$ and maximize over the $N_{\boldsymbol{\theta}}$ values.}
	\end{enumerate}
	\textbf{RA}\\[-1.5\baselineskip]
	\begin{enumerate}
		\item[i)] Simultaneously diagonalize $\widetilde{\Phi}_{\boldsymbol{\theta}}$ and $\widetilde{\Phi}_{c,\boldsymbol{\theta}}$ (each of size $q\times q$) to obtain eigenvalues $\{\lambda_{\boldsymbol{\theta},l}\}_{l=1,\ldots,q}$ and $\{\mu_{\boldsymbol{\theta},l}\}_{l=1,\ldots,q}$.
		\item[ii)] For each $k = 1, 2, \ldots$\\
		\hspace*{.02\linewidth}\parbox{.98\linewidth}{Compute bound of $\|E_\Delta^k\|_2$ or $\|E^k\|_2$ by using $\{\lambda_{\boldsymbol{\theta},l}\}$ and $\{\mu_{\boldsymbol{\theta},l}\}$ in formulas for $\|E_\Delta^k\|_1$ and $\|E_\Delta^k\|_\infty$ or for $\|E^k\|_1$ and $\|E^k\|_\infty$ for all $\boldsymbol{\theta}$ and $l$ and maximize over the $qN_{\boldsymbol{\theta}}$ values.}
	\end{enumerate}
\end{enumerate}

The most important advantages and disadvantages of the three mode analysis tools are as follows:\\[-.5\baselineskip]

\textbf{LFA}\\[-1.5\baselineskip]
\begin{itemize}\itemsep0em
	\item[$+$] relatively low computational cost
	\item[$-$] good predictions only for large numbers of time steps, since finiteness of time interval is ignored; exactness property not captured
\end{itemize}

\textbf{SAMA}\\[-1.5\baselineskip]
\begin{itemize}\itemsep0em
	\item[$+$] accurate predictions for scalar and systems cases; takes finiteness of time interval into account; captures exactness property
	\item[$-$] high computational cost, especially for large numbers of time steps
\end{itemize}

\textbf{RA}\\[-1.5\baselineskip]
\begin{itemize}\itemsep0em
	\item[$+$] ease of computing a bound for small and large numbers of time steps; takes finiteness of time interval into account; captures exactness property when considering powers of iteration matrices as in Remark \ref{rem:RApowers}
	\item[$-$] bound relies on the assumption of unitary diagonalizability
\end{itemize}

}

%
%
\section{Numerical results}\label{sec:num_results}

Numerical experiments presented in this section are organized in two parts: first, in Section \ref{sub:comparison}, we compare and contrast the three mode analysis tools of Section \ref{sec:analysis_tools}. Secondly, Sections \ref{sub:investigating_adv} and \ref{sub:investigating_elasticity} are devoted to using appropriate tools for gaining insight into effects of model and algorithmic parameters on the convergence behavior of the methods presented in Section \ref{sec:pint_methods}.

\subsection{Comparing the three mode analysis tools}\label{sub:comparison}
The three mode analysis tools differ, in particular, in the treatment of the time dimension. While space-time LFA uses a Fourier ansatz in space and time, SAMA couples LFA in space with algebraic computations in time, and the two-level reduction analysis considers error propagation only on the coarse time grid. In this section, we compare and contrast the predictions of the three methods for the two hyperbolic model problems, linear advection in one dimension and incompressible linear elasticity in two dimensions. Here, both problems are discretized on a space-time mesh of size $64\times 64$ or $64^2\times 64$, respectively, with spatial mesh size ${\Delta x}=1/2$ and time-step size $\Delta t=1/10$. For the linear advection problem, the flow speed is chosen to be $c=1$ and the material parameters of the elasticity problem are chosen to be $\mu = \rho = 1$; the influence of these model parameters on the convergence behavior is considered in Sections \ref{sub:adv_influence_parameters} and \ref{sub:investigating_elasticity}, respectively.

In all of the mode analysis tools, we consider a Fourier ansatz in
space. Note that for the two-level reduction analysis, this ansatz
saves the cost of a computationally expensive solution of an 
eigenvalue problem, and (for all cases) is equivalent to rigorously 
analysing the spatial problems with periodic boundary conditions. 
The results presented here sample the Fourier frequency, 
$\theta\in(-\pi,\pi]$, or the Fourier frequency pair,
  $\boldsymbol{\theta}=(\theta_x,\theta_y)\in(-\pi,\pi]^2$,
    respectively, on a discrete mesh with spacing $h_\theta =
    \pi/32$. For the space-time LFA predictions, the temporal Fourier
    base frequency, $\omega_0\in(-\pi/m,\pi/m]$, is additionally
      sampled on a discrete mesh with spacing $h_\omega = \pi/32$. The
      impact of finer Fourier frequency meshes was negligible in the
      examples considered here. Figure \ref{fig:stlfa_sama_ra} shows
      the error reduction factors, $\sigma_\text{LFA}$,
      $\sigma_\text{SAMA}$, and $\sigma_\text{RA}$, defined in
      Equations \eqref{eq:err_red_factor_lfa},
      \eqref{eq:err_red_factor_sama}, and
      \eqref{eq:err_red_factor_ra_F}-\eqref{eq:err_red_factor_ra_FCF},
      respectively, for the first 10 two-level iterations with F- and
      FCF-relaxation with factor-2 temporal coarsening applied to
      linear advection and linear elasticity. Note that
      the error reduction factors, $\sigma_\text{LFA}$ and
      $\sigma_\text{SAMA}$, of space-time LFA and of SAMA are based on
      measuring the $L_2$-norm of powers of the full iteration
      matrices, i.\,e., $\sigma_\text{LFA}$ and
      $\sigma_\text{SAMA}$ are upper bounds for the worst-case error reduction 
      at all time points, whereas the error reduction factor,
      $\sigma_\text{RA}$, of the two-level reduction analysis provides
      bounds on the $L_2$-norm of powers of the coarse-grid iteration
      matrices, i.\,e., upper bounds for the worst-case error reduction only at 
      C-points. Results show that LFA only predicts the initial
      convergence behavior, while both SAMA and RA also enable good
      predictions of short-term and long-term convergence behavior, covering the
      superlinear convergence of the methods, the effect of non-normality in early iterations, and the exactness property of the algorithms \revision{(not shown here)}. Furthermore, the results demonstrate the difference in predictivity of SAMA and RA for the scalar case and the systems case. While in the scalar case of linear advection, predictions of SAMA and RA are close, in the systems case of linear elasticity, the prediction of RA is pessimistic due to the necessary condition number factor involved in the bound of the $L_2$-norm of the operators.

\begin{figure}[ht]
	\centerline{\includegraphics[width=.48\textwidth]{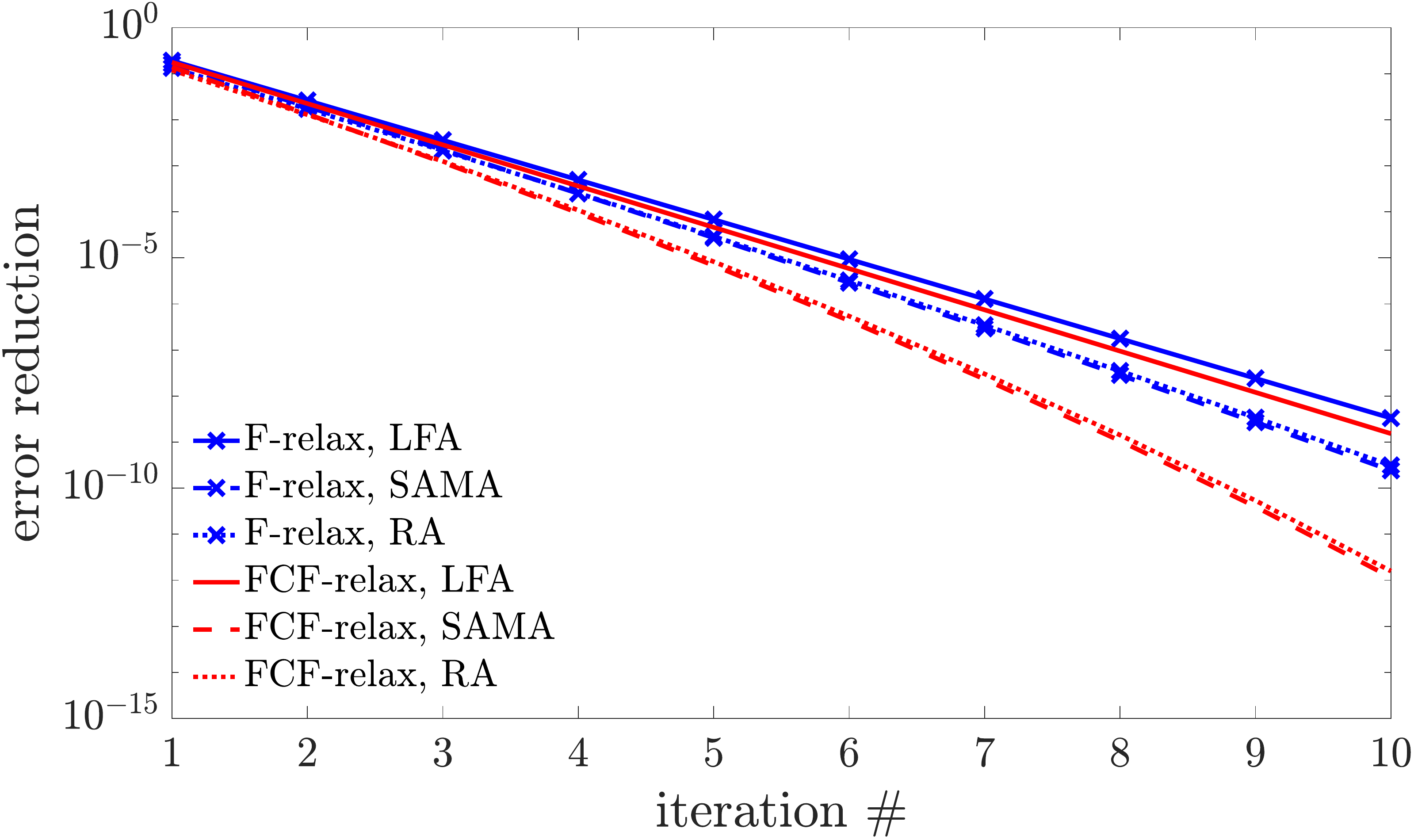}\quad\includegraphics[width=.48\textwidth]{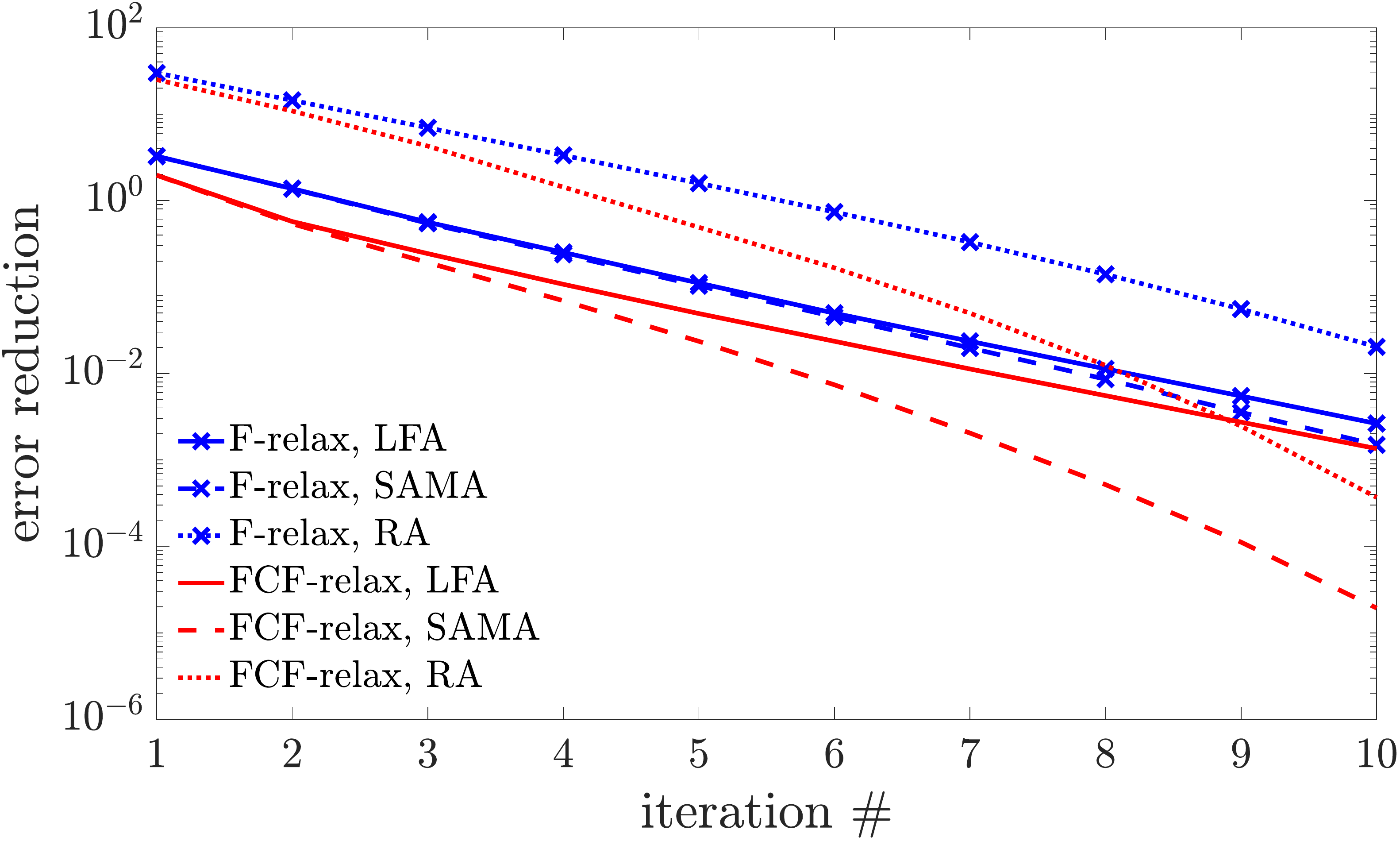}}
	\caption{Error reduction factors predicted by space-time LFA,
          SAMA, and RA for the first 10 two-level iterations with F-
          and FCF-relaxation and factor-2 coarsening applied to the linear advection problem (left) and to the linear elasticity problem (right).\label{fig:stlfa_sama_ra}}
\end{figure}

\revision{Figure \ref{fig:sama_exp} compares error reduction factors predicted by SAMA with experimentally measured error reduction factors for the first 10 two-level iterations with F- and FCF-relaxation with factor-2 temporal coarsening applied to linear elasticity, i.\,e., for the setting considered in the left plot of Figure \ref{fig:stlfa_sama_ra}. As SAMA predicts the worst-case error reduction, MGRIT convergence is measured for three initial conditions, consistenting of either exclusively low or high frequency modes, or for a linear combination of low and high frequency modes. More precisely, the initial condition is chosen as $u_0(x) = \varphi(-\theta,x)+\varphi(\theta,x) = 2\cos(\theta x)$, with $\theta = \pi/16$ or $\theta = 5\pi/8$ and $u_0(x) = \varphi(-\pi/8,x)+\varphi(\pi/8,x)+\varphi(-15\pi/16,x)+\varphi(15\pi/16,x) = 2\cos(\pi x/8) + 2\cos(15\pi x/16)$. Results demonstrate that SAMA predictions provide a good upper bound for the error reduction in each iteration. Note that when considering a random initial space-time guess, convergence is essentially the same for all initial conditions, since the initial space-time error consists of all frequencies.  In contrast, convergence is much faster for a low-frequency initial condition when considering a zero initial guess, since only low-frequency errors are present.}

\begin{figure}[ht]
	\centerline{\includegraphics[width=.48\textwidth]{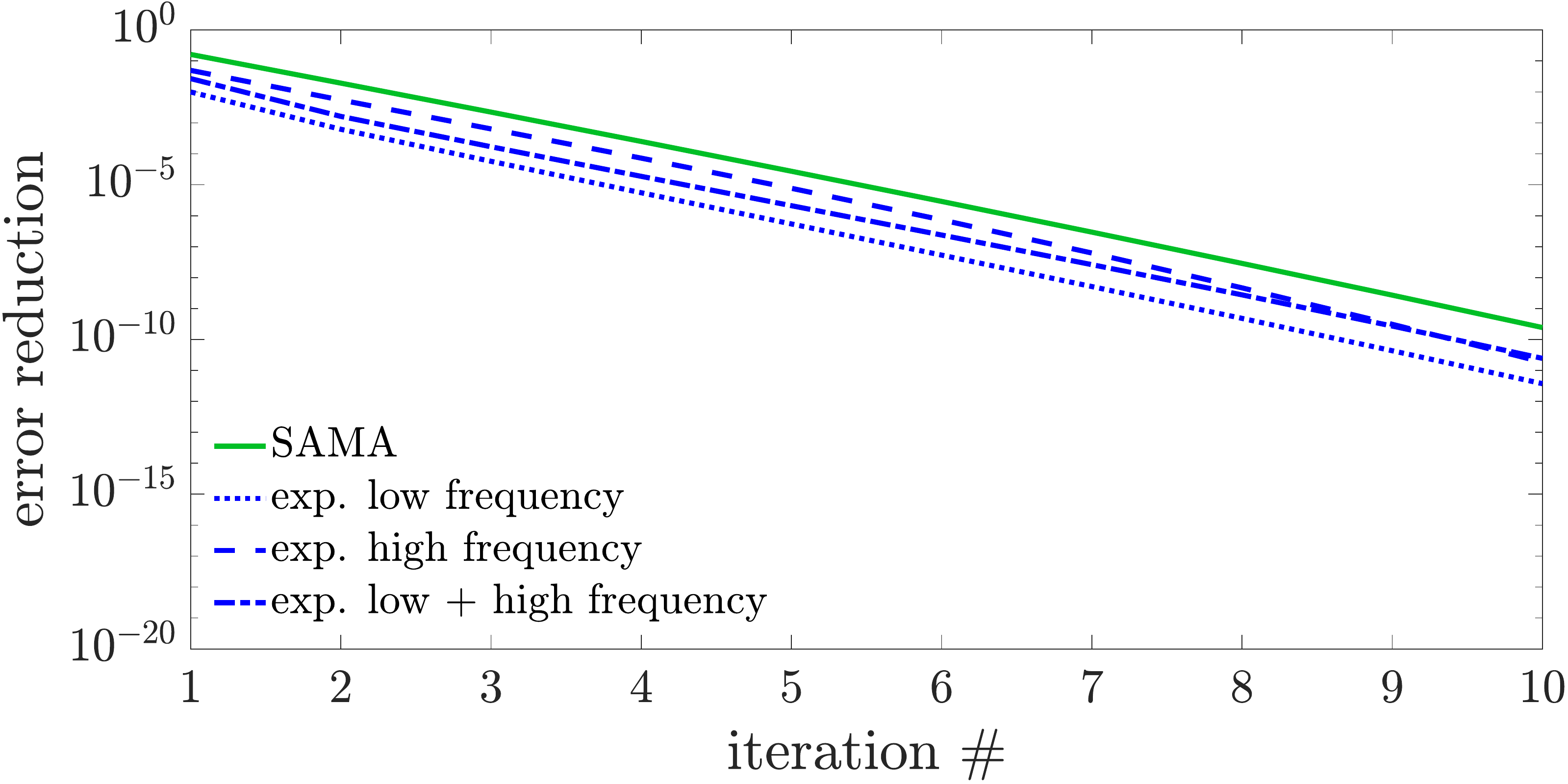}\quad\includegraphics[width=.48\textwidth]{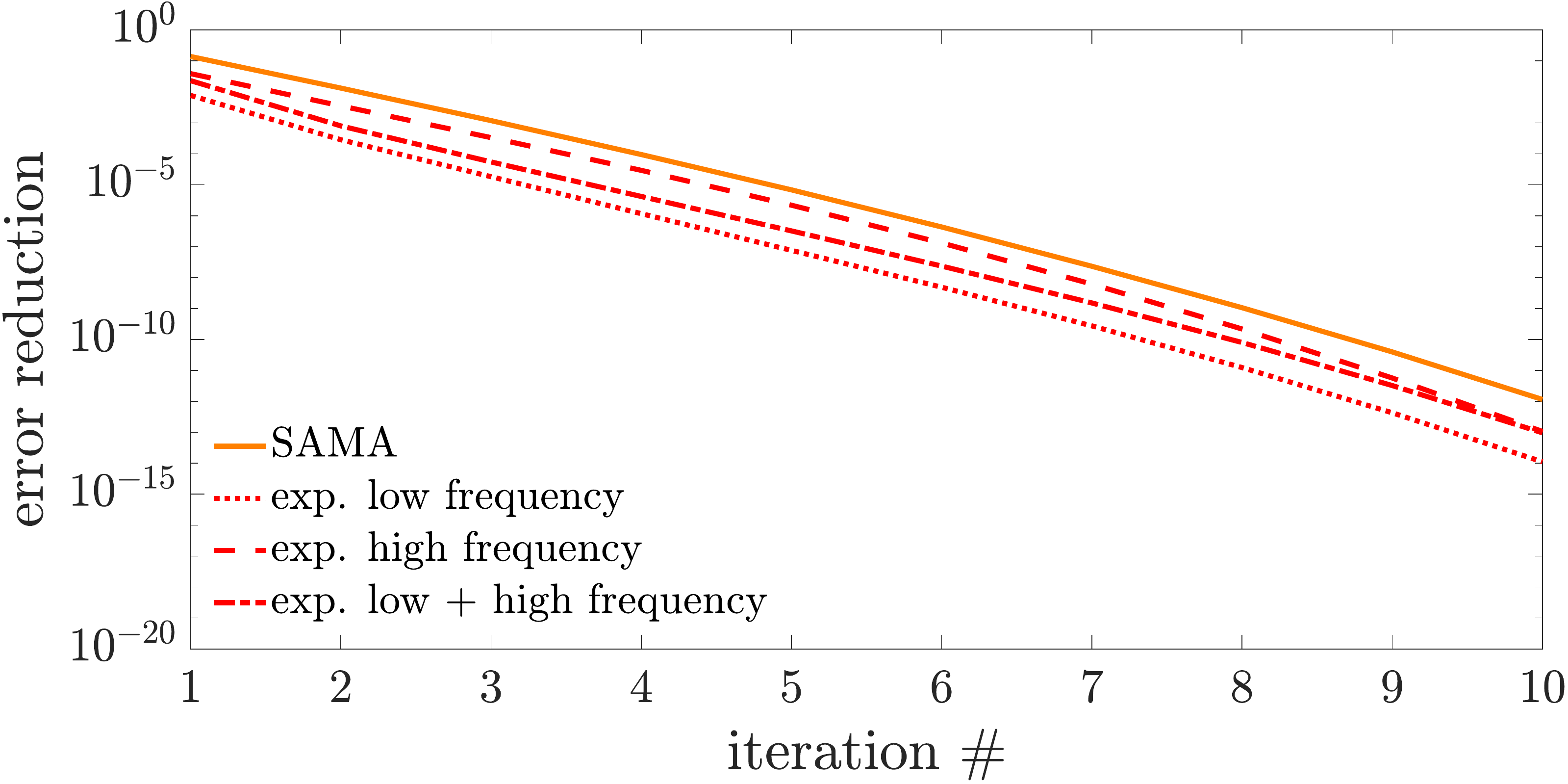}}
	\centerline{\includegraphics[width=.48\textwidth]{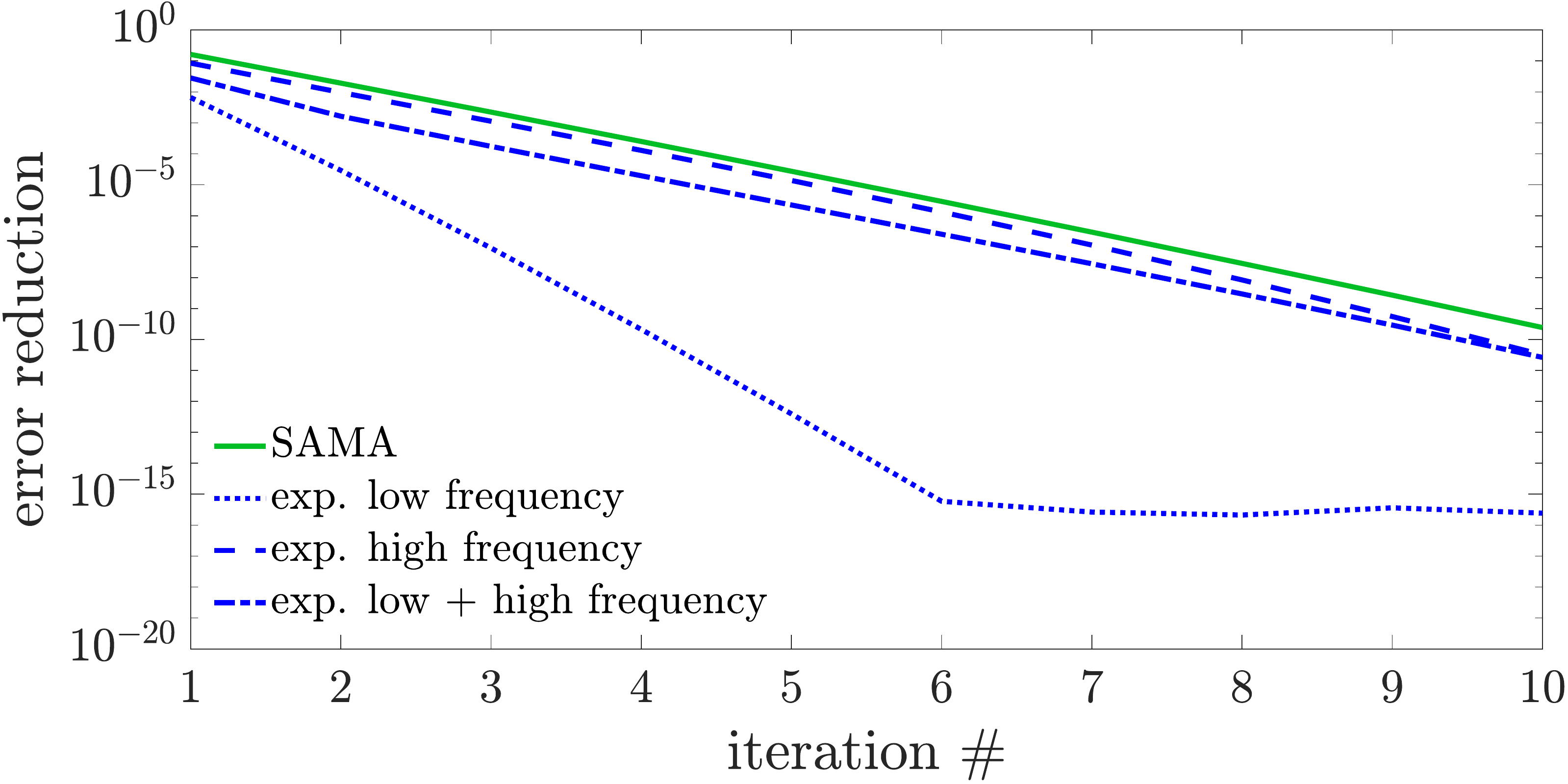}\quad\includegraphics[width=.48\textwidth]{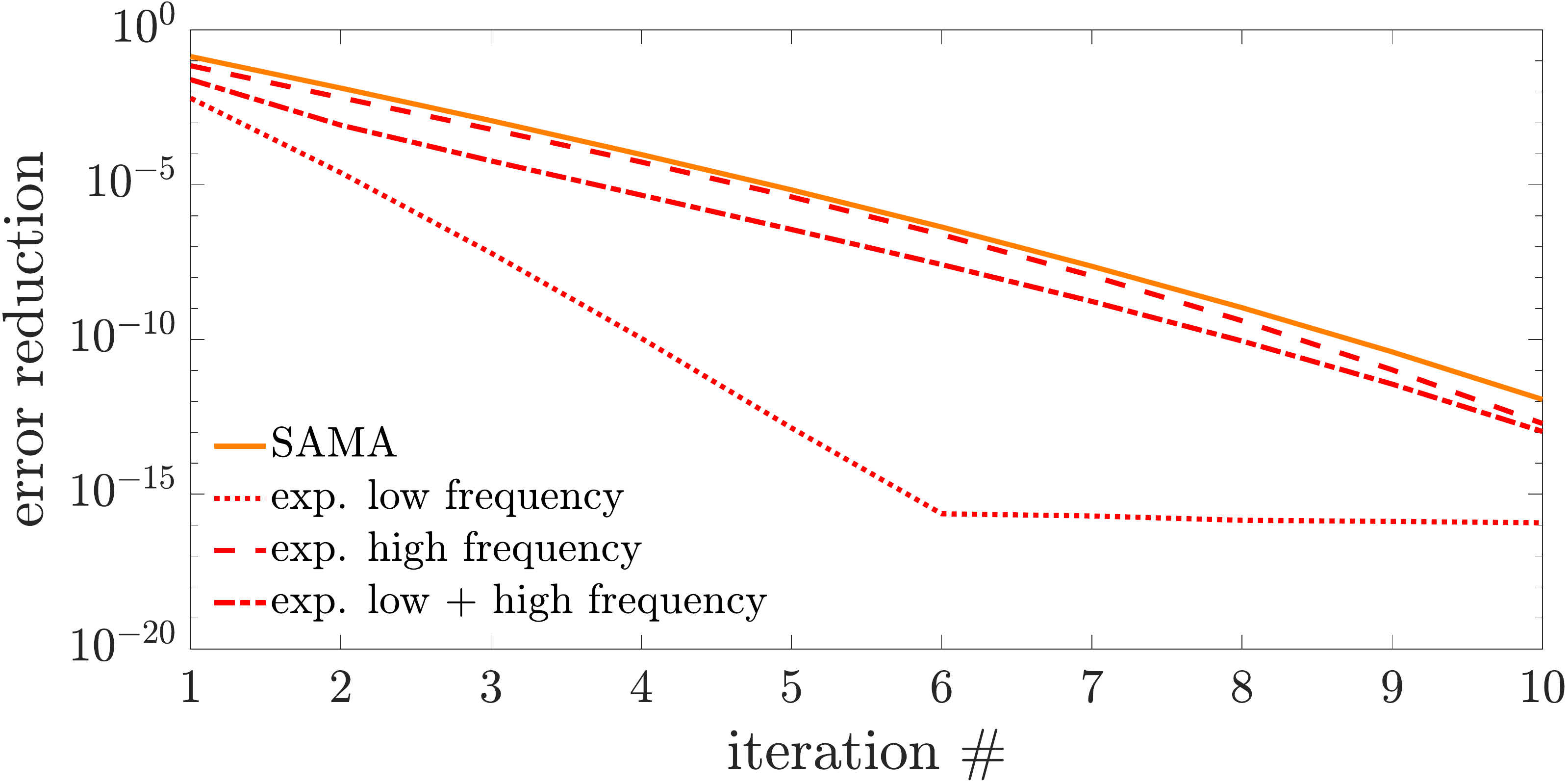}}
	\caption{\revision{Error reduction factors predicted by SAMA and measured error reduction factors for the first 10 two-level iterations with F- (left)
          and FCF-relaxation (right) and factor-2 coarsening applied to the linear advection problem with random (top row) or zero (bottom row) initial space-time guess. The initial condition is chosen as $u_0(x) = 2\cos(\pi x/16)$ (low frequency), $u_0(x) = 2\cos(5\pi x/8)$ (high frequency), or $u_0(x) = 2\cos(\pi x/8) + 2\cos(15\pi x/16)$ (low + high frequency).}\label{fig:sama_exp}}
\end{figure}


\subsubsection{Combining SAMA and reduction analysis techniques}
The above results show that, in the systems setting, SAMA provides better predictions than RA. However, when computing exact Euclidean operator norms of matrices with size equal to the number of time steps multiplied by the dimension of the spatial LFA symbol, SAMA becomes prohibitively expensive for large numbers of time steps. On the other hand, computing bounds instead of exact Euclidean operator norms makes the reduction analysis computationally tractable. Moreover, considering the iteration matrices only on the coarse grid instead of the full iteration matrices further reduces the cost of RA. Therefore, in this section, we explore combining the SAMA and RA approaches. Again, both model problems are discretized on a space-time mesh of size $64\times 64$ or $64^2\times 64$, respectively, with spatial mesh size ${\Delta x} = 1/2$ and time-step size $\Delta t = 1/10$. Furthermore, spatial Fourier frequencies are sampled on a discrete mesh with spacing $h_\theta = \pi/32$ in both dimensions.

In Figures \ref{fig:sama_ra_advection} and
\ref{fig:sama_ra_elasticity}, we explore the effects of considering
coarse-grid iteration matrices or full iteration matrices (labeled
C-pts or full in the legends) and of computing bounds or exact
Euclidean operator norms (labeled 2-norm bound or 2-norm in the
legends of the figures). More precisely, in addition to the error
reduction factors, $\sigma_\text{SAMA}$ and $\sigma_\text{RA}$ (solid
lines in the figures), also considered in Section
\ref{sub:comparison}, Figures \ref{fig:sama_ra_advection} and
\ref{fig:sama_ra_elasticity} show various variants of SAMA- and
RA-predicted error reduction factors. In particular, for RA, we
consider using full iteration matrices instead of only coarse-grid
iteration matrices. For SAMA, three additional error reduction factors
are shown: one error reduction factor based on full iteration
matrices, but computing bounds of the Euclidean operator norms, and
two error reduction factors based on coarse-grid iteration matrices,
computing either exact Euclidean operator norms or their
bounds. Results using F- and FCF-relaxation show similar relationships
between the different variants of error reduction factors. The
difference between considering coarse-grid iteration matrices or full
iteration matrices is at most a factor of about 1.4 in all
cases. Since for a temporal coarsening factor of $m$ this factor is given by
$\sqrt{m}$\cite{Hessenthaler_etal_2018a}, for small coarsening factors, considering only
coarse-grid iteration matrices gives good estimates of actual error reduction
factors that may lie in between the full and the C-point bounds or below both bounds. Furthermore, 
when using the bound on the 2-norm instead of computing the exact 2-norm, for the advection 
problem, we observe that error reduction
factors increase by a factor of at most 1.7 for linear advection, and by a factor of at most
3.7  in the case of linear elasticity. Thus, for
the analysis of MGRIT applied to the linear
elasticity probem, considering coarse-grid iteration matrices and
using the bound on the 2-norm for the SAMA prediction gives a
practical improvement over RA, with predicted error reduction factors
being a factor of about five or ten, respectively, smaller than those
of RA, and not relying on the (unsatisfied) assumption of unitary diagonalizability.

\begin{figure}[h!t]
	\centerline{\includegraphics[width=.48\textwidth]{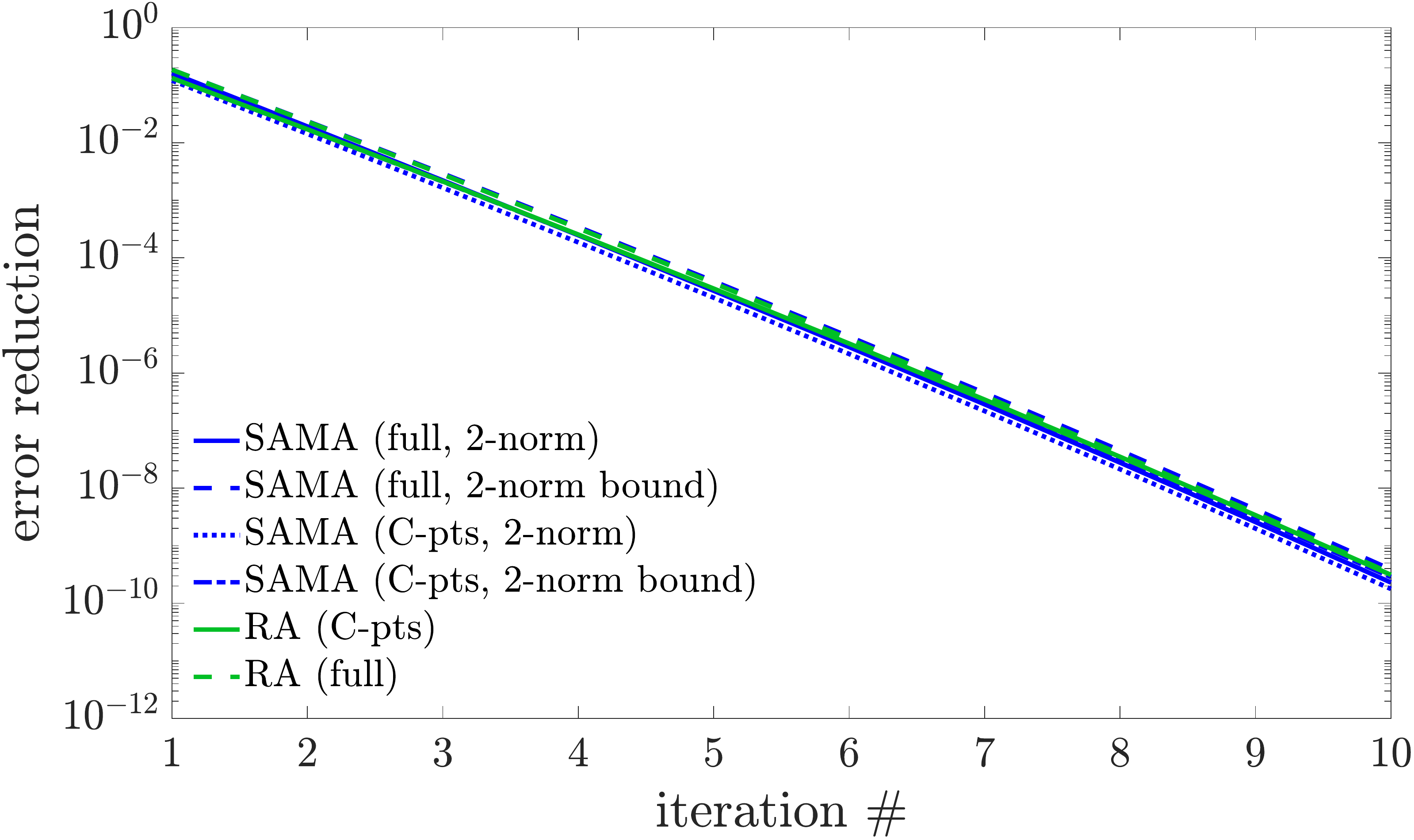}\quad\includegraphics[width=.48\textwidth]{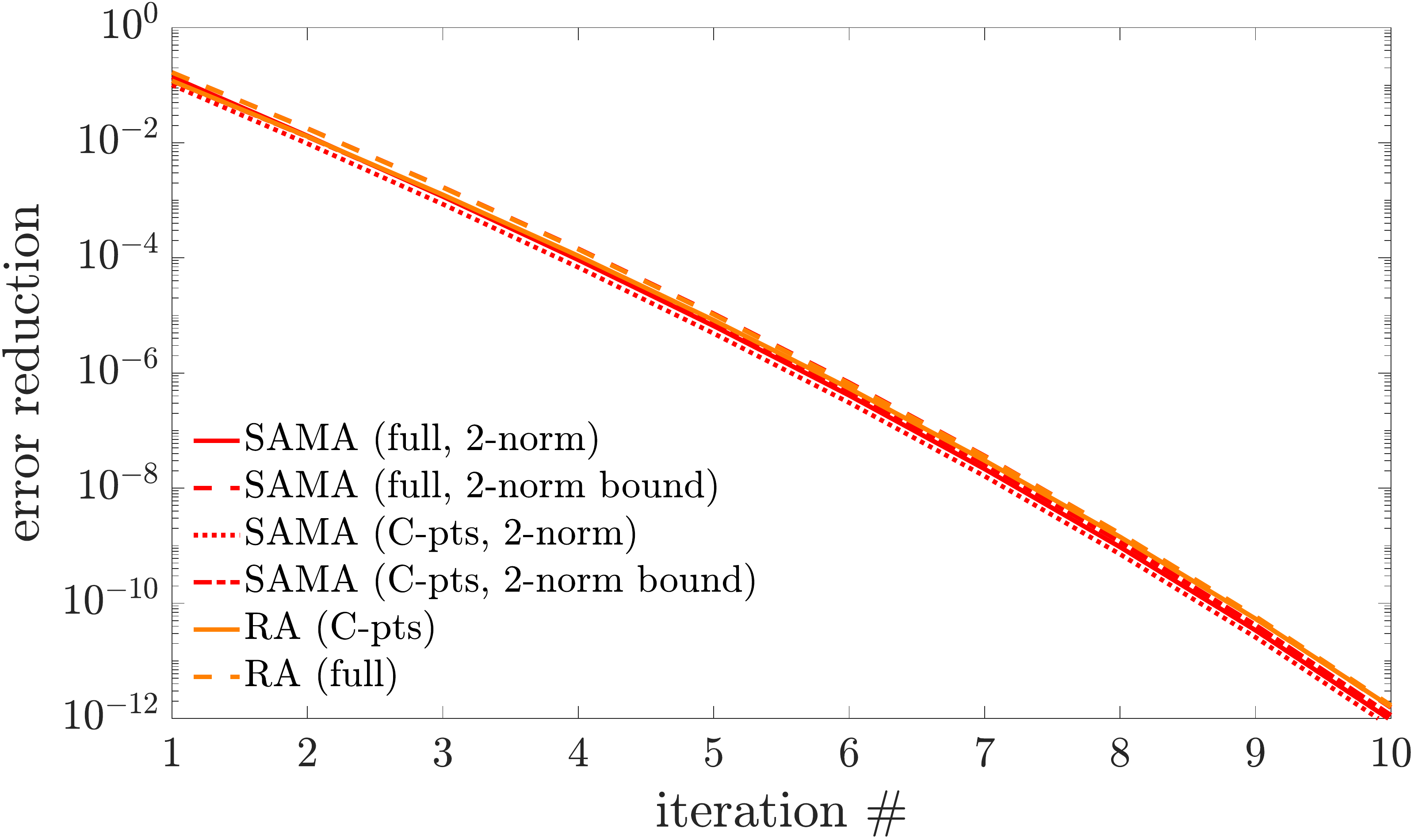}}
	\caption{Linear advection: Several variants of SAMA- and
          RA-predicted error reduction factors for the first 10
          two-level iterations with F- (left) and FCF-relaxation (right) and factor-2 temporal coarsening.\label{fig:sama_ra_advection}}
\end{figure}

\begin{figure}[h!t]
	\centerline{\includegraphics[width=.48\textwidth]{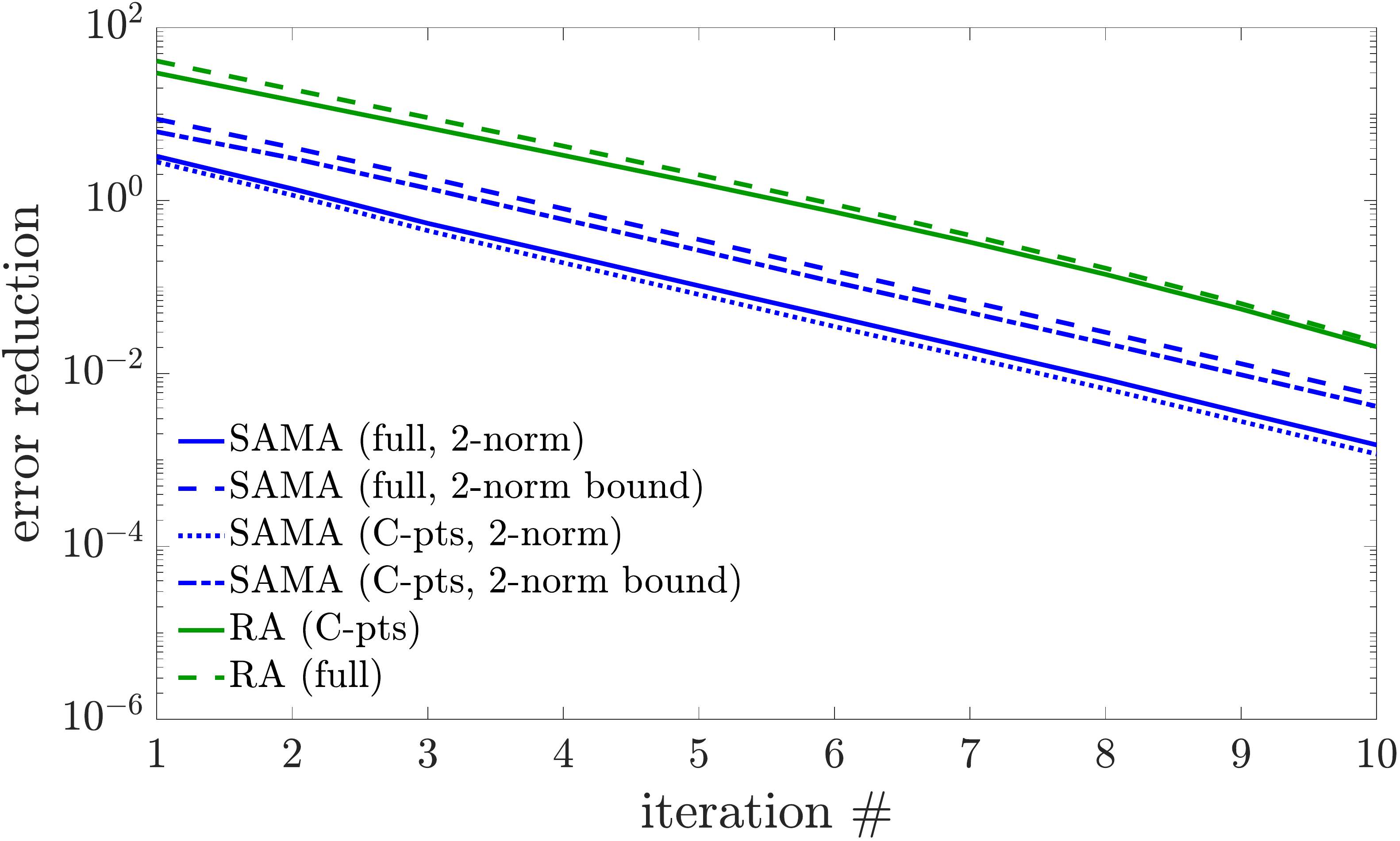}\quad\includegraphics[width=.48\textwidth]{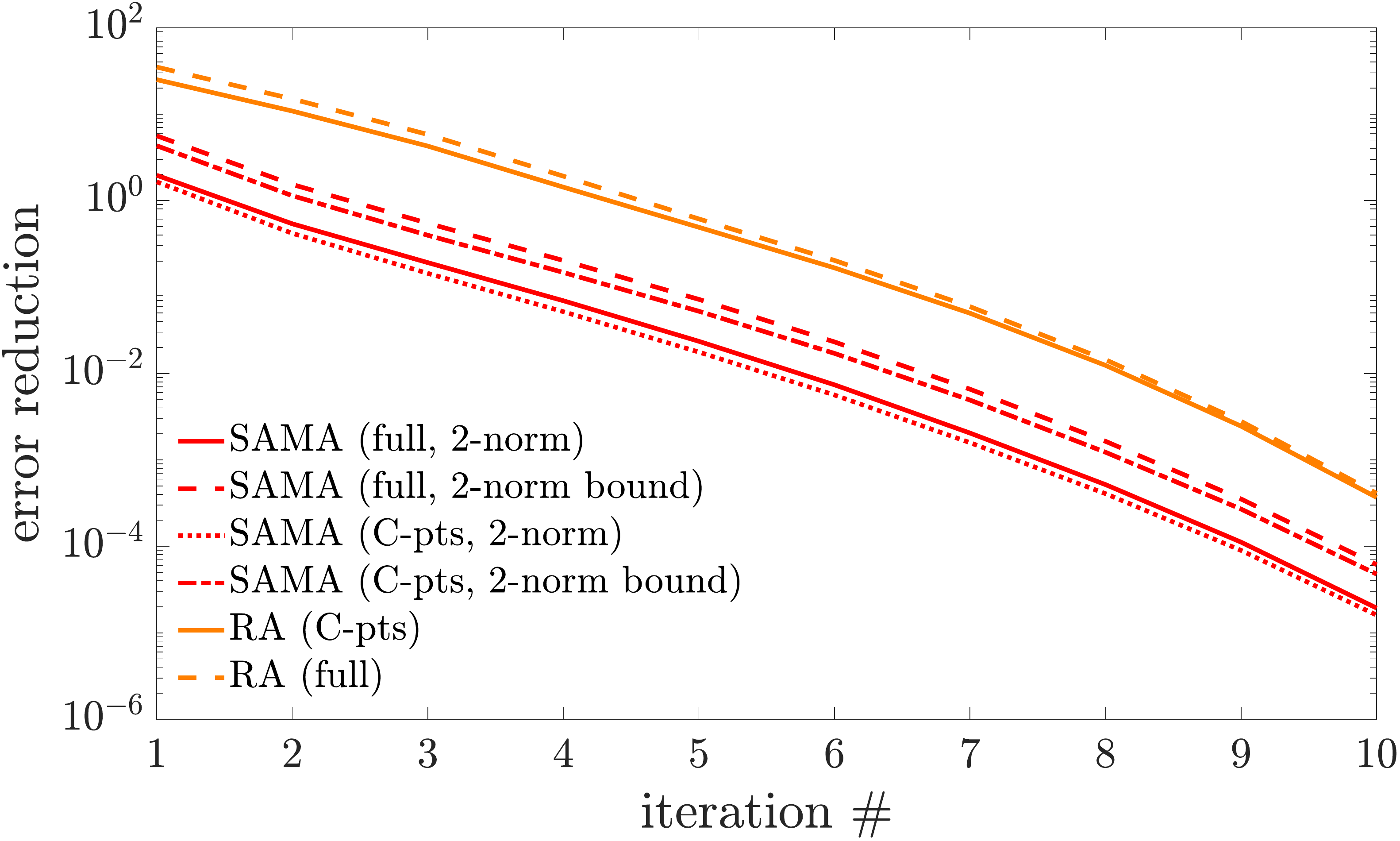}}
	\caption{Linear elasticity: Several variants of SAMA- and
          RA-predicted error reduction factors for the first 10
          two-level iterations with F- (left) and FCF-relaxation
          (right) and factor-2 temporal coarsening.\label{fig:sama_ra_elasticity}}
\end{figure}


\subsubsection{Predictivity of space-time LFA}
Motivated by the observation that long time intervals are needed for
LFA to be predictive for parabolic problems, in this section, we
investigate whether similar observations apply to hyperbolic problems
by comparing space-time LFA predictions with SAMA predictions for the
linear advection problem on time intervals of increasing lengths. More
precisely, the linear advection problem is discretized on a space-time
mesh of size $64\times N_t$ using various numbers of time steps,
$N_t$, with a spatial mesh size of ${\Delta x}=1/2$ and with fixed
time-step size $\Delta t=1/10$; again, a flow speed of $c=1$ and
factor-2 temporal coarsening are considered. We compute the error
reduction factors $\sigma_\text{LFA}$ and $\sigma_\text{SAMA}$ for the
first 20 iterations of the two-level methods. The slopes of the
best-fit lines of these error reduction factors as a function of the
iterations can be used to compute the predicted average error
reductions per iteration, $\overline{\sigma}_\text{LFA}$ and
$\overline{\sigma}_\text{SAMA}$, respectively. Figure
\ref{fig:stlfa_sama_advection} shows these predicted average error
reduction factors (left) and the difference,
$|\overline{\sigma}_\text{LFA}-\overline{\sigma}_\text{SAMA}|$, of
LFA- and SAMA-predicted average error reductions (right) as functions
of increasing numbers of time steps, $N_t$. Results show that the
difference between LFA and SAMA predictions decreases with increasing
numbers of time steps. More precisely, while for $N_t=128$,
LFA-predicted average error reduction factors using F- and
FCF-relaxation are about 23\% or 90\%, respectively, larger than SAMA
predictions, for $N_t=1024$, LFA predictions are only about 2\% or
6\%, respectively, larger. Thus, results suggest that space-time LFA
is a feasible option for predicting the convergence behavior for large
numbers of time steps, usually corresponding to long time intervals as
$T = N_t\Delta t$. Note that for hyperbolic problems, considering long
time intervals may be interesting in practice, in contrast to the case of parabolic problems for which LFA becomes predictive for time intervals that generally are longer than the diffusion time scale. 

\begin{figure}[ht]
	\centerline{\includegraphics[width=.48\textwidth]{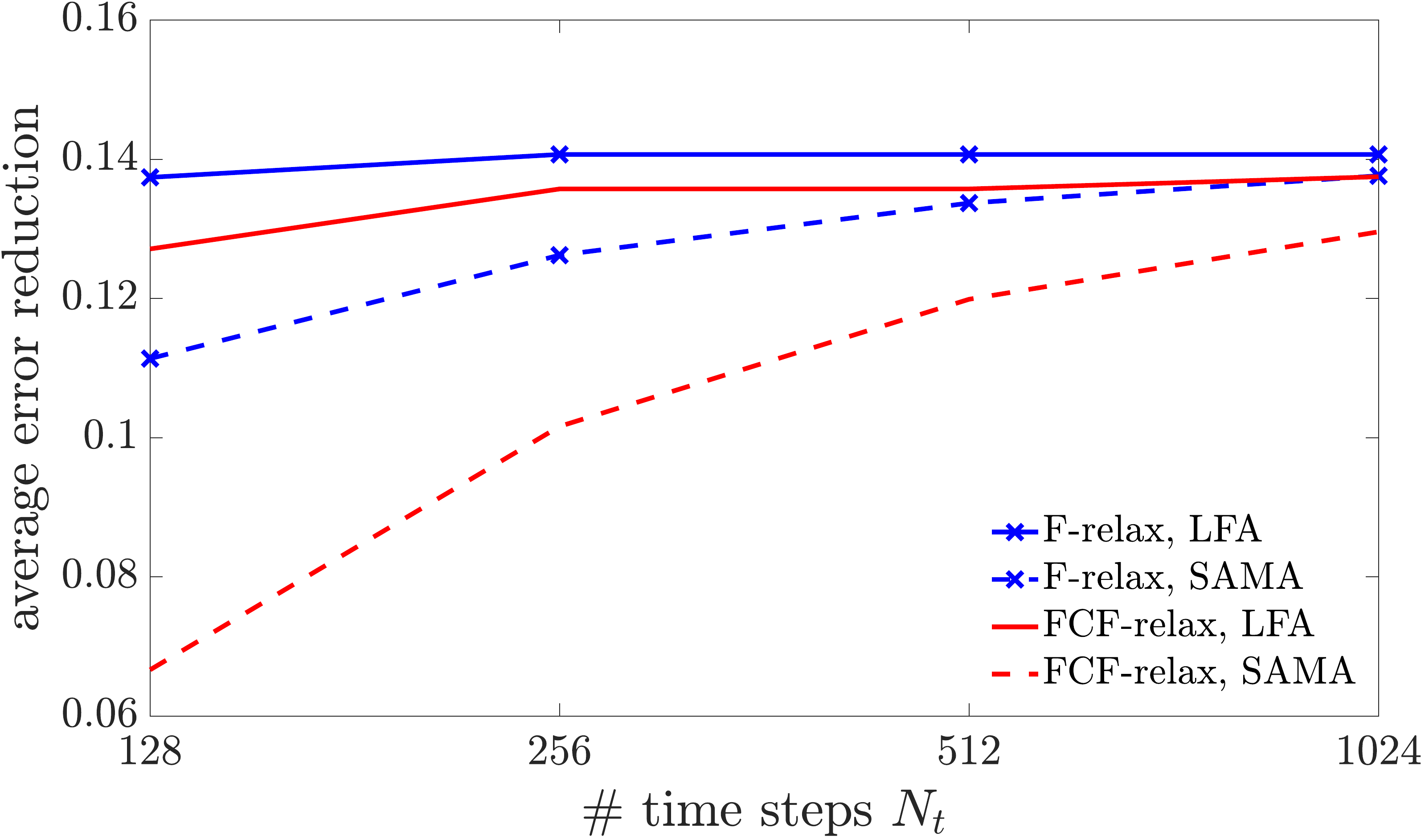}\quad
	\includegraphics[width=.48\textwidth]{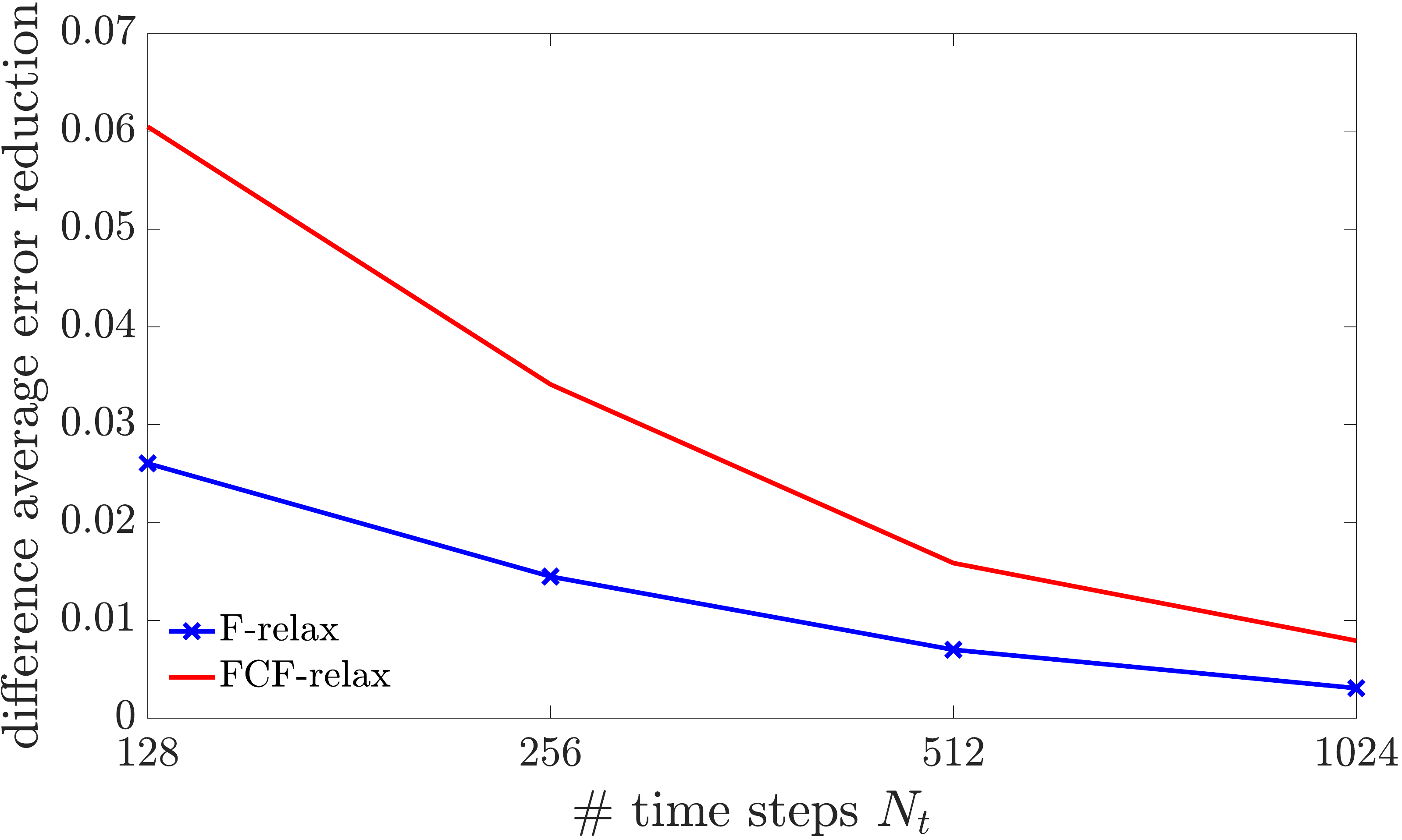}}
	\caption{Difference between SAMA and space-time LFA for linear advection. Shown are average error reduction factors measured over the first 20 iterations (left) and the difference between LFA- and SAMA-predicted average error reduction factors (right) for increasing numbers of time steps.
    \label{fig:stlfa_sama_advection}}
\end{figure}

Interestingly, in results not shown here, we see that space-time LFA
is much more predictive of behaviour over the first few MGRIT
iterations than it is over later iterations.  As a concrete example, in Figure
\ref{fig:stlfa_sama_ra}, we see that the space-time LFA and SAMA
predictions agree nearly perfectly for the first iteration, but
diverge after this.  When averaging over the first 10 iterations, in
place of the first 20 used in Figure \ref{fig:stlfa_sama_advection},
we see improvements of about a factor of two in the differences in
average error reduction predicted compared to the results shown in
Figure \ref{fig:stlfa_sama_advection}.  These results are offset by
correspondingly higher errors in iterations 11-20, where SAMA
accurately predicts improvements in convergence factors that are
missed by LFA.


\subsection{Investigating MGRIT convergence for linear advection}\label{sub:investigating_adv}
So far, we have only presented results for flow speed $c=1$ and
two-level cycling with factor-2 temporal coarsening. In this section,
we investigate the effects of other wave speeds, of multilevel
cycling, and of other coarsening factors on convergence. In
particular, we are interested in answering the question of which
frequencies cause slow convergence when MGRIT performance degrades.

\subsubsection{Influence of model parameters}\label{sub:adv_influence_parameters}
The discrete advection problem \eqref{eq:advection:discrete} depends
on the factor $\lambda := (c\Delta t)/{\Delta x}$, which can be seen
as an effective CFL number, relating the flow speed $c$ and the
discretization parameters ${\Delta x}$ and $\Delta t$. To determine
the effect of $\lambda$ on convergence, we consider predicted error
reduction for varying flow speeds $c$, while keeping the
discretization parameters ${\Delta x}$ and $\Delta t$ fixed. As SAMA provides 
the most accurate predictions among the three analysis tools and allows insights into 
error reduction as a function of spatial modes, in this section, we only consider SAMA. 
Figure \ref{fig:vary_wave_speed_advection} shows the error reduction factor
$\sigma_\text{SAMA}$ for the first 20 two-level iterations with F- and
FCF-relaxation and factor-2 temporal coarsening applied to linear
advection with flow speeds $c=2$ and \revision{$c=.5$} (corresponding to $\lambda
=.4$ and \revision{$\lambda=.1$}), respectively, discretized on a space-time mesh
of size $64\times 512$ with spatial mesh size ${\Delta x}=1/2$ and
time-step size $\Delta t=1/10$. The right plot in Figure
\ref{fig:vary_wave_speed_advection} details the error reduction for
the first iteration, showing the SAMA-predicted error reduction as a
function of the spatial Fourier modes. Convergence with both
relaxation strategies is similar, with convergence degrading with increasing effective CFL. 
\revision{We note that the error reduction is slowest for a set of low frequency modes that are close to, but do not include, the constant mode at $\theta = 0$, indicating that the convergence rate is limited by the coarse-grid correction process.}

\begin{figure}[h!t]
	\centerline{\includegraphics[width=.48\textwidth]{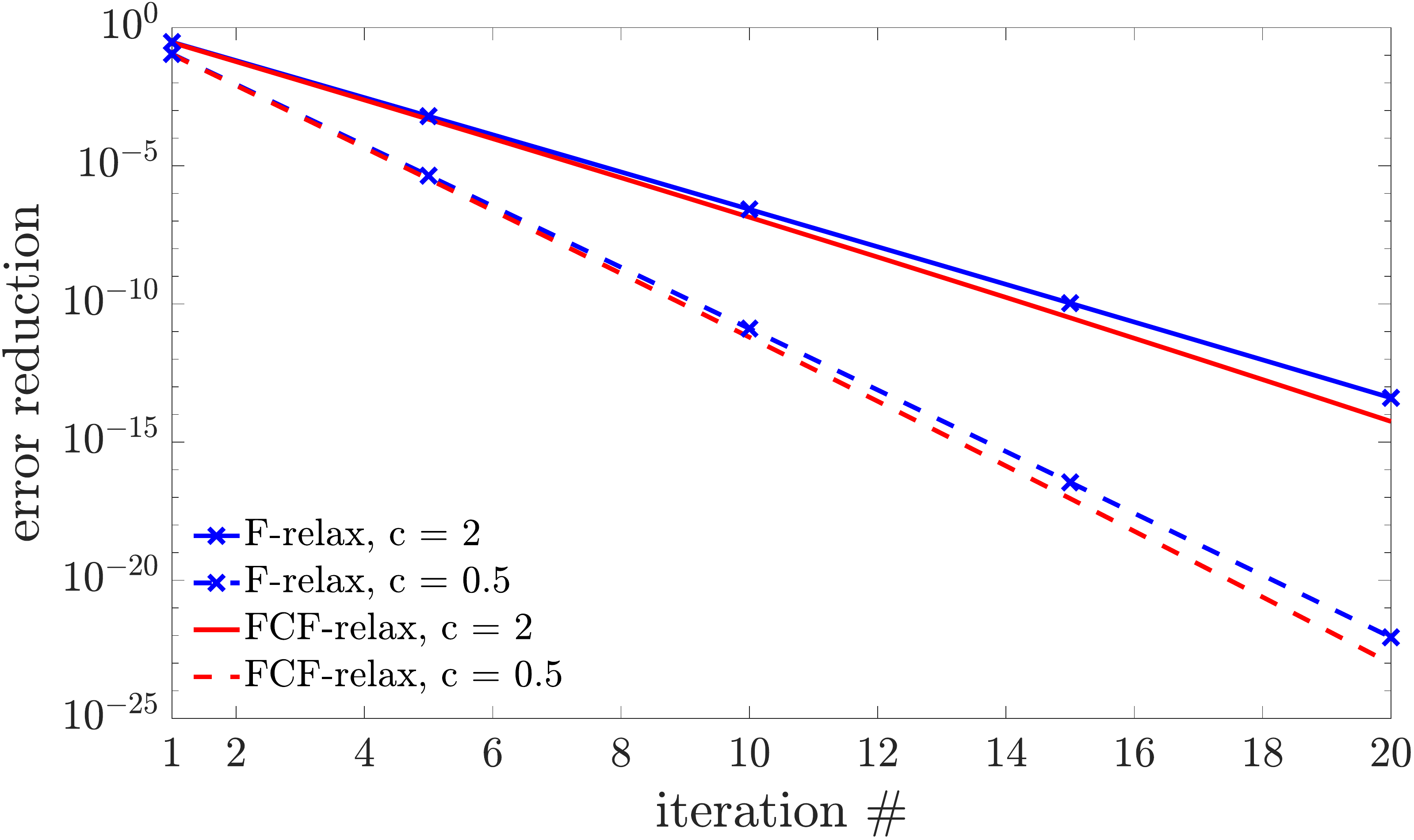}\quad
	\includegraphics[width=.48\textwidth]{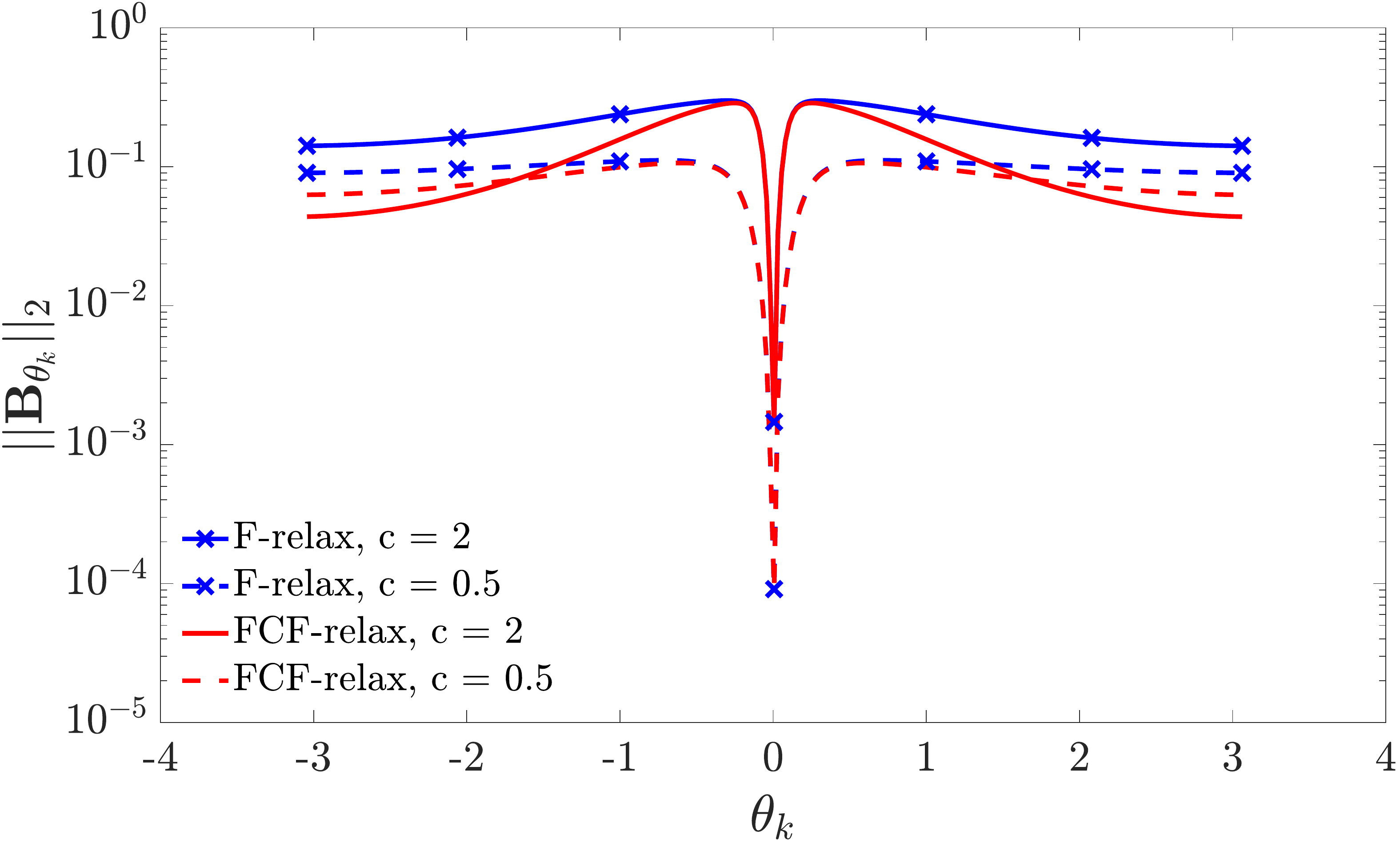}}
	\caption{Error reduction for varying wave speeds. At left, worst-case error reduction and at right, error reduction per spatial Fourier mode for the first iteration.\label{fig:vary_wave_speed_advection}}
\end{figure}


\subsubsection{Effects of multilevel cycling and coarsening}\label{sub:adv_influence_cycling}
In Parareal and MGRIT algorithms, computations on the coarsest
temporal grid are sequential. Larger temporal coarsening factors
reduce the size of the coarsest-grid problem and, thus, the cost of
computations. On the other hand, large coarsening factors increase the
cost of relaxation. To determine the effect of different coarsening
strategies on convergence, we consider two- and three-level MGRIT
variants applied to the advection problem discretized on a $64\times
256$ space-time grid with spatial mesh size ${\Delta x} = 1/2$ and
time-step size $\Delta t = 1/10$. The left plot in Figure
\ref{fig:2lvl_vs_3lvl_advection} shows SAMA-predicted error reduction
factors for two-level iterations with F- and FCF-relaxation with
factor-2 and with factor-4 temporal coarsening. Increasing the
coarsening factor leads to slower convergence, especially for
F-relaxation. When considering the predicted average error reduction,
$\overline{\sigma}_\text{SAMA}$, over the first 10 iterations, average
error reduction per iteration degrades from 0.13 to 0.31 for
F-relaxation and from 0.11 to 0.24 for FCF-relaxation. Note that
considering a three-level method with coarsening factors $m$ and $m_2$
leads to the same number of time points on the coarsest grid as a
two-level method with a coarsening factor of $mm_2$. Therefore, we
compare convergence of two-level iterations with F- and FCF-relaxation
with factor-4 coarsening with convergence of three-level schemes using
factor-2 coarsening in both coarsening steps. The right plot in Figure
\ref{fig:2lvl_vs_3lvl_advection} shows error reduction factors for the
two-level methods and for three-level V- and F-cycle methods. Neglecting cycle 
costs, V-cycles lead to slower convergence than the corresponding two-level 
schemes, while F-cycles converge faster than the two-level schemes. However, 
taking into account that three-level F-cycles are twice as expensive as three-level 
V-cycles, the convergence rates of three-level V-cycles are better than those of 
three-level F-cycles, with an average error reduction over the first 10 iterations 
of three-level F-cycles with F- or FCF-relaxation of about 0.15 and 0.12, respectively, 
compared to squared average error reduction rates of three-level V-cycles with 
F- or FCF-relaxation of about 0.11 and 0.08, respectively. Note that the two-level 
results in Figure \ref{fig:2lvl_vs_3lvl_advection} with $m=4$ and FCF-relaxation come
close to achieving convergence due to the exactness property, since
the coarse-grid problem has only 64 temporal meshpoints, so we expect
exact convergence in 32 two-level iterations.  This may be the
underlying reason for the prominent difference in convergence between
F- and FCF-relaxation over the later iterations.

\begin{figure}[h!t]
	\centerline{\includegraphics[width=.48\textwidth]{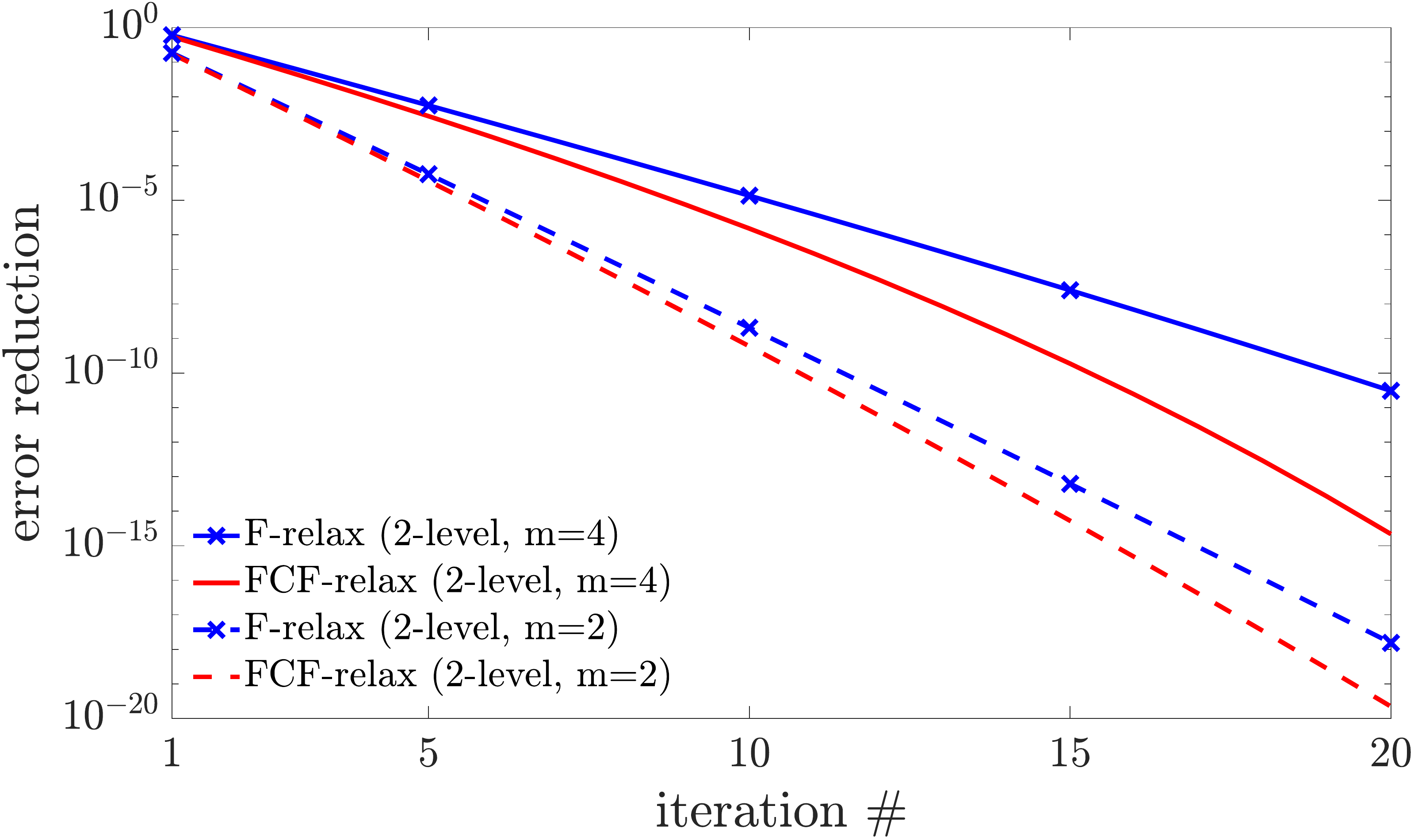}\quad
	\includegraphics[width=.48\textwidth]{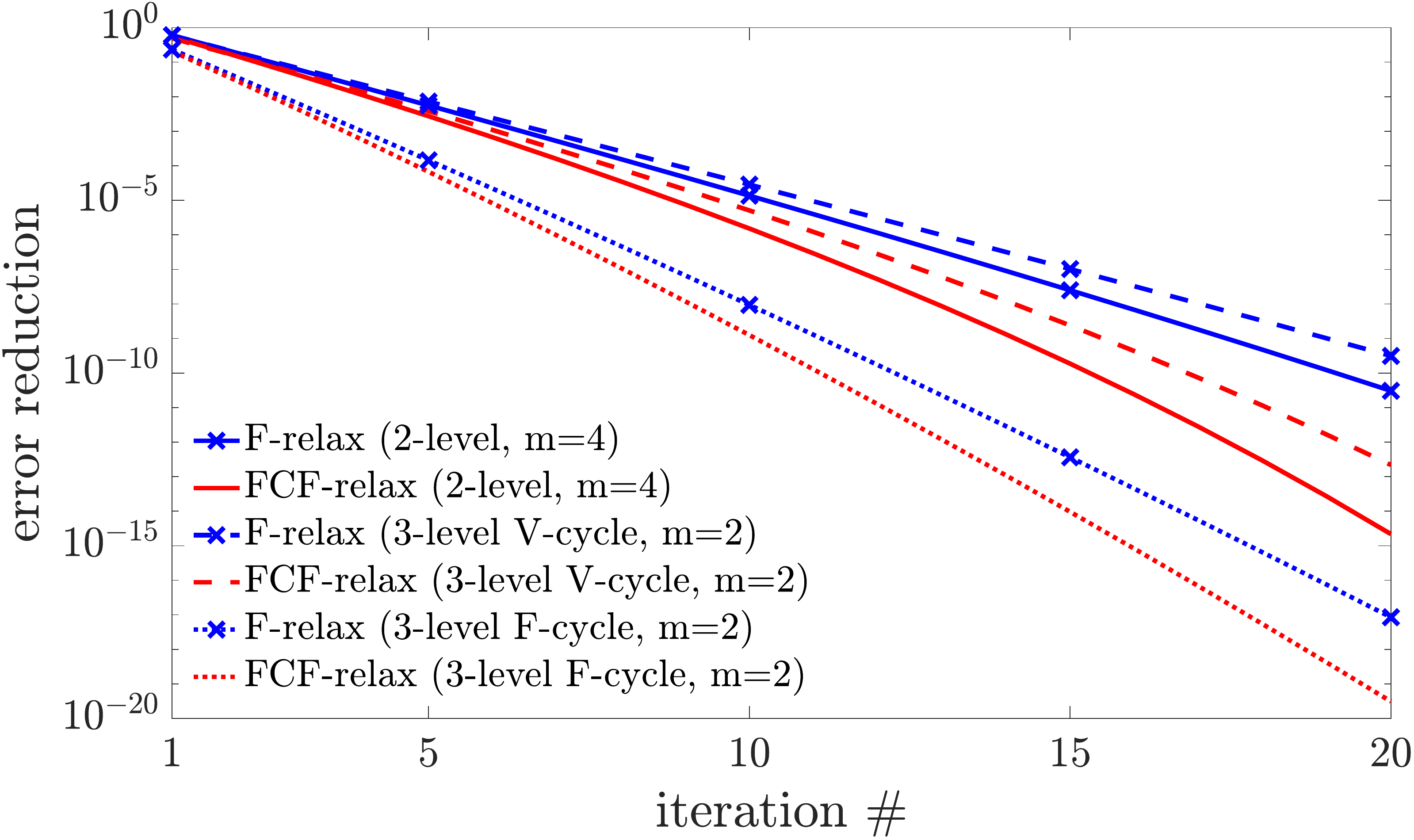}}
	\caption{Error reduction for two- and three-level variants
          applied to advection. At left, two-level methods with factor-2 and factor-4 temporal coarsening and at right, two- and three-level variants with same coarsest grid size.\label{fig:2lvl_vs_3lvl_advection}}
\end{figure}


\subsection{Investigating MGRIT convergence for linear elasticity}\label{sub:investigating_elasticity}
Performing similar experiments for the linear elasticity problem as for the linear advection problem in the previous section is computationally challenging. Additionally to the spatial Fourier symbol being a block matrix as a result of the system structure of the problem and of the mixed finite-element discretization, we consider Fourier frequency pairs on a two-dimensional tensor-product mesh instead of on a one-dimensional mesh. Results in Section \ref{sub:comparison} show that LFA may not be a feasible option, and predictions of RA are pessimistic due to the necessary condition number factor in the computations. SAMA gives a practical improvement over RA, which is also true for SAMA when computing bounds instead of exact Euclidean operator norms. We therefore focus on computing the SAMA bounds given in Equation \eqref{eq:err_red_factor_sama_bound} when analyzing MGRIT convergence.

One remaining drawback of the modified SAMA analysis is that it still
becomes expensive for large numbers of time steps. When considering
larger coarsening factors, however, we need to increase the number of
time steps due to the exactness property of MGRIT. To make the SAMA analysis computationally tractable for larger numbers of time steps, we sample spatial Fourier frequencies on a discrete mesh with spacing $h_\theta=\pi/16$. Figure \ref{fig:samaBound_vary_htheta_elasticity} shows that with this mesh spacing, SAMA captures the \revision{dominant behavior} and shows details in the spectral Fourier domain that are not visible with a coarser mesh spacing of $h_\theta=\pi/8$, and are not significantly different when considering a finer mesh spacing of $h_\theta=\pi/32$.

\begin{figure}[h!t]
	\centerline{\includegraphics[width=.385\textwidth]{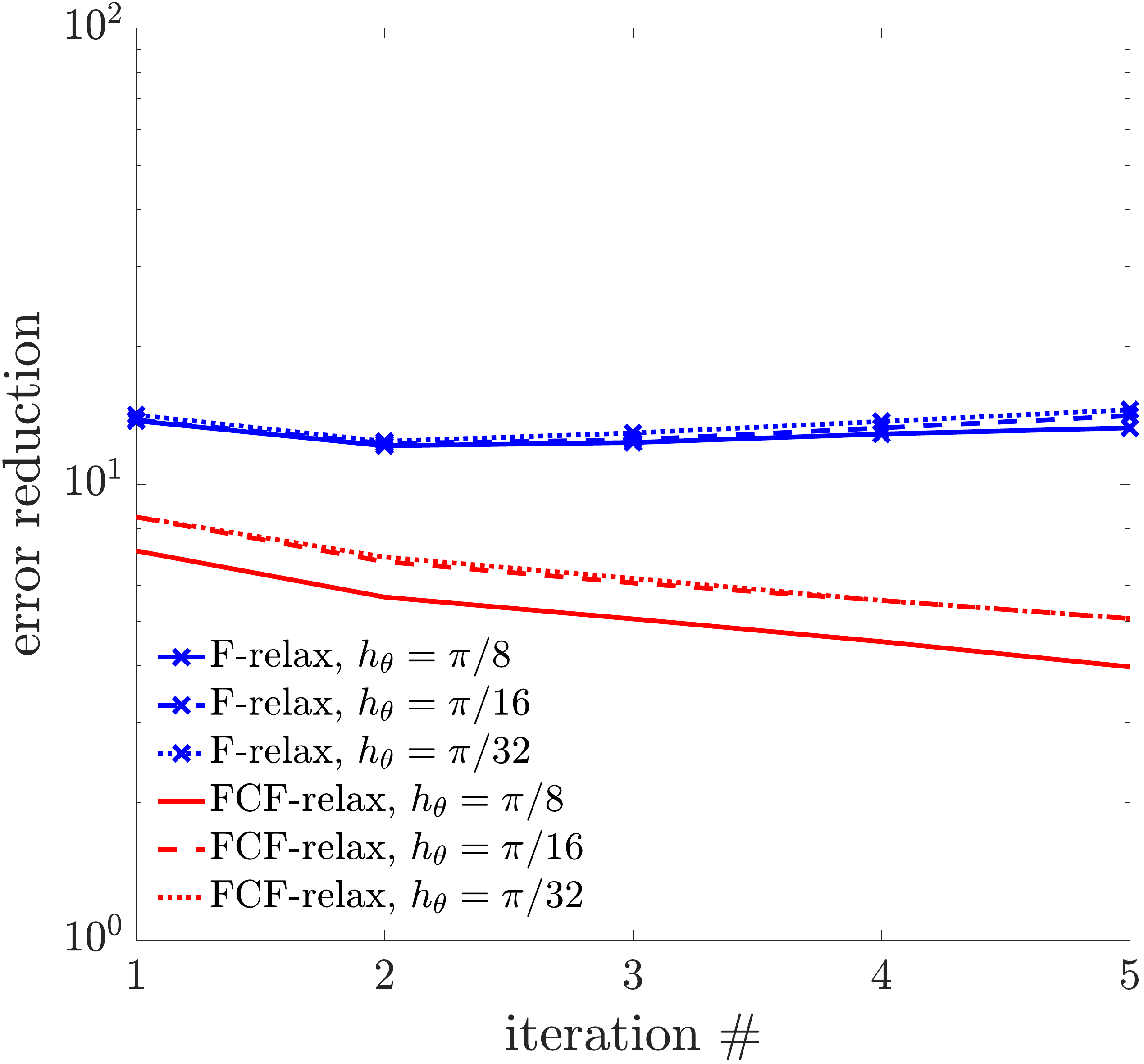}\qquad\includegraphics[width=.5\textwidth]{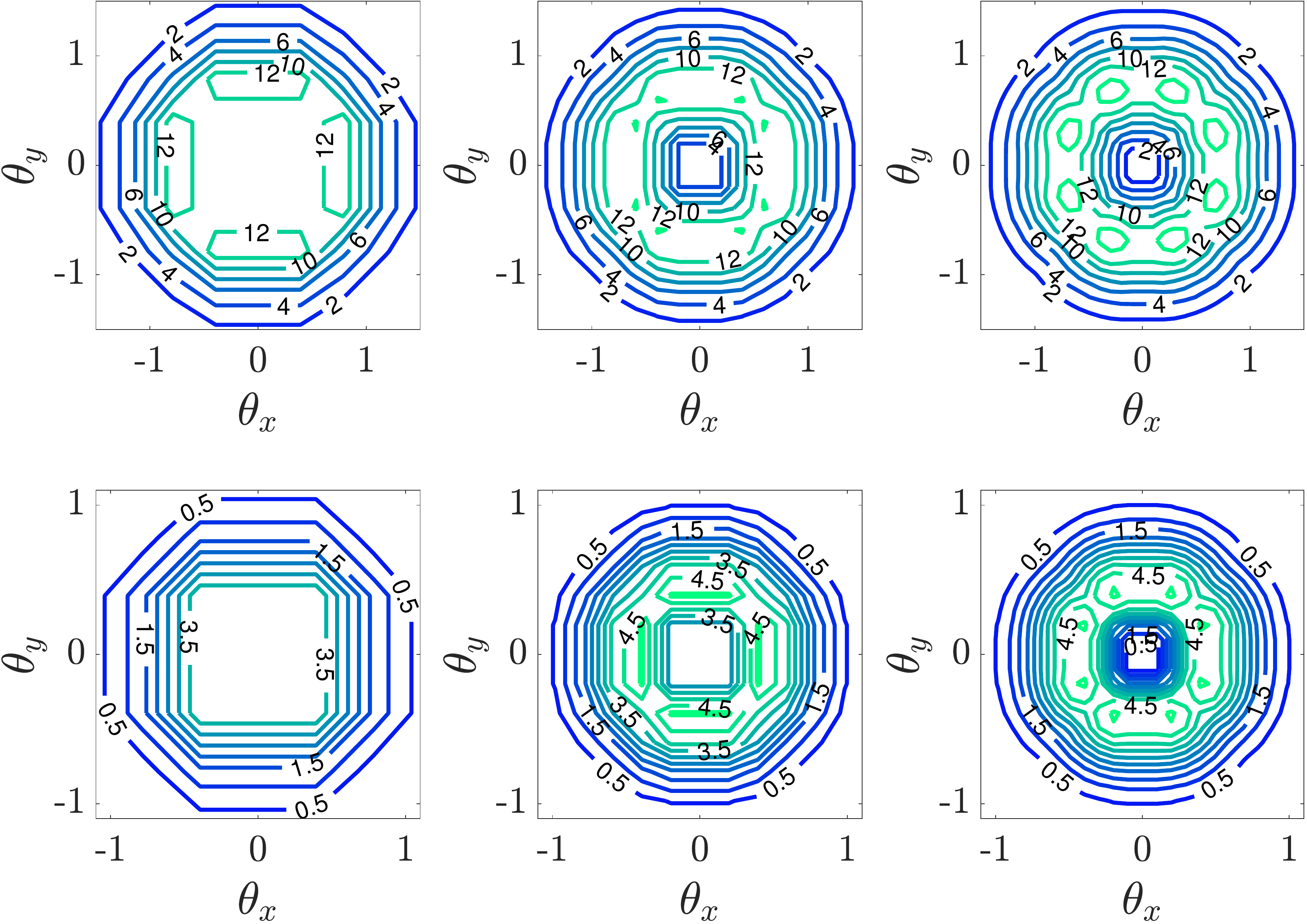}}
	\caption{Error reduction factors predicted by SAMA using
          various mesh spacings, $h_\theta$, for sampling the spatial
          Fourier frequencies for two-level MGRIT with factor-8
          temporal coarsening applied to the linear elasticity problem
          discretized on a $32^2\times 256$ space-time mesh. At left,
          worst-case error reduction and at right, error reduction per
          spatial Fourier mode for the fifth iteration.  The top row
          of figures at right correspond to results with
          F-relaxation, while the bottom row shows results with
          FCF-relaxation; the columns at right correspond to
          $h_\theta = \pi/8$ (left), $h_\theta = \pi/16$ (middle), and
          $h_\theta = \pi/32$ (right).\label{fig:samaBound_vary_htheta_elasticity}}
\end{figure}

In Figure \ref{fig:samaBound_vary_m_elasticity}, we look at the
effects of the coarsening factor on error reduction. The left-hand
side of the figure plots error reduction factors for the first 16
iterations of two-level MGRIT with factor-2 and with factor-8 temporal
coarsening. While using a coarsening factor of two leads to good
convergence behavior of both methods, we see divergence for factor-8
temporal coarsening in the first iterations. Note that the exactness
property drives convergence of the iteration with FCF-relaxation and
factor-8 coarsening at later iterations. At the right of Figure
\ref{fig:samaBound_vary_m_elasticity}, we plot error reduction factors
of the tenth iteration for both schemes (top and bottom rows) and both
coarsening strategies (left and right columns) as functions of the
Fourier frequency pair $(\theta_x,\theta_y)$.  Note that the overall
``structure'' of these plots is quite similar across the different
algorithmic parameters (although they are plotted on slightly
different axes to best show individual details), with excellent
convergence at frequencies both close to and far from the origin in
frequency space, and worst convergence achieved for a range of small
(but not zero) frequencies.  This is quite possibly related to 
phase errors between the fine- and coarse-scale
propagators\cite{Ruprecht2018}; to what extent this mode analysis tool 
may provide insight into how to cure the poor convergence is a question
left for future work.

\begin{figure}[h!t]
	\centerline{\includegraphics[width=.5\textwidth]{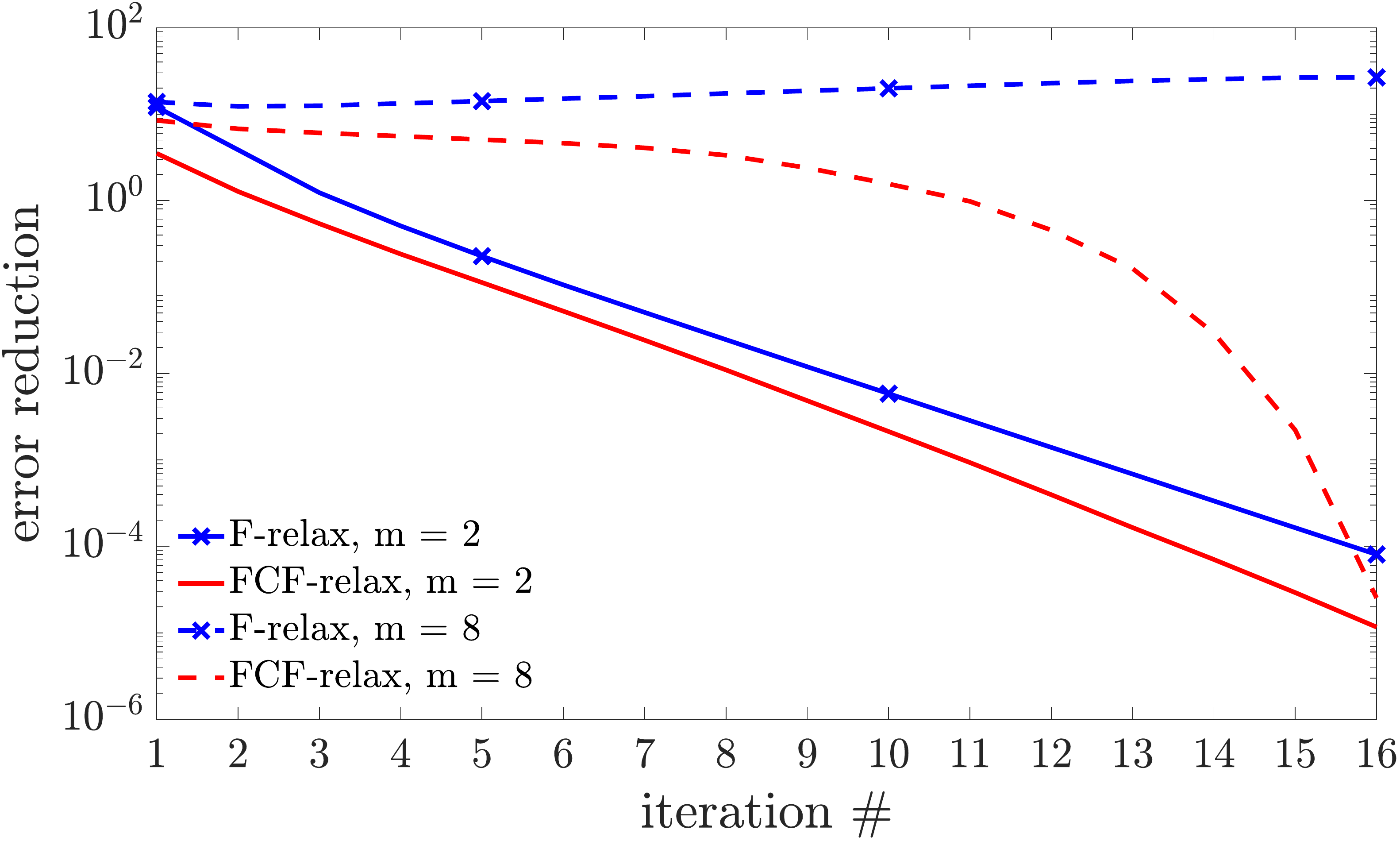}\quad\includegraphics[width=.4\textwidth]{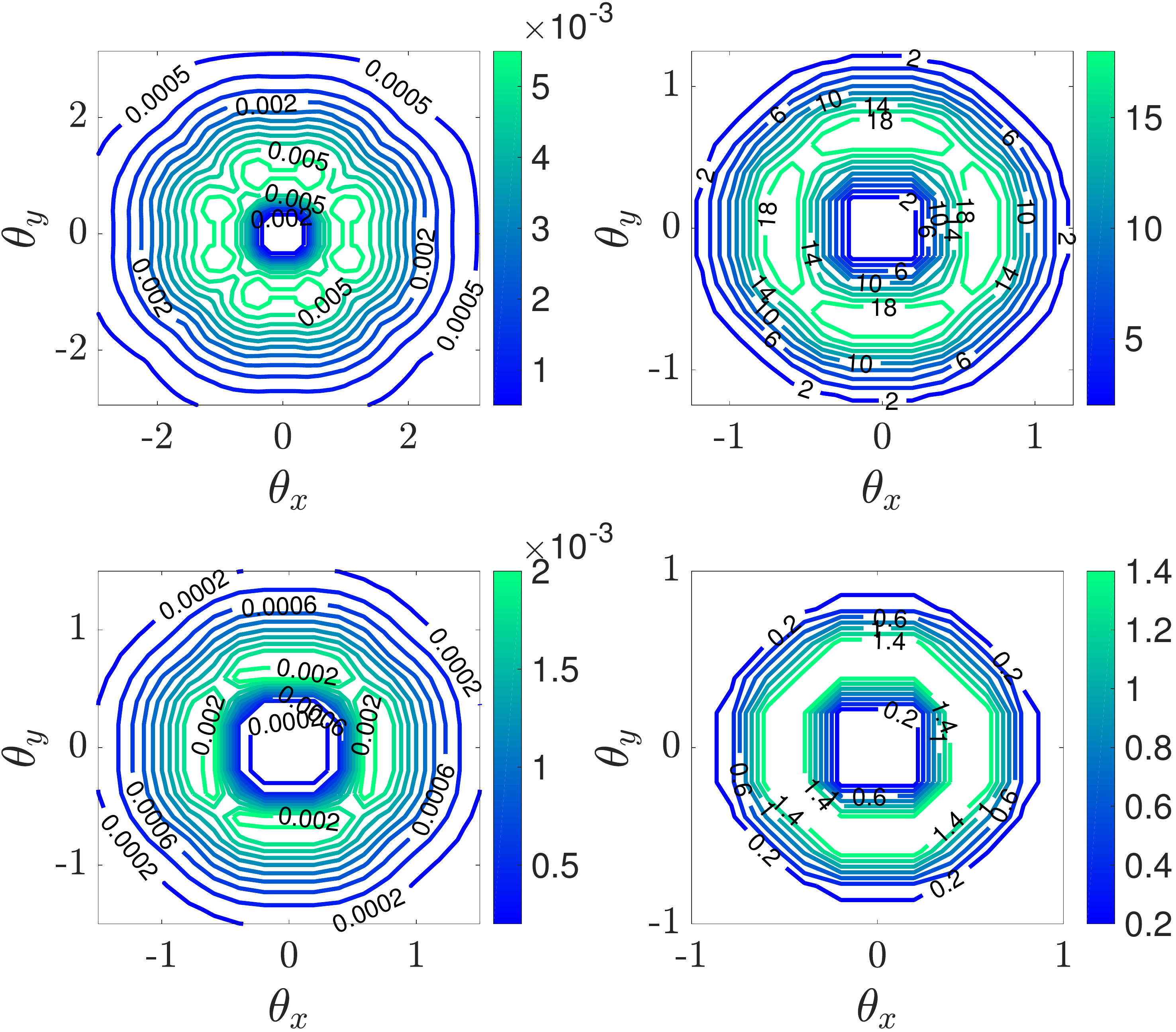}}
	\caption{Error reduction factors predicted by SAMA for
          two-level MGRIT with factor-2 and factor-8 temporal
          coarsening applied to the linear elasticity problem
          discretized on a $32^2\times 256$ space-time mesh. At left,
          worst-case error reduction and at right, error reduction for
          the tenth iteration with F-relaxation and factor-2 (top
          left) and factor-8 coarsening (top right) and with
          FCF-relaxation and factor-2 (bottom left) and factor-8 (bottom right) coarsening as functions of the Fourier frequency, $\boldsymbol{\theta}$.\label{fig:samaBound_vary_m_elasticity}}
\end{figure}

The discrete elasticity problem, given in Equations
\eqref{eq:vt:discrete}-\eqref{eq:incompr:discrete}, depends on the
material parameters $\rho$ and $\mu$ as well as on the discretization
parameters ${\Delta x}$ and $\Delta t$. To determine relevant
parameter sets for convergence studies, we make the following
observations: In \eqref{eq:incompr:discrete}, the scaling of ${\Delta
  x}$ in the discrete divergence operator and of $\Delta t$ clearly
does not matter at all. In \eqref{eq:ut:v:discrete}, the value of
$\Delta t$ matters as an absolute.  For standard finite-element
discretizations on uniform meshes in two spatial dimensions (as we
consider here), the mass matrix can be written as a scaling factor of $(\Delta
x)^2$ times an operator that is independent of $\Delta x$, while the
entries in the stiffness matrix are independent of $\Delta x$.  Thus,
we can rescale \eqref{eq:vt:discrete} by dividing through by $\rho
{\Delta x}^2$, and rescaling $p_i$ to obtain
\[
(1/{\Delta x}^2)Mv_i + (\Delta t^2\mu/\rho {\Delta x}^2) Kv_i +
(1/\Delta x)B\widetilde{p}_i =
(1/{\Delta x}^2)Mv_{i-1} - (\Delta t\mu/\rho {\Delta x}^2) Ku_{i-1}.
\]
Since the scaling on the two terms involving $K$ differs by a factor
of $\Delta t$, there are two natural parameters to consider:\linebreak$\nu :=
\Delta t\mu/\rho {\Delta x}^2$ and $\Delta t$.

To perform a thorough set of experiments with these parameters with a
reasonable computational complexity, we fix $N_t = 128$. The left plot
in Figure \ref{fig:sama_vary_T_elasticity} shows that the choice
$N_t=128$ is reasonable for experimenting with the parameters in the
system since this choice ensures that convergence is not affected by
the exactness property of MGRIT in the first few iterations.  Furthermore, 
aside from effects of the exactness property for small $N_t$, performance  
of the two-level methods does not depend on the number of time points. 
Note that varying the number of time steps changes the length of the time interval. 
This length can also be controlled by varying the time-step size (on a 
uniform-in-time mesh as we consider here). In the right plot of Figure 
\ref{fig:sama_vary_T_elasticity}, we consider three different time-step 
sizes $\Delta t$ for fixed $N_t = 128$ and $\rho = \mu = \nu = 1$. Results 
show that as $\Delta t$ decreases, there is an increasing initial jump in 
the worst-case error, but asymptotically, convergence appears to be 
independent of the time-step size. Furthermore, convergence of the 
two-level methods with $F$- and with $FCF$-relaxation is similar. 
After 10 iterations, $FCF$-relaxation gives about a factor of six improvement 
in error reduction over $F$-relaxation at twice the cost per itertation.  

\begin{figure}[h!t]
	\centerline{\includegraphics[width=.48\textwidth]{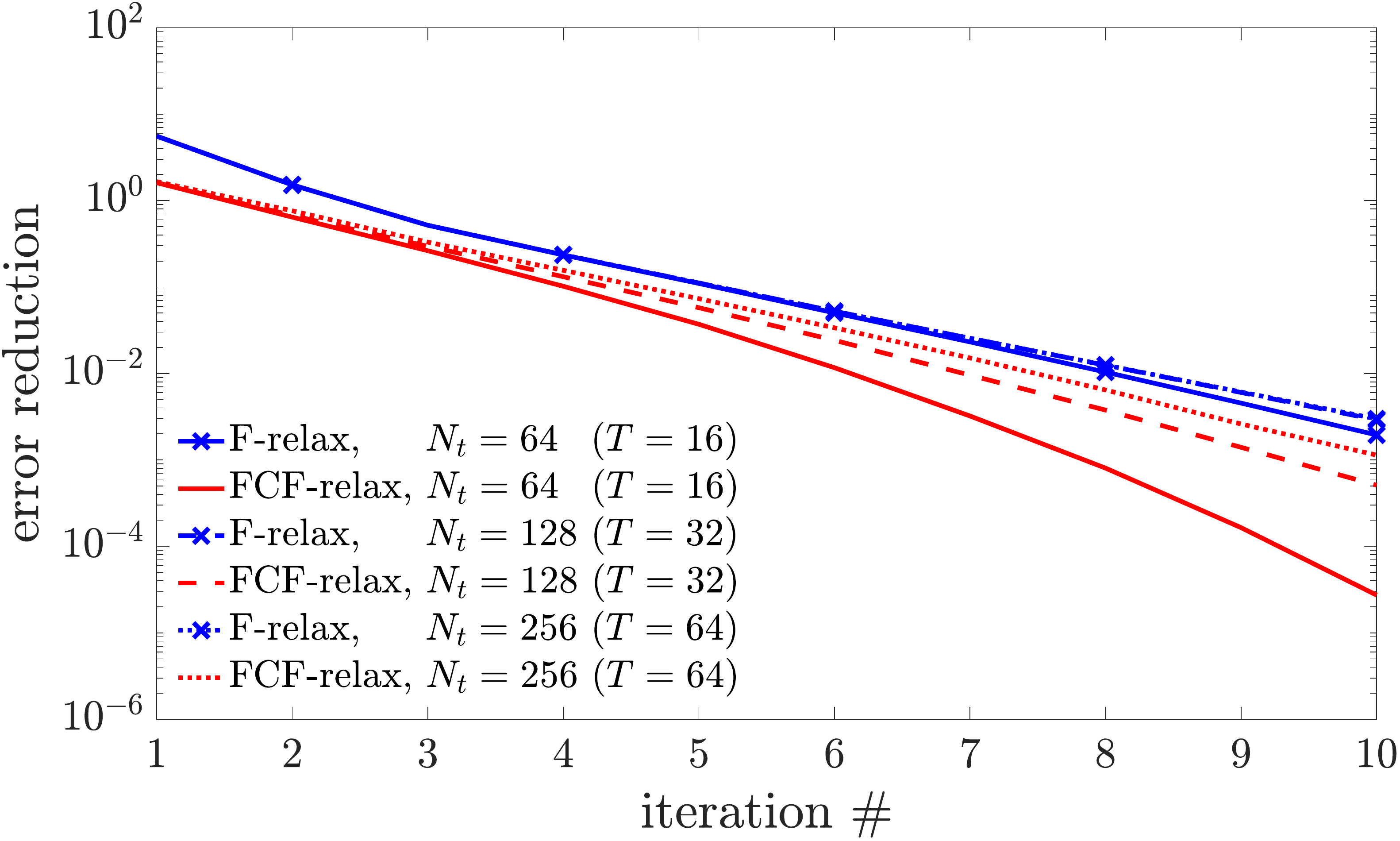}\quad\includegraphics[width=.48\textwidth]{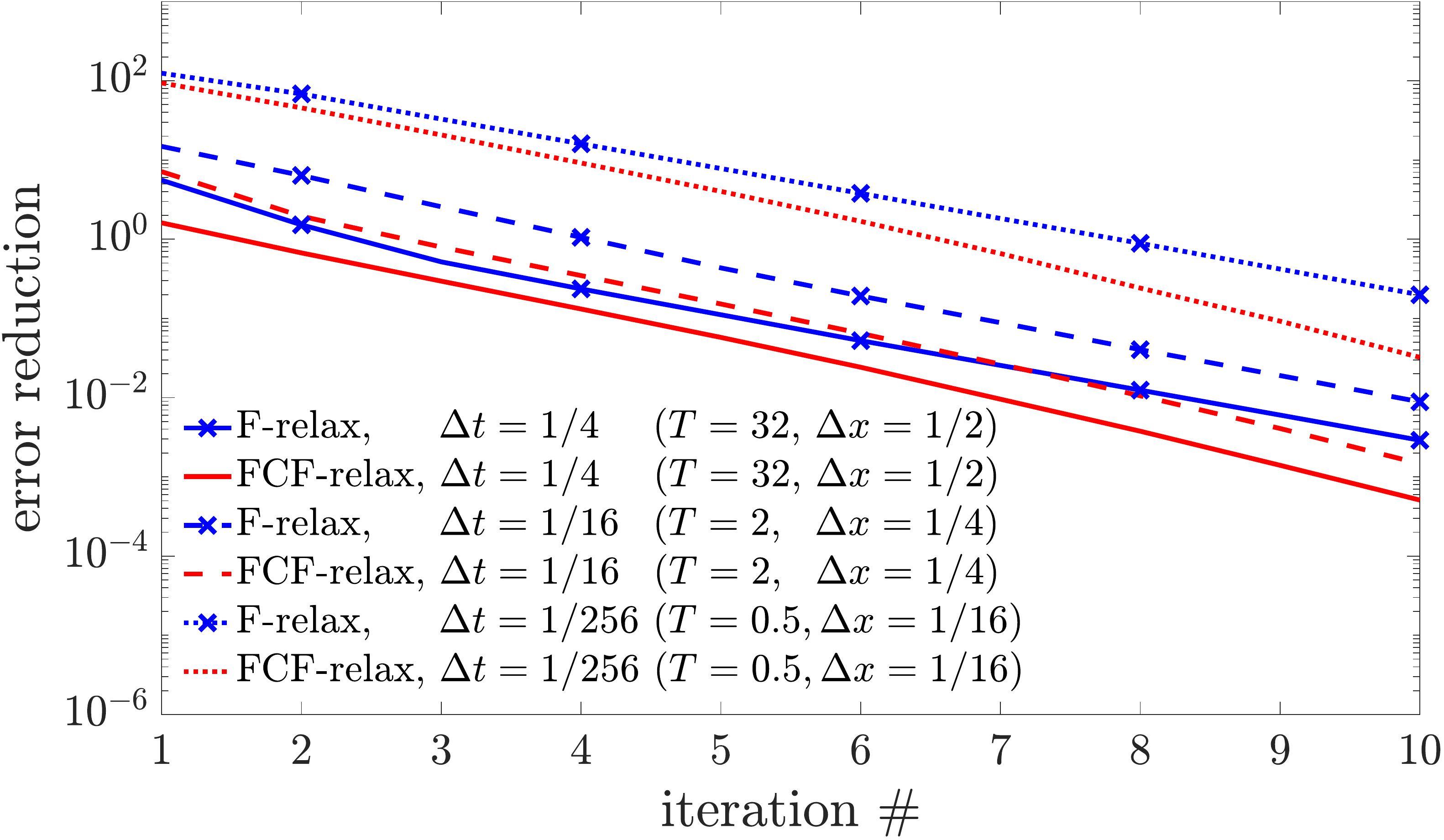}}
	\caption{Error reduction on space-time grids of size $32^2\times N_t$ for two-level methods with factor-2 temporal coarsening applied to the elasticity model problem on $(0,16)^2\times [0,T]$ discretized with fixed time step $\Delta t = 1/4$ (left) or with a fixed number of time steps $N_t=128$ (right). \label{fig:sama_vary_T_elasticity}}
\end{figure}

To gain insight into effects of model parameters on convergence, we consider the parameter, $\nu$, factored in the form $\nu = (\Delta t/{\Delta x}^2)(\mu/\rho)$, i.\,e., we group discretization and material parameters, and we study its effect on convergence of the two-level methods. More precisely, we look at convergence for fixed $\nu$ by simultaneously varying $\Delta t$ and ${\Delta x}^2$ or $\mu$ and $\rho$, respectively, and at error reduction for varying $\nu$ by fixing three of the four parameters and varying the remaining one. Figure \ref{fig:sama_vary_material_parameters_elasticity} shows effects of the material parameters on error reduction when the discretization parameters $\Delta t$ and $\Delta x$ are held constant. The left-hand side of the figure plots error reduction factors for fixed ratio $\mu/\rho = 1$, i.\,e., simultaneously varying $\mu$ and $\rho$ (corresponding to fixed $\nu$), while at the right, error reduction factors are shown for various ratios $\mu/\rho$ for fixed $\mu = 1$ (corresponding to varying $\nu$). The plots show that performance of both two-level schemes only depends on the ratio $\mu/\rho$. In the limit of large $\mu/\rho$, convergence is extremely fast, especially for the two-level scheme with $FCF$-relaxation, and unsteady for $F$-relaxation. This behavior is not surprising as this limit corresponds to a stiff material and, thus, the solution is approaching zero, i.\,e., oscillations are rapidly damped because of the Euler discretization.

\begin{figure}[h!t]
	\centerline{\includegraphics[width=.48\textwidth]{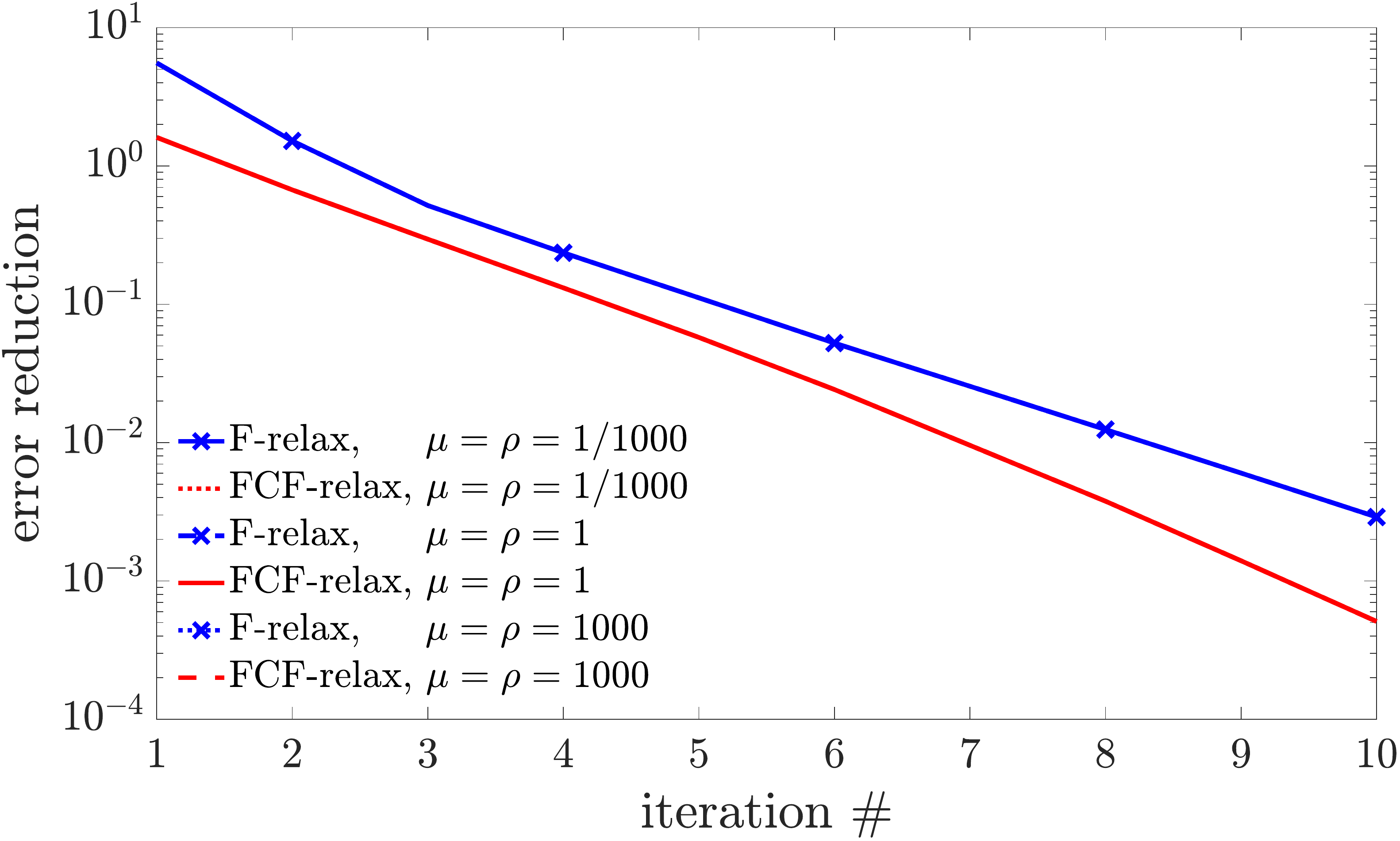}\quad\includegraphics[width=.48\textwidth]{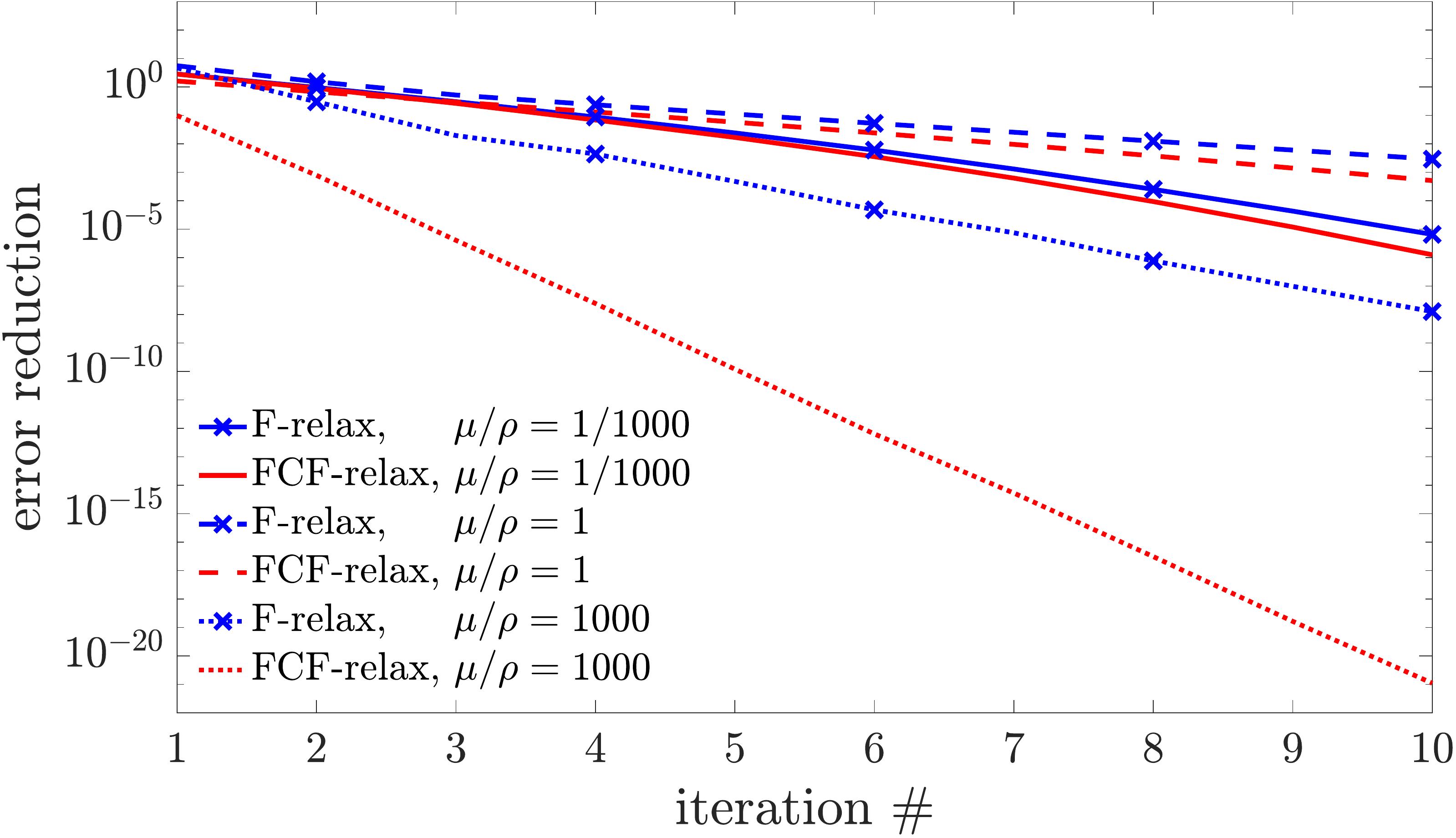}}
	\caption{Error reduction for two-level methods with factor-2 temporal coarsening applied to elasticity discretized on a $32^2\times 128$ space-time grid with mesh sizes $\Delta x = 1/2$ and $\Delta t = 1/4$ for various material parameters $\mu$ and $\rho$. At left, we simultaneously vary $\rho$ and $\mu$ (corresponding to fixed parameter $\nu = (\Delta t/{\Delta x}^2)(\mu/\rho) = 1$) and at right, we vary the material parameter $\rho$ and fix $\mu = 1$ (corresponding to varying $\nu$).\label{fig:sama_vary_material_parameters_elasticity}}
\end{figure}

The ratio $\mu/\rho$ allows us to determine effects of the parameter $\nu$ on error reduction for large variations in $\nu$. Using the slopes of the best-fit lines as a function of the iterations, we compute the average error reduction per iteration. The right-hand side of Figure \ref{fig:samaBound_avg_error_red_elasticity} shows the average error reduction over iterations two thru 10 as a function of the parameter $\nu$ for $\Delta t = 1/4$, $\Delta x = 1/2$, and $\mu = 1$ fixed, and various values of $\rho$. In all cases, average error reduction is bounded by 0.5, showing good and robust convergence in a reasonable parameter regime. At the left of Figure \ref{fig:samaBound_avg_error_red_elasticity}, we plot average error reduction as a function of the time-step size for fixed parameter $\nu = 16$. Results demonstrate that convergence does not change until we get to large time steps, again, indicating good and robust convergence.

\begin{figure}[h!t]
	\centerline{\includegraphics[width=.48\textwidth]{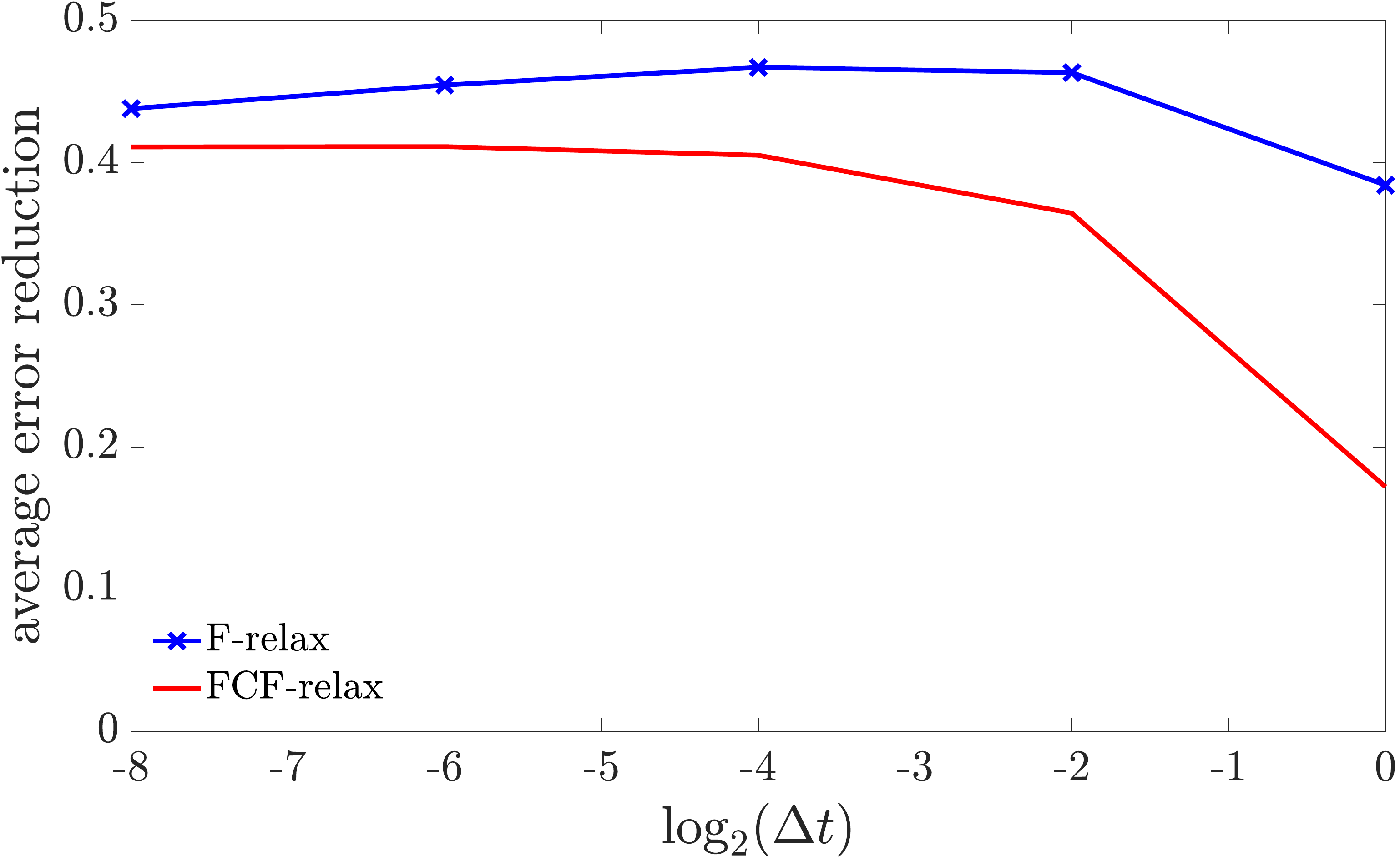}\quad
	\includegraphics[width=.48\textwidth]{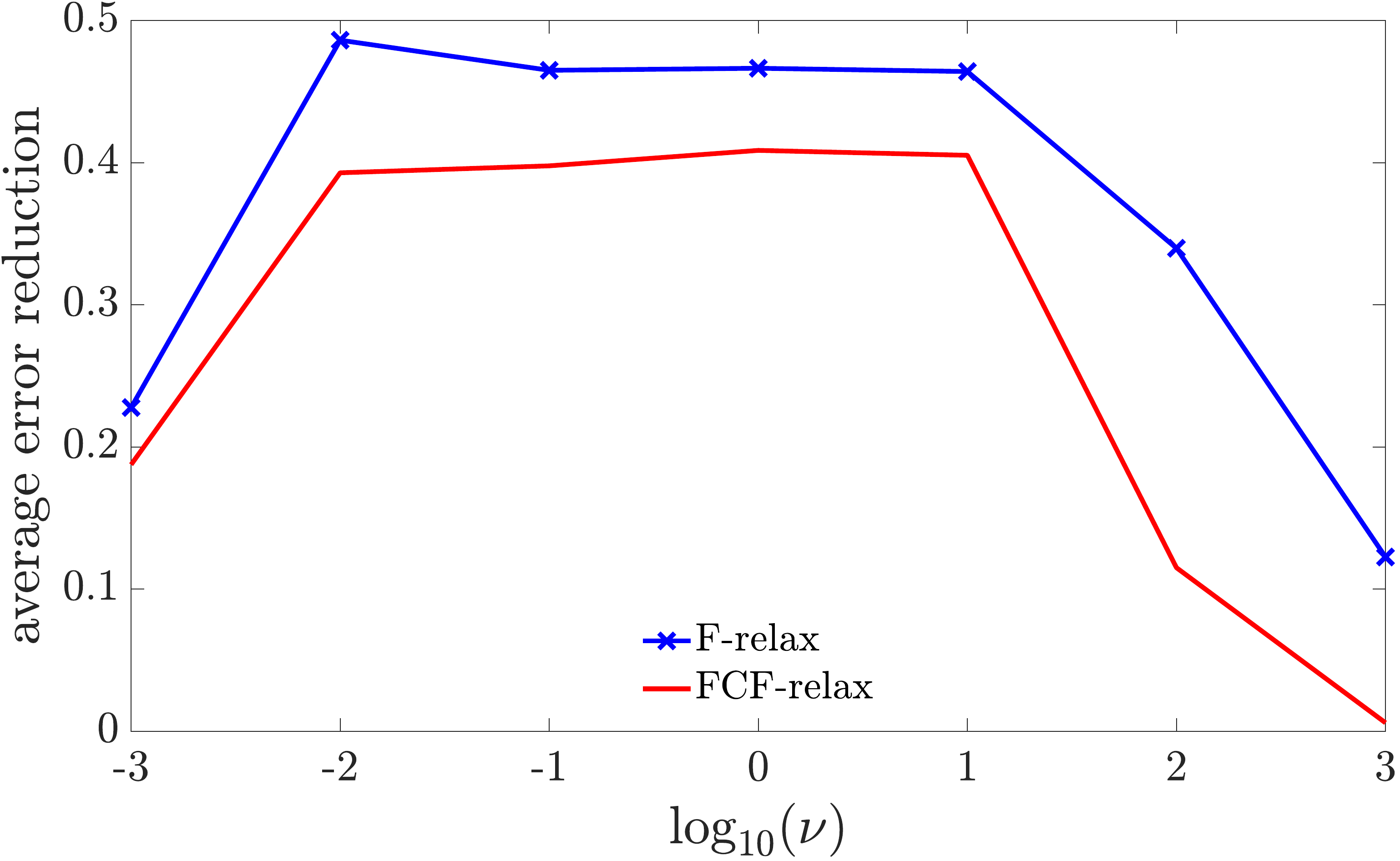}}
	\caption{Average error reduction over iterations two thru ten for two-level methods with factor-2 temporal coarsening applied to elasticity discretized on a $32^2\times 128$ space-time grid. At left, average error reduction as a function of $\Delta t$ for fixed $\nu = 16$ and at right, we vary the material parameter $\rho$ and fix $\Delta t = 1/4$, $\Delta x = 1/2$, $\mu = 1$, and show the average error reduction as a function of $\nu = (\Delta t/{\Delta x}^2)(\mu/\rho)$. \label{fig:samaBound_avg_error_red_elasticity}}
\end{figure}

\section{Conclusions}\label{sec:conclusions}

Currently, a key challenge in the development of parallel-in-time
algorithms is achieving scalable algorithmic performance for
hyperbolic PDEs.  While some insight has been gained by both
trial-and-error and more systematic computational studies in recent
years, predictive analytical tools to aid in this development have
been lacking in the literature.  Here, we examine and extend mode
analysis approaches, long used in the spatial multigrid community, to
examine performance of methods from the Parareal/MGRIT class for two
hyperbolic model problems.  Our extensions lead to tighter bounds on
performance that can be computed more efficiently than those existing
in the literature.  When applied to these model problems, we gain some
insight into what poses essential challenges in developing algorithms
for hyperbolic PDEs, and what parameter regimes, both physical and
computational, are more or less difficult to handle. \revision{As identified 
in other contexts\cite{Ruprecht2018, HDeSterck_etal_2019}, convergence 
of MGRIT appears to be limited by differences between the fine-grid and 
coarse-grid time propagators for a set of smooth spatial modes near zero 
frequency, but not at zero frequency.  While this is not a unique observation, 
we point out that the mode analysis realization of this statement has been 
key to the developments in De Sterck et al. \cite{HDeSterck_etal_2019}, 
who derive an optimization perspective on the construction of MGRIT 
coarse operators that is particularly effective for explicit and higher-order 
discretizations of the linear advection equation. This new approach
offers substantially more robust MGRIT performance for the linear 
advection equation than yet seen in the literature. Another key 
observation for the problems and discretizations considered in this paper} 
is that, while the FCF-relaxation that is typical in MGRIT
is generally more effective than the F-relaxation typical in Parareal,
the differences between these approaches in a two-level setting is
generally not substantial; when FCF-relaxation works best,
F-relaxation has good performance, too, and when F-relaxation is
ineffective, FCF-relaxation is not a magic cure (unless the exactness
property becomes significant).  \revision{Note that this conclusion is not 
universal. For linear advection, FCF-relaxation may be beneficial in the 
case of higher order Runge-Kutta discretizations, as discussed by Dobrev et al.
\cite{VADobrev_etal_2017} and De Sterck et al. \cite{HDeSterck_etal_2019}} 
While this work does
not immediately give direction as to how to improve algorithmic performance, 
it offers predictive tools that can be used to identify and diagnose convergence 
difficulties and may help in designing or optimizing improved algorithms.  
\revision{ Possible avenues for future research clearly include the application of 
these analysis tools to a broader class of PDEs and discretizations, as well as their 
extension to consider the case of convergence to the continuum solution, as 
considered for Parareal by Ruprecht\cite{Ruprecht2018}.}

%
%
\appendix

\section{Spatial Fourier Symbols for Elasticity Operator}
\label{sec:appendix:spatialLFA}
The Fourier symbols of the $Q2$ mass and stiffness matrices using
nodal basis functions can be computed using tensor products,
\[
	\widetilde{M}_x(\theta_1,\theta_2) = \widetilde{M}_y(\theta_1,\theta_2) = \widetilde{M}_\text{1D}(\theta_2)\otimes\widetilde{M}_\text{1D}(\theta_1)
\]
and
\[
	\widetilde{K}_x(\theta_1,\theta_2) = \widetilde{K}_y(\theta_1,\theta_2) = \widetilde{M}_\text{1D}(\theta_2)\otimes\widetilde{K}_\text{1D}(\theta_1) + \widetilde{K}_\text{1D}(\theta_2)\otimes\widetilde{M}_\text{1D}(\theta_1),
\]
respectively, with symbols
\[
	\widetilde{M}_\text{1D}(\theta) = \frac{{\Delta x}}{30}\begin{bmatrix}
		8 - 2\cos\theta & 4\cos\frac{\theta}{2}\\
		4\cos\frac{\theta}{2} & 16
	\end{bmatrix} \quad\text{and}\quad \widetilde{K}_\text{1D}(\theta) = \frac{1}{3{\Delta x}}\begin{bmatrix}
		14+2\cos\theta & -16\cos\frac{\theta}{2}\\
		-16\cos\frac{\theta}{2} & 16
	\end{bmatrix},
\]
of the 1D $Q2$ mass and stiffness matrices, respectively\cite{YHe_SPMacLachlan_2018a,YHe_SPMacLachlan_2018b}. The Fourier symbols of the derivative operators, $\widetilde{B}_x$ and $\widetilde{B}_y$, are given by\cite{YHe_SPMacLachlan_2018b}
\[
	\widetilde{B}_x(\theta_1,\theta_2)^T = \begin{bmatrix}
		-\frac{\imath {\Delta x}}{9}\sin\theta_1; & -\frac{4\imath {\Delta x}}{9}\sin\frac{\theta_1}{2}; & -\frac{2\imath {\Delta x}}{9}\sin\theta_1\cos\frac{\theta_2}{2}; & -\frac{8\imath {\Delta x}}{9}\sin\frac{\theta_1}{2}\cos\frac{\theta_2}{2}
	\end{bmatrix}
\]
and
\[
	\widetilde{B}_y(\theta_1,\theta_2)^T = \begin{bmatrix}
		-\frac{\imath {\Delta x}}{9}\sin\theta_2; & -\frac{2\imath {\Delta x}}{9}\sin\theta_2\cos\frac{\theta_1}{2}; & -\frac{4\imath {\Delta x}}{9}\sin\frac{\theta_2}{2}; & -\frac{8\imath {\Delta x}}{9}\sin\frac{\theta_2}{2}\cos\frac{\theta_1}{2}
	\end{bmatrix},
\]
respectively. The Fourier symbols of the time integrators, $\Phi_c$
and $\Phi_{cc}$, on the first and second coarse grid can be derived
analogously, simply by adjusting the value of $\Delta t$ in the
definition of $\widetilde{\Phi}$.

%
%

\section*{Acknowledgments}

The work of H.D.S and S.M. was partially supported by NSERC Discovery Grants.

\bibliography{pint_refs}%

\end{document}